\documentclass[12pt]{article}
\usepackage{amssymb,amsmath,amsthm,latexsym}
\usepackage[cp1251]{inputenc}
\usepackage{dsfont}
\usepackage{amssymb}
\usepackage{indentfirst}
\usepackage{graphicx}
\usepackage{caption}
\usepackage{multicol}
\usepackage{enumitem}
\usepackage{hyperref}
\usepackage{tikz-cd}
\hypersetup{citecolor=red,}   
\usepackage{amsfonts, array, hhline} 
\usepackage{color}
\usepackage{lmodern}
\usepackage[left=1in,right=1in,top=1in,bottom=1.5in]{geometry}

\makeatletter
\theoremstyle{definition}
\newtheorem*{rep@theorem}{\rep@title}
\newcommand{\newreptheorem}[2]{%
\newenvironment{rep#1}[1]{%
 \def\rep@title{#2 \ref{##1}}%
 \begin{rep@theorem}}%
 {\end{rep@theorem}}}
\makeatother

\numberwithin{equation}{section}

\theoremstyle{definition}
\newtheorem{dfn}{Definition}[section]
\newtheorem{exmp}[dfn]{Example}
\newtheorem{thm}[dfn]{Theorem}
\newtheorem{lm}[dfn]{Lemma}
\newtheorem{prop}[dfn]{Proposition}
\newtheorem{crl}[dfn]{Corollary}

\newtheorem{thml}{Theorem}
\newreptheorem{lm}{Lemma}

\newtheorem{crll}{Corollary}

\theoremstyle{remark}
\newtheorem{rmk}[dfn]{Remark}
\newtheorem{cla}{Claim}[dfn]

\newcommand{\e}{\varepsilon}
\newcommand{\pt}{\partial}
\newcommand{\mc}[1]{\mathcal{#1}}
\newcommand{\mf}[1]{\mathfrak{#1}}
\newcommand{\ol}[1]{\overline{#1}}
\newcommand{\ra}{\rightarrow}
\newcommand{\R}{\mathbb{R}}
\newcommand{\Z}{\mathbb{Z}}
\newcommand{\C}{\mathbb{C}}
\renewcommand{\H}{\mathbb{H}}
\newcommand{\E}{\mathbb{E}}
\renewcommand{\S}{\mathbb{S}}
\newcommand{\len}{{\rm len}}

\newcommand{\area}{{\rm area}}
\newcommand{\conv}{{\rm conv}}
\newcommand{\clconv}{\overline{\rm conv}}
\newcommand{\ang}{{\angle}}
\newcommand{\dev}{{\rm dev}}
\renewcommand{\area}{{\rm area}}
\newcommand{\curv}{{\rm curv}}
\newcommand{\inter}{{\rm int}}

\newcommand{\dS}{d\mathbb{S}}

\renewcommand{\hat}{\widehat}
\renewcommand{\tilde}{\widetilde}

\DeclareMathOperator{\sgn}{sgn}
\let\sp\relax
\DeclareMathOperator{\sp}{sp}
\DeclareMathOperator{\cosp}{cosp}
\DeclareMathOperator{\sys}{sys}
\DeclareMathOperator{\tr}{tr}
\DeclareMathOperator{\arcsinh}{arcsinh}
\DeclareMathOperator{\curl}{curl}

\title{Hyperbolic 3-manifolds with boundary of polyhedral type}
\author{Roman Prosanov \thanks{This research was funded in whole by the Austrian Science Fund (FWF) \url{https://doi.org/10.55776/ESP12}. For open access purposes, the author has applied a CC BY public copyright license to any author-accepted manuscript version arising from this submission.}}
\date{}

\begin{document}
\maketitle
\abstract{Let $M$ be a compact orientable 3-manifold with hyperbolizable interior and non-empty boundary such that all boundary components have genii at least 2. We study an Alexandrov-Weyl-type problem for convex hyperbolic cone-metrics on $\partial M$. We consider a class of hyperbolic metrics on $M$ with convex boundary, which we call \emph{bent} metrics, and which naturally generalize hyperbolic metrics on $M$ with convex polyhedral boundary. We show that for each convex hyperbolic cone-metric $d$ on $\partial M$, with few simple exceptions, there exists a bent metric on $M$ such that the induced intrinsic metric on $\partial M$ is $d$. Next, we prove that if a bent realization is what we call controllably polyhedral, then it is unique up to isotopy. We exhibit a large subclass of hyperbolic cone-metrics on $\partial M,$ called \emph{balanced}, which is open and dense among all convex hyperbolic cone-metrics in the sense of Lipschitz topology, and for which we show that their bent realizations are controllably polyhedral. We additionally prove that any convex realization of a convex hyperbolic cone-metric on $\partial M$ is bent. 

Finally, we deduce that there exists an open subset of the space of convex cocompact metrics on the interior of $M$, including all metrics with polyhedral convex cores, such that the metrics in this subset are (1) globally rigid with respect to the induced intrinsic metrics on the boundaries of their convex cores; (2) infinitesimally rigid with respect to their bending laminations. This gives partial progress towards conjectures of W.~Thurston.}

\section{Introduction}

\subsection{Motivation}

Two remarkable related results of convex geometry state that (1) every non-negatively curved metric on the 2-sphere can be realized as the induced intrinsic metric on the boundary of a convex body in the Euclidean 3-space $\E^3$, and (2) convex bodies in $\E^3$ are uniquely determined by the intrinsic geometry of the boundary. This story has two classical chapters: polyhedral and smooth, though there is a common generalization obtained in the works of Alexandrov and Pogorelov. It was further extended to convex bodies in the other two model constant-curvature spaces, $\S^3$ and $\H^3$. (We will say \emph{model spaces} for short for all three.) In this paper we will be mostly interested in the polyhedral side of the story, and will generalize these results to the case of non-trivial topology, for which purpose we need to restrict ourselves to the realm of hyperbolic geometry.

Let $P$ be a compact convex polyhedron in one of the model spaces. The boundary of $P$ is the topological 2-sphere, and it is not hard to see that the induced intrinsic metric on the boundary is what is called a convex cone-metric.

\begin{dfn}
Let $S$ be a surface. We say that $d$ is a \emph{Euclidean cone-metric} (resp. \emph{spherical} or \emph{hyperbolic}) on $S$ if it is locally isometric to the Euclidean plane (resp. to the standard sphere or the hyperbolic plane) except a discrete set of points where it is locally isometric to a Euclidean cone (resp. spherical or hyperbolic) of total angle $\neq 2\pi$. The exceptional points are called \emph{vertices} of $d$. We say that $d$ is \emph{convex} if all cone angles are $<2\pi$. 
\end{dfn}

We note that we allow the set of vertices to be empty, in which case $d$ is a smooth metric on $S$ of constant curvature $0$ or $\pm 1$.

In~\cite{Ale} Alexandrov proved 

\begin{thm}
\label{alex}
Every Euclidean convex cone-metric on the 2-sphere can be realized on the boundary of a convex polyhedron $P \subset \E^3$, which is unique up to isometry.
\end{thm}

Later Alexandrov extended Theorem~\ref{alex} to the other model spaces, see~\cite{Ale2} and Alexandrov's books~\cite{Ale4, Ale3}. The existence part of Theorem~\ref{alex} is a polyhedral counterpart to the renowned \emph{Weyl problem}, which asks if a smooth metric of positive curvature on the 2-sphere can always be realized as the induced intrinsic metric on the boundary of a convex body in $\E^3$. It was resolved positively by Nirenberg~\cite{Nir} employing the methods of partial differential equations (preceded by a proof in the analytic class due to Lewy~\cite{Lew}), and independently by Alexandrov~\cite{Ale4} and Pogorelov~\cite{Pog}. The uniqueness part of Theorem~\ref{alex} extends the classical Legendre--Cauchy theorem on the rigidity of polyhedra with the same combinatorics, and is parallel to the rigidity of smooth convex bodies in $\E^3$ established by Herglotz~\cite{Her} and preceded by a proof in the analytic class due to Cohn-Vossen~\cite{CoV}. 

In 70s Thurston rocked the field of low-dimensional topology by formulating his geometrization conjecture (whose proof was finalized by Perelman), showing many its spectacular corollaries and providing non-trivial evidences in its favor. One of the main parts of the conjecture highlighted the unexpected ubiquity and diversity of hyperbolic manifolds in dimension 3. This attracted many geometers to a deeper study of hyperbolic 3-manifolds and their properties. Some attention was paid to hyperbolic 3-manifolds with convex boundary. In particular, Labourie proved in~\cite{Lab} an analogue of the Weyl problem for compact hyperbolic 3-manifolds with boundary. For the rest of the paper, let $M$ be an oriented smooth compact 3-manifold with non-empty boundary, distinct from the solid torus and from the 3-ball, admitting a hyperbolic metric with convex boundary, and let $N$ be the interior of $M$. The result of Labourie is

\begin{thm}
\label{lab}
For every smooth Riemannian metric $d$ on $\pt M$ of Gaussian curvature $>-1$ there exists a hyperbolic metric $g$ on $M$ such that the boundary is convex and the induced intrinsic metric is isometric to $d$.
\end{thm}

Thanks to the foundational work of Mostow~\cite{Mos} it is known that for any two hyperbolic metrics on a closed manifold of dimension at least 3 there exists an isometry between them homotopic to identity. In dimension 3 this was strengthened by Gabai--Meyerhoff--Thurston~\cite{GMT} to isotopy.

\begin{thm}
\label{GBM}
Every two hyperbolic metrics on a closed 3-manifold are isotopic.
\end{thm}

Theorem~\ref{GBM} together with the rigidity of smooth convex bodies in $\H^3$ make it reasonable to conjecture that the geometry of compact hyperbolic 3-manifolds with non-empty smooth convex boundary should be dictated by their topology and by the intrinsic geometry of the boundary. This was confirmed by Schlenker in~\cite{Sch} provided that the boundary is strictly convex, which means that the shape operator of the boundary is everywhere positive-definite.

\begin{thm}
\label{sch}
Let $g_1$, $g_2$ be two hyperbolic metrics on $M$ such that the boundary is smooth, strictly convex and the induced intrinsic boundary metrics are isotopic. Then $g_1$ and $g_2$ are isotopic.
\end{thm}

We note that in~\cite{Sch} Schlenker also gave another prove of Theorem~\ref{lab} and considered the toy case of the solid torus. 

Curiously enough, there were no polyhedral counterparts of Theorems~\ref{lab} and~\ref{sch}, and the main purpose of the present paper is to fill this gap. One of the main difficulties is that convex hyperbolic cone-metrics on $\pt M$ admit convex realizations that are not really polyhedral and that are somewhat more difficult to handle. Another difficulty is that the core proof-ideas of Theorems~\ref{lab} and~\ref{sch} in~\cite{Lab} and~\cite{Sch} are quite specific to the smooth setting, and do not admit direct polyhedral generalizations. In the same time, standard approaches to the rigidity of polyhedra in the model spaces do not have direct generalizations to non-trivial 3-manifolds with polyhedral boundary. Thus, although our general framework (\emph{the continuity method}) is similar to the one employed by Schlenker in~\cite{Sch}, in our implementation of all main technical steps we have to choose a different route to follow.

\subsection{Our results}
\label{ressec}

Let us specify what we mean by a hyperbolic metric on $M$ with (convex) polyhedral boundary.

\begin{dfn}
We say that a hyperbolic metric $g$ on $M$ is \emph{polyhedral} if $\partial M$ is locally convex and each point $p \in \partial M$ has a neighborhood isometric to $U \cap P$, where $P \subset \H^3$ is a convex polyhedron in $\H^3$ and $U \subset \H^3$ is an open domain intersecting $\pt P$.
\end{dfn}

Now we generalize this definition.

\begin{dfn}
We say that a hyperbolic metric $g$ on $M$ is \emph{bent} if $\partial M$ is locally convex and each point $p \in \partial M$, except finitely many, belongs to the relative interior of a geodesic segment. The exceptional points are called \emph{vertices} of $g$.
\end{dfn}

Here and further by a geodesic segment we mean a geodesic segment of $g$. If we need to consider a geodesic segment in $\partial M$ with respect to the induced intrinsic metric, possibly bent in $(M, g)$, we will be referring to it as to an \emph{intrinsic geodesic segment}. It is easy to see that a polyhedral hyperbolic metric on $M$ is bent. Moreover, one can notice that in the model spaces a closed convex surface satisfying the bent condition is actually a polyhedron. This is no longer true for compact hyperbolic 3-manifolds with boundary. Well-known examples of non-polyhedral bent hyperbolic metrics without vertices are given by the convex cores of so-called convex cocompact hyperbolic 3-manifolds, which we define now.

\begin{dfn}
A subset $C$ of a complete Riemannian manifold is called \emph{totally convex} if it is closed and contains every geodesic segment between any two points of $C$.
\end{dfn}

\begin{dfn}
A hyperbolic metric $\overline g$ on $N$ is called \emph{convex cocompact} if it is complete and $(N, \overline g)$ contains a compact totally convex subset. 
\end{dfn}

It is not hard to see that $M$ admits a hyperbolic metric with convex boundary if and only if its interior $N$ admits a convex cocompact metric. Moreover, for every hyperbolic metric $g$ on $M$ with convex boundary there exists a unique convex cocompact metric $\ol g$ on $N$ and a unique isometric embedding $(M, g) \hookrightarrow (N, \ol g)$, where the uniqueness is in some sense, which we will clarify in Section~\ref{cocosec}. We will mostly use the notation $\ol g$ to denote the metric on $N$ obtained from $g$ by this construction, but sometimes we will just write $\ol g$ to denote a convex cocompact metric on $N$ in the absence of metric $g$. Our restrictions on $M$ imply that $M$ is irreducible and atoroidal. Since $M$ is distinct from the solid torus and the 3-ball, this particularly means that all components of $\pt M$ have genii at least 2. Thurston's hyperbolization theorem implies that the converse holds: every compact 3-manifold satisfying these topological conditions admits a hyperbolic metric with convex boundary, see~\cite{Mar2, Kap}. In particular, there are plenty of such 3-manifolds.

\begin{dfn}
\label{cocodef}
The \emph{convex core} $C(\ol g)$ of a convex cocompact hyperbolic metric $\ol g$ on $N$ is the inclusion-minimal closed totally convex subset of $(N, \ol g)$. 
\end{dfn}

We will recall in Section~\ref{cocosec} that the convex core is well-defined and, except few degenerate cases, is homeomorphic to $M$. It was shown by Thurston in~\cite{Thu} that the induced intrinsic metric on $\pt C(\ol g)$ is hyperbolic, but $\pt C(\ol g)$ may be bent in $(N, \ol g)$ along an irrational geodesic lamination, in which case it is neither smooth nor polyhedral. However, $\pt C(\ol g)$ satisfies our bent condition. In Section~\ref{bentsec} we will sketch an argument showing the existence of bent metrics with vertices. 

We need to introduce two important subclasses of polyhedral hyperbolic metrics on~$M$.

\begin{dfn}
\label{scpdef}
We say that a polyhedral hyperbolic metric $g$ on $M$ is \emph{strictly polyhedral} if $\partial M$ is at positive distance from $C(\overline g)$. We say that a polyhedral hyperbolic metric $g$ on $M$ is \emph{controllably polyhedral} if $\partial M \cap C(\overline g)$ is empty or consists only of simple closed geodesics.
\end{dfn}

In particular, every strictly polyhedral metric is controllably polyhedral. We also need to distinguish a special class of metrics on $\pt M$.

\begin{dfn}
\label{fucdef}
We say that a convex hyperbolic cone-metric $d$ on $\pt M$ is \emph{proto-Fuchsian} if~$M$~is homeomorphic to $S \times [-1, 1]$ for a surface $S$ and either (1) $S$ is closed, $d$ does not have vertices and the restrictions of $d$ to the components of $\pt M$ are isotopic as metrics on $S$; or (2) $S$ is compact with non-empty boundary and $(\pt M, d)$ is the double of a hyperbolic metric on $S$ with convex piecewise-geodesic boundary and with no cone-points in the interior of $S$. Otherwise, we say that $d$ is \emph{non-proto-Fuchsian}.
\end{dfn}

Note that in case (2) $M$ is a handlebody. Now we can formulate the main result of our paper.

\begin{thml}
\label{main1}
For each non-proto-Fuchsian convex hyperbolic cone-metric $d$ on $\pt M$ there exists a bent metric $g$ on $M$ such that the induced intrinsic metric on $\partial M$ is $d$. If either (1) $g$ is controllably polyhedral, or (2) $g$ is polyhedral and does not have vertices, then $g$ is unique up to isotopy (among bent metrics on $M$).
\end{thml}

We make few remarks to explain the roles of classes of metrics participating in Theorem~\ref{main1}.
First, when $d$ is proto-Fuchsian, its realization degenerates to a surface and it is not possible to realize such $d$ by a hyperbolic metric on $M$ with convex boundary. Thereby, the realization part of Theorem~\ref{main1} covers all realizable convex hyperbolic cone-metrics on $\pt M$. 
Second, in the smooth setting the condition that the Gaussian curvature of a metric $d$ is $>-1$ implies that a convex realization of $d$ is at positive distance from the convex core. In the polyhedral setting this corresponds to a strictly polyhedral boundary. Hence, Theorem~\ref{main1} includes a polyhedral counterpart to Theorem~\ref{sch}. This research was actually initiated by an attempt to prove the rigidity in the strictly polyhedral class. However, we found out that it is not really possible without extending our consideration to the larger class of controllably polyhedral manifolds. Third, the class of bent metrics on $M$ is the right class to consider convex realizations of convex hyperbolic cone-metrics on $\pt M$. Indeed, we will prove in Section~\ref{bentsec} that the induced metric on the boundary of a bent metric on $M$ is a convex hyperbolic cone-metric. Furthermore, we will establish

\begin{thml}
\label{sub1}
Let $g$ be a hyperbolic metric on $M$ such that the boundary is convex and the induced intrinsic metric $d$ on $\pt M$ is a hyperbolic cone-metric. Then $g$ is bent.
\end{thml}

Our approach to Theorem~\ref{main1} implies a special corollary, which we present now. By $\mc{CH}(N)$ we denote the space of convex cocompact hyperbolic metrics on $N$ up to isotopy, endowed with a natural topology, which we define in Section~\ref{cocosec}. As we mentioned, the induced intrinsic metric on the boundary of the convex core of a convex cocompact metric is hyperbolic. This defines a map $\mc I: \mc{CH}(N) \ra \mc T(\pt M)$, where $\mc T(\pt M)$ is the Teichm\"uller space of $\pt M$. It follows from the work of Labourie~\cite{Lab} and from an approximation argument that this map is surjective. (The present paper provides an alternative proof of the surjectivity of this map as a particular case of Theorem~\ref{main1}.) There is another object called \emph{the bending lamination} (some authors say \emph{pleating lamination}) of the convex core measuring how much the boundary is bent, see Section~\ref{cocoresec} for details. (In the polyhedral case this is just the data of the exterior dihedral angles.) This defines a map $\mc I^*: \mc{CH}(N) \ra \mc {ML}(\pt M)$, where $\mc {ML}(\pt M)$ is the space of \emph{measured geodesic laminations} on $\pt M$. In this case the map is not surjective, and its image is described in the works~\cite{BO, Lec} of Bonahon--Otal and Lecuire. Note that the space $\mc{ML}(\pt M)$ lacks a natural differential structure, but has a natural PL-structure. 

Thurston conjectured that $\mc I$ is a homeomorphism and $\mc I^*$ is a homeomorphism onto the image outside the Fuchsian locus. The injectivity of these maps remains an open problem, though it was shown in~\cite{BO} that if $\ol g$ has a non-degenerate polyhedral convex core, then $\mc I^*(\ol g)$ has a single $\mc I^*$-preimage.\footnote{After this paper was finished, Thurston's conjecture for $\mc I^*$ was resolved in a spectacular work \cite{DS} by Dular--Schlenker. It is interesting that one of the ingredients of their proof does not work for $\mc I$.} By the work~\cite{Bon} of Bonahon, the map $\mc I$ is $C^1$ and the map $\mc I^*$ is tangentiable, which is a relaxation of the notion of differentiability to the PL-setting. 
We show

\addtocounter{thml}{1}
\begin{crll}
\label{main2}
Let $\ol g \in \mc{CH}(N)$ have non-degenerate polyhedral convex core. Then there exists a neighborhood $U$ of $\ol g$ in $\mc{CH}(N)$ such that\\
(1) the restriction of $\mc I$ to $U$ is a $C^1$-diffeomorphism onto the image. Moreover, for each $\ol g' \in U$ the metric $\mc I(\ol g')$ has a single $\mc I$-preimage in $\mc {CH}(N)$.\\
(2) For each $\ol g' \in U$ the tangent map of $\mc I^*$ is an isomorphism at $\ol g'$.
\end{crll} 

The key point here is the infinitesimal rigidity of $\mc I$ at $\ol g$. Together with the inverse function theorem it implies the first claim of Part (1). We will show that local rigidity implies global rigidity, obtaining the second claim of Part (1). The infinitesimal rigidity problems for $\mc I$ and for $\mc I^*$ are equivalent as implied by the works of Bonahon~\cite{Bon, Bon4}. Note that the rigidity part of~\cite{BO} does not imply infinitesimal or local injectivity of $\mc I$ or $\mc I^*$ at $\ol g$.

A convex realization of a convex hyperbolic cone-metric might be also obtained from Theorem~\ref{lab} together with an approximation argument, though some non-trivial work is required to exclude possible degenerations. In the case when $M$ is homeomorphic to $S \times [-1, 1]$ this approach was implemented by Slutskiy in~\cite{Slu}. Theorem~\ref{sub1} says that such a realization is necessarily bent, leading to an alternative proof of the existence part of Theorem~\ref{main1}. The most non-trivial contribution of the present paper is the proof of the uniqueness part of Theorem~\ref{main1}. The proof of Theorem~\ref{sub1} follows the ideas of Olovyanishnikov, who similarly proved that a convex realization of a Euclidean cone-metric on the 2-sphere in the Euclidean 3-space is a polyhedron~\cite{Olo}. 
We provide Theorem~\ref{sub1} to complement our main result Theorem~\ref{main1}.

As we already mentioned, our proof of Theorem~\ref{main1} is obtained from a generalization of the classical continuity method popularized by Alexandrov. Let $V \subset \pt M$ be a finite set, possibly empty. By $\mc D_c(\pt M, V)$ we denote the space of convex hyperbolic cone-metrics on $\pt M$ with the vertex set $V$ up to isotopy. By $\mc P_b(M, V)$ we denote the space of bent hyperbolic metrics on $M$ with the vertex set $V$ up to isotopy. (There are some technical restrictions on the isotopies, see details in Sections~\ref{conesec} and~\ref{bentsec}.) We will endow these spaces with some natural smooth topologies making them finite-dimensional manifolds of the same dimension. There is the realization map
$$\mc I_V: \mc P_b(M, V) \ra \mc D_c(\pt M, V)$$
sending a bent metric on $M$ to the induced intrinsic metric on the boundary. 

We will frequently abuse the notation and write $g  \in \mc P_b(M, V)$ or $d \in \mc D_c(\pt M, V)$ meaning that $g$ or $d$ are metrics with their isotopy classes belonging to the spaces $\mc P_b(M, V)$ and $\mc D_c(\pt M, V)$. The uniqueness part of Theorem~\ref{main1} follows from the following intermediate steps. 

\begin{prop}
\label{prop1}
If $g \in \mc P_b(M, V)$ is locally rigid, i.e., the map $\mc I_V$ is a local homeomorphism around $g$, then $g$ is globally rigid, i.e., for every other $g' \in \mc P_b(M, V)$ such that $\mc I_V(g)=\mc I_V(g')$, the metrics $g$ and $g'$ are isotopic.
\end{prop}


\begin{prop}
\label{prop2-3}
Controllably polyhedral metrics on $M$ and polyhedral metrics without vertices are locally rigid.
\end{prop}


Proposition~\ref{prop2-3} for controllably polyhedral metrics is proved in three steps: we show that the subset of controllably polyhedral metrics is open in $\mc P_b(M, V)$, that the map $\mc I_V$ is differentiable over this subset, and that the differential is non-degenerate. We remark that two most technically complicated places in our paper are the proofs of the latter two claims occupying Sections~\ref{diffsec} and~\ref{infrigsec}. Interestingly, while the differential $d\mc I_V$ is continuous over strictly polyhedral metrics, we cannot really show this over all controllably polyhedral metrics, but we found a way to deal with it. To show the infinitesimal rigidity we develop the theory of affine automorphic vector fields, which allow us to connect together separate approaches to the infinitesimal theory of hyperbolic cone-3-manifolds.

Contrary to the smooth case, in our setting it is not immediately obvious if there are many convex hyperbolic cone-metrics on $\pt M$ admitting controllably polyhedral realizations. We will, however, show that this is a generic case, by exhibiting a large subclass of such metrics, which we call \emph{balanced}. The actual definition is slightly technical, and we postpone it to Section~\ref{balsec}, but the heuristic is that a cone-metric is balanced if the set of vertices is sufficiently dense. We~have

\begin{prop}
\label{prop4}
The subset of balanced metrics is open and dense in the space of convex hyperbolic cone-metrics on a closed surface $S$ with respect to the Lipschitz topology.
\end{prop}

\begin{prop}
\label{prop5}
Every convex realization of a convex balanced hyperbolic cone-metric on $\pt M$ is controllably polyhedral.
\end{prop}


\subsection{Related work}
\label{relsec}

Here we review some related work. The only previously known case of the rigidity part of Theorem~\ref{main1} was established by Fillastre~\cite{Fil}, who considered the case of so-called \emph{Fuchsian manifolds} with convex boundary. These are homeomorphic to ${S \times [-1, 1]}$, where $S$ is a closed surface of genus greater than one, and contain an embedded totally geodesic surface isotopic to $S \times \{0\}$. Fillastre also proved that for every convex hyperbolic cone-metric on $S$ with vertices there exists a Fuchsian manifold with convex boundary (which is necessarily strictly polyhedral) realizing it. In~\cite{Fil3} Fillastre went further to allow so-called ideal and hyperideal vertices. There is a similar work of Fillastre--Izmestiev~\cite{FI} on the case when the manifold is homeomorphic to $T \times (-1, 1]$, where $T$ is the 2-torus, but this does not fit to our restrictions on $M$.

Similar results were obtained for the model Lorentzian spaces of constant curvature. In the case of the \emph{de Sitter 3-space} $\dS^3$ (of constant curvature 1) there is a well-known duality between points in $\dS^3$ and planes in $\H^3$, which allows to consider the corresponding questions as ``dual'' to questions in $\H^3$. The cornerstone is the work~\cite{HR} of Rivin--Hodgson, where they gave a description of the dual metrics on convex polyhedra in $\H^3$. It builds on a previous work of Andreev~\cite{And} characterizing hyperbolic polyhedra with acute dihedral angles. It was further generalized by Rivin~\cite{Riv} to \emph{ideal} (non-compact) hyperbolic polyhedra. This is of special interest as it allowed to resolve a long-standing question of Steiner to describe the combinatorial types of polyhedra admitting Euclidean realizations with all vertices on the standard sphere. The work~\cite{HR} was generalized to the smooth setting by Schlenker~\cite{Sch2}, and was recently generalized by the author to the setting of strictly polyhedral hyperbolic manifolds~\cite{Pro3}. The case of strictly polyhedral manifolds with ideal vertices was previously considered by Schlenker in an unpublished manuscript~\cite{Sch3}. In~\cite{Fil3} Fillastre examined similar questions for Fuchsian Lorentzian spacetimes, and in~\cite{FI2} Fillastre--Izmestiev dealt with the case of dual metrics on $T \times (-1, 1]$. Quite recently Smith~\cite{Smi2} resolved the smooth Weyl problem for \emph{Minkowski spacetimes}. In~\cite{Tam} Tamburelli solved the realization part for \emph{anti-de Sitter spacetimes}.

Up to now we were mentioning in parallel the results on convex realizations in the smooth and in the polyhedral cases. It is natural to expect that there should be some common generalization of these settings. This is the topic of the Alexandrov--Pogorelov theory of general convex surfaces. Consider convex surfaces in the Euclidean 3-space $\E^3$. Alexandrov managed to give a complete intrinsic description of the induced intrinsic metrics on them. These are so-called \emph{CBB(0) metrics} in modern terminology, whose definition is slightly technical and we refer to the textbooks~\cite{Ale3, BBI}. This notion saw a huge development in last decades due to the connection with degenerations of Riemannian manifolds. With the help of an approximation argument, Alexandrov showed in~\cite{Ale} that every CBB(0) metric on the 2-sphere is realized as the induced intrinsic metric on the boundary of a convex body. The respective rigidity question was solved by Pogorelov~\cite{Pog} and is known as notoriously difficult, another approach was given by Volkov~\cite{Vol2}. In recent years there was a streamline of works transferring the general Alexandrov realization theorem to various settings, see~\cite{FIV, FS2, Lab2, Slu}. Of special relevance is the above-mentioned work of Slutskiy~\cite{Slu}, where he considered general convex realization results for $M$ homeomorphic to $S \times [-1, 1]$. However, there are no generalizations of the Pogorelov rigidity theorem, except that to the other model spaces (due to Pogorelov himself, see~\cite{Pog}, and to Milka~\cite{Mil2}), and except the recent work of the author~\cite{Pro2} on the case of Fuchsian manifolds.

Speaking of the injectivity problems for the maps $\mc I$ and $\mc I^*$, in the work~\cite{Bon5} Bonahon considered manifolds $N$ homeomorphic to $S \times (-1, 1)$ for a closed surface $S$ of higher genus and showed that $\mc I^*$ is a homeomorphism onto the image near the so-called \emph{Fuchsian locus} of $\mc {CH}(N)$. In the work~\cite{Ser} Series proved the conjecture on $\mc I^*$ for manifolds homeomorphic to $S \times (-1, 1)$, where $S$ is a once-punctured torus (the latter manifolds have cusps and are outside of the scope of this paper). As the infinitesimal rigidity problems for $\mc I$ and $\mc I^*$ are equivalent, these works also imply the corresponding infinitesimal and local rigidity results for $\mc I$. We also remark that bending laminations can be considered as degenerate cases of the dual metrics, which we mentioned just above.

The questions whether the maps $\mc I$ and $\mc I^*$ are homeomorphisms onto the images are very similar to the known parametrization of the space $\mc{CH}(N)$. Every convex cocompact hyperbolic metric on $N$ determines a conformal structure ``at infinity'' on $\pt M$, defining another map $\mc{CH}(N) \ra \mc T(\pt M)$ to the Teichm\"uller space $\mc T(\pt M)$. The famous Bers \emph{double uniformization} theorem~\cite{Ber} asserts that this map is a homeomorphism in the case of $N$ homeomorphic to $S \times (-1, 1)$. It was extended to the case of arbitrary $N$ admitting a convex cocompact hyperbolic metric (and, more generally, a geometrically finite hyperbolic metric) in the collective works of Ahlfors, Bers, Kra, Marden, Maskit and Sullivan. See, e.g., the discussion in~\cite[Chapter 5]{Mar2}. A long string of efforts was devoted to generalize it to the case of all complete hyperbolic metrics on $N$ (where we need other type of invariants at infinity, called \emph{ending laminations}) culminated in the works~\cite{Min, BCM}. Recently Chen and Schlenker~\cite{CH} explored the case when on each boundary component one prescribes either a conformal structure at infinity or a (smooth) induced metric/dual metric.

A recognized trouble with the continuity method is that while it gives a proof of the existence of a convex realization, it provides no recipe how to construct it. The first approach to overcome it in the case of Theorem~\ref{alex} was undertook by Volkov in~\cite{Vol} by a variational method, and another approach was proposed in~\cite{BI} by Bobenko--Izmestiev. It is interesting to ask if any kind of constructive approach can be given to the realization part of Theorem~\ref{main1}. In the work~\cite{Pro2} the author built on this approach to show the rigidity of Fuchsian manifolds with general convex boundary. Hence, it can be helpful also on the way to generalize Pogorelov's rigidity to other hyperbolic manifolds, and to prove the injectivity of $\mc I$ and $\mc I^*$ for non-polyhedral convex cores.

The key infinitesimal rigidity statement in our paper is obtained with the help of the works on the rigidity of hyperbolic cone-3-manifolds, see Section~\ref{cone3sec} for the definition. There was a long string of results showing their rigidity in various settings, see~\cite{HK, MM3, Mon, Wei, Wei2}. We specifically employ the work~\cite{Wei2} of Weiss. The rigidity of hyperbolic cone-3-manifolds has interesting applications to some other geometry problems, see~\cite{BLP, Bro, HK2, Pro} for examples.
\vskip+0.2cm

\textbf{Acknowledgments.} We thank Fran\c{c}ois Fillastre, Ivan Izmestiev and Jean-Marc Schlenker for many valuable discussions and thank Michael Eichmair for his support.

\section{Preliminaries}

For a background in hyperbolic geometry we refer to~\cite{BP, Mar2, Mar3}. We will mostly rely either on the hyperboloid model or on the Klein model of the hyperbolic space. We will frequently use the hyperbolic version of the well-known Busemann--Feller lemma~\cite{Mil}

\begin{lm}
\label{bflemma}
Let $C \subset \H^3$ be a closed convex set with non-empty interior, $\tau \subset (\H^3 \backslash \inter(C))$ be a rectifiable curve and $\tau' \subset \pt C$ be its nearest-point projection on $C$. Then $\tau'$ is rectifiable and its length is at most the length of $\tau$. 
\end{lm}

We mentioned in Section~\ref{relsec} that the space of oriented planes of $\H^3$ can be identified with \emph{the de Sitter space} $\dS^3$, the model Lorentzian space of curvature 1. Vice versa, points of $\H^3$ correspond to space-like planes in $\dS^3$. Sometimes we will be using some simple facts on this duality. We refer to~\cite{HR, Sch2, Pro3} for more details.

\subsection{Trigonometry}

In this subsection we review some elementary facts from hyperbolic trigonometry that we will use in the remainder. 

By a trapezoid we mean a convex hyperbolic quadrilateral $p_1p_2q_2q_1$ such that $\ang p_1q_1q_2=\ang p_2q_2q_1=\pi/2$. Take a trapezoid $p_1p_2q_2q_1$ and denote the lengths $p_1q_1$, $p_2q_2$, $p_1p_2$, $q_1q_2$ by $a$, $b$, $c$, $\gamma$ and the angles $\ang q_1p_1p_2$, $\ang q_2p_2p_1$ by $\beta$ and $\alpha$ respectively. In what follows we will be using the analogues of the cosine and sine laws for trapezoids. The proofs are straightforward following, e.g., the approach of Thurston in~\cite[Chapter 2.6]{Thu}.

\begin{lm}[The cosine laws]
\label{coslaw}
For a trapezoid we have
$$\cos (\alpha)=\frac{\sinh (b)\cosh (c) - \sinh (a)}{\cosh (b)\sinh (c)},$$
$$\cosh (\gamma) = \frac{\sinh (a) \sinh (b) +\cosh (c)}{\cosh (a) \cosh (b)}.$$
\end{lm}

\begin{lm}[The sine law]
\label{sinlaw}
For a trapezoid we have
$$\frac{\sin (\alpha)}{\cosh (a)}=\frac{\sin (\beta)}{\cosh (b)}=\frac{\sinh (\gamma)}{\sinh (c)}.$$
\end{lm}

Next, we recall few standard formulas on right-angled hyperbolic pentagons and hexagons. Proofs can be found in~\cite[Chapter 2.3]{Bus}.

\begin{lm}
\label{pent}
Consider a right-angled hyperbolic pentagon with side lengths $a$, $b$, $\alpha$, $c$, $\beta$ clockwise. Then we have
$$\cosh(c)=\sinh(a)\sinh(b),$$
$$\cosh(c)=\coth(\alpha)\coth(\beta).$$
\end{lm}

\begin{lm}
\label{hexcos}
Consider a right-angled hyperbolic hexagon with side lengths $a$, $\gamma$, $b$, $\alpha$, $c$, $\beta$ clockwise. Then we have
$$\cosh(\alpha)=\frac{\cosh(b)\cosh(c)+\cosh(a)}{\sinh(a)\sinh(b)}.$$
\end{lm}

The next lemma, due to Kubota~\cite{Kub}, is a hyperbolic analogue of the well-known Ptolemy theorem.

\begin{lm}
\label{ptolem}
Consider a hyperbolic quadrilateral inscribed in a circle with side lengths $a$, $b$, $c$ and $d$ clockwise and diagonal lengths $x$ and $y$. Then
$$\sinh\left(\frac{a}{2}\right)\sinh\left(\frac{c}{2}\right)+\sinh\left(\frac{b}{2}\right)\sinh\left(\frac{d}{2}\right)=\sinh\left(\frac{x}{2}\right)\sinh\left(\frac{y}{2}\right).$$
\end{lm}

Last, we will need a geometric estimate on hyperbolic pairs of pants following from

\begin{lm}
\label{hex}
Let $H=p_1p_2p_3p_4p_5p_6 \subset \H^2$ be a convex right-angled hexagon with sides $a, \gamma, b, \alpha, c, \beta$ ($p_1p_2=a$ and then cyclically). Assume that for some $\Delta>0$ we have $a, b, c \geq \Delta$. Then there exist a number $\delta=\delta(\Delta)>0$ and a point $p \in H$ such that
$$d_{\H^2}(p, \partial H) \geq \delta.$$
\end{lm}

\begin{proof}
Take any $\delta_1 > 0$. First suppose that $\alpha \geq \delta_1$. It is easy to see that there exists $\delta_2=\delta_2(\delta_1, \Delta)>0$ such that since $b \geq \Delta$ and $\alpha \geq \delta_1$, the inradius of the right-angled triangle $p_3p_4p_5$ is at least $\delta_2$.  We can take the incenter of $p_3p_4p_5$ as $p$. 

The same procedure applies if $\beta$ or $\gamma$ are at least $\delta_1$. Hence, it remains to suppose that all $\alpha,\beta,\gamma \leq \delta_1$.

Let $q$ be the intersection of the altitudes and let the foot of the altitude divide $p_1p_2$ in the segments of lengths $a_1, a_2$ respectively. We define similarly  $\gamma_1, \gamma_2$, $b_1, b_2$ and so on (in the cyclic order). By $h_a$ denote $d_{\H^2}(q, p_1p_2)$. 
Then
$$\tanh (h_a) = \cosh (a_1) \tanh (\beta_2).$$
From Lemma~\ref{pent} we can also express
$$\tanh (\beta_2) = \frac{\coth (\gamma)}{\cosh (a)}.$$
Substitute this and get 
$$\tanh (h_a)=\frac{\cosh (a_1) \cosh (\gamma)}{\cosh (a) \sinh (\gamma)}=\frac{\cosh (a_1) \cosh (\gamma) \sinh (\beta)}{\cosh (\beta) \cosh (\gamma) + \cosh (\alpha)}=$$ $$=\frac{\sinh (a_1)\cosh (a_1)\cosh(\gamma) \sinh (\beta)}{\sinh (a_1)(\cosh (\beta) \cosh (\gamma) + \cosh (\alpha))}=\frac{\coth (a_1)\cosh(\gamma){\cosh (\alpha_1)}}{\cosh (\beta) \cosh (\gamma) + \cosh (\alpha)}.$$
Here we used Lemma~\ref{pent} and Lemma~\ref{hexcos}. It is not hard to see that under the condition $\alpha,\beta,\gamma \leq \delta_1$ the last expression is bounded away from zero. Hence, for some $\delta_3=\delta_3(\delta_1)>0$ we have $h_a \geq \delta_3$.

Let $q_a$ be the foot of the altitude to $p_1p_2$. Either $a_1$ or $a_2$ is $\geq \Delta/2$. Assume that $a_1 \geq \Delta/2$. Consider the right-angled triangle $qq_ap_1$. There exists $\delta_4=\delta_4(\delta_1, \Delta)>0$ such that since $a_1 \geq \Delta/2$ and $h_a \geq \delta_3$, the inradius of $qq_ap_1$ is at least $\delta_4$. In this case we can take the incenter of $qq_ap_1$ as $p$.

Hence, we can set $\delta$ as the minimum of $\delta_2$ and $\delta_4$.
\end{proof}

\begin{crl}
\label{pants}
There exists a continuous function $\delta=\delta(\Delta): \R_{>0} \rightarrow \R_{>0}$ such that if $P$ is a hyperbolic pair of pants with the length of each boundary component at least $\Delta>0$, then for some $p \in P$ we have $d(p, \partial P)>\delta$, where $d$ is the metric on $P$.
\end{crl}

Until the end of the paper we assume that we fixed a function $\delta$ from Corollary~\ref{pants}.

\subsection{Cone-metrics on surfaces}
\label{conesec}

Let $S$ be a closed surface and $V \subset S$ be a finite set of marked points. In this section we review some facts about hyperbolic cone-metrics on $S$. In what follows we will be saying just ``cone-metric'' always meaning ``hyperbolic cone-metric''. We denote the set of cone-points of a cone-metric $d$ by $V(d)$. We say that $d$ is a cone-metric on $(S, V)$ if $d$ is a cone-metric on $S$ and $V \supseteq V(d)$. For $v \in V$ we denote its cone-angle by $\kappa_v(d)$ or just $\kappa_v$ (which is equal to $2\pi$ if $v \notin V(d)$). 

\begin{dfn}
A \emph{triangulation} $\mathcal T$ of $(S, V)$ is a collection of simple disjoint arcs with endpoints in $V$ that cut $S$ into triangles. Here a triangle is a topological 2-disk with three marked points at the boundary (possibly coinciding as points of $V$). We call two triangulations of $(S, V)$ \emph{equivalent} if they are isotopic on $S$ by an isotopy fixing $V$. 
\end{dfn}

We denote the set of edges of $\mc T$ be $E(\mc T)$.

\begin{dfn}
A \emph{geodesic triangulation} of $(S, V, d)$ is a triangulation of $(S, V)$ such that all edges are geodesics in $d$. 
\end{dfn}

\begin{dfn}
Let $\mathcal T$ be a triangulation of $(S, V)$ and $d$ be a cone-metric on $(S, V)$. We say that $\mathcal T$ is \emph{realized} by $d$ if there is a geodesic triangulation of $(S, V, d)$ equivalent to $\mathcal T$.
\end{dfn}

It is not hard to notice that for every cone-metric $d$ on $(S, V)$ there exists a geodesic triangulation of $(S, V, d)$. Of special importance are \emph{Delaunay triangulations}. Generically a Delaunay triangulation of $(S, V, d)$ is unique, but in some cases it is not and we need to consider a more general \emph{Delaunay decomposition}.

\begin{dfn}
Let $d$ be a cone-metric on $(S, V)$. A \emph{Delaunay decomposition} of $(S, V, d)$ is a collection of simple disjoint geodesic arcs with endpoints in $V$ that cut $S$ into cells isometric to compact convex hyperbolic polygons such that each of the polygons can be inscribed in a circle and if we develop any two adjacent cells to $\H^2$, then the closed disk bounded by the circumcircle of each polygon does not contain vertices of the other polygon except the endpoints of the common edge. 
\end{dfn}

Note that here a priori by a circle we mean either a metric circle in $\H^2$, or a horocycle, or a hypercycle. However, it is known that actually the circumcircle of each cell in the Delaunay decomposition of a closed surface is a metric circle, see~\cite[Theorem 14]{GGLSW} or~\cite[Remark 3.7]{Pro}. Another well-known fact is 

\begin{thm}[\cite{MS}, Section 4]
A Delaunay decomposition of $(S, V, d)$ exists and is unique.
\end{thm}

\begin{dfn}
A \emph{Delaunay triangulation} of $(S, V, d)$ is a triangulation subdividing the Delaunay decomposition.
\end{dfn}

In particular, there are finitely many Delaunay triangulations, and any two can be connected by a finite sequence of flips via Delaunay triangulations.

For a map $f: (S_1, d_1) \rightarrow (S_2, d_2)$ between two compact metric spaces we define its \emph{distortion} ${\rm dist}(f)$ as
$${\rm dist}(f):=\sup_{p, q\in S_1,~p \neq q}\left|\ln\frac{d_2(f(p), f(q))}{d_1(p,q)}\right|.$$
We say that a sequence of metrics $\{d_i\}$ on a surface $S$ converges to a metric $d$ \emph{in the strong Lipschitz sense} if the identity maps $(S, d_i) \rightarrow (S, d)$ have distortions $\e_i \rightarrow 0$. By $\mf D(S, V)$ we denote the space of hyperbolic cone-metrics on $(S, V)$ endowed with the topology of the strong Lipschitz convergence. 

We need to introduce two equivalence relations on $\mf D(S,V)$. Let $H_0(S, V)$ be the group of self-homeomorphisms of $S$ fixing $V$ and isotopic to identity, endowed with the compact-open topology, and let $H_0^\sharp(S, V)$ be its normal subgroup consisting of homeomorphisms for which the isotopy can be chosen to fix $V$. Denote the respective quotients of $\mf D(S, V)$ by $\mc D(S, V)$ and $\mc D^\sharp(S,V)$. It follows that $\mc D(S, V)$ is the quotient of $\mc D^\sharp(S, V)$ by \emph{the pure braid group} of $(S,V)$, which is $B_0(S, V):=H_0(S, V)/H_0^\sharp(S,V)$. In most of the paper we will use the space $\mc D(S, V)$, but for technical reasons we also have to appeal to $\mc D^\sharp(S, V)$. We will abuse the notation, and denote by $d$ sometimes a metric, sometimes its class in $\mc D(S, V)$, but we will denote the elements of $\mc D^\sharp (S, V)$ by~$d^\sharp$.

For a triangulation $\mathcal T$ of $(S, V)$ by $\mf D(S, \mc T) \subset \mf D(S, V)$ and $\mathcal D^\sharp(S, \mathcal T) \subset \mathcal D^\sharp(S, V)$ we denote the subsets of (classes of) metrics realizing $\mathcal T$. It is evident that if $d^\sharp$ realizes $\mc T$, then $d^\sharp$ is determined by $\mathcal T$ and by the edge lengths. Hence, the set $\mathcal D^\sharp(S, \mathcal T)$ can be considered as an open polyhedron in $\R^{E(\mathcal T)}_+$ defined by strict triangle inequalities. The charts $\mathcal D^\sharp(S, \mathcal T)$ endow $\mathcal D^\sharp(S, V)$ with the structure of a smooth manifold of dimension $3(n-k)$ where $n:=|V|$ and $k:=\chi(S)$. It is not hard to check that this topology coincides with the topology induced from $\mf D(S, V)$. 
One can also check that $B_0(S, V)$ acts on $\mc D^\sharp(S, V)$ freely and properly discontinuously. Thus, $\mc D(S, V)$ also gets endowed with the structure of a smooth manifold so that the projection $\mc D^\sharp(S, V) \ra \mc D(S, V)$ is a covering map.

It is known from the works~\cite{Tro} of Troyanov or~\cite{GGLSW} of Gu--Guo--Luo--Sun--Wu that $\mathcal D^\sharp(S, V)$ is diffeomorphic to the product of $\R^n$ with the Teichm\"uller space $\mc T(S, V)$ of $(S, V)$, hence 
\begin{thm}
\label{topold}
The space $\mc D^\sharp(S, V)$ is diffeomorphic to $\R^{3(n-k)}$.
\end{thm}

%

By $\mf D_c(S, V)$ we denote the subset of $\mf D_c(S, V)$ consisting of convex cone-metrics with $V=V(d)$. By $\ol{\mf D}_c(S, V)$ we denote the subset of $\mf D_c(S, V)$ consisting of convex cone-metrics.  It is easy to see that $\mf D_c(S, V)$ is open in $\mf D_c(S, V)$ and $\ol{\mf D}_c(S, V)$ is its closure. The same notation applies to the respective subsets of $\mc D(S, V)$ and $\mc D^\sharp(S, V)$.

By $\S^1_l$ we denote the metric circle of length $l$. A \emph{closed $\e$-quasi-geodesic} is the image of $S^1_l$ under an $\e$-quasi-isometry, i.e., under a map $\S^1_l \rightarrow (S, d)$ with distortion at most $\e$. The following consequence of the Arzel\'a--Ascoli theorem was shown in~\cite[Theorem 6.1]{HR}:

\begin{lm}
\label{geodconv}
Let $\{d_i\}$ be a sequence of metrics on $S$ converging in the strong Lipschitz sense to a metric $d$. Let $\{\psi_i \subset S\}$ be closed $\e_i$-quasi-geodesics in $d_i$ with $\e_i \rightarrow 0$ and the lengths $l_i$ of $\psi_i$ converging to $l>0$. Then, up to extracting a subsequence, $\psi_i$ converge to a closed geodesic $\psi \subset S$ in $d$ of length $l$.
\end{lm}

We recall the following well-known notion

\begin{dfn}
Let $d$ be a metric on $S$. The \emph{systole} $\sys(d)$ of $d$ is the infimum of lengths of non-contractible loops of $S$.
\end{dfn}

Lemma~\ref{geodconv} implies that the infimum is achieved.

\begin{lm}
\label{sys}
The function $\sys(.)$ is continuous over $\mf D(S, V)$.
\end{lm}

\begin{proof}
Let $\{d_i\}$ be a sequence of metrics on $S$ converging in the strong Lipschitz sense to a metric $d$ and $\{\psi_i \subset S\}$ be a sequence of systolic loops. Lemma~\ref{geodconv} shows that, up to taking a subsequence, they converge to a loop $\psi \subset S$ of length $\liminf_{i \ra \infty} \sys(d_i)$. Obviously, since $\psi_i$ are non-contractible, so is $\psi$. Hence
$$\liminf_{i \ra \infty} \sys(d_i) \geq \sys(d).$$
On the other hand, if $\psi'$ is a systolic loop of $d$, then for a sequence $\e_i \ra 0$ the lengths of $\psi'$ in $d_i$ is $\sys(d)+\e_i$. Thus,
$$\limsup_{i \ra \infty} \sys(d_i) \leq \sys(d).$$
\end{proof} 

%

\subsubsection{Balanced cone-metrics}
\label{balsec}

\begin{dfn}
Let $S$ be a closed surface, $V \subset S$ be a finite set and $d$ be a metric on $S$. The \emph{sparsity} $\sp(V, d)$ of $V$ in $d$ is
$$\sp(V, d):=\sup_{p \in S}d(p, V).$$
For a cone-metric $d$ we define its \emph{cone-sparsity} $\cosp(d):=\sp(V(d), d)$.
\end{dfn}

Due to compactness, it is easy to see that there exist $p \in S$, $v \in V$ such that $d(p,v)=\sp(V, d)$. 

\begin{lm}
\label{cosp}
The function $cosp$ is upper semicontinuous over $\mf D(S, V)$.
\end{lm}

\begin{proof}
Let $\{d_i\}$ be a sequence of cone-metrics on $(S, V)$ converging in the strong Lipschitz sense to a metric $d$ and $\{p_i \subset S\}$ be a sequence of points realizing $\cosp(d_i)$. Up to taking a subsequence, they converge to a point $p \in S$. Since for each $v \in V$ the function $\kappa_v$ is continuous over $\mf D(S, V)$, if $v \notin V(d_i)$ for infinitely many $d_i$, then $v \notin V(d)$. Hence, $$d(p, V(d))\geq \limsup_{i \rightarrow \infty} d(p_i, V(d_i)),$$ $$\cosp(d) \geq \limsup_{i \rightarrow \infty}\cosp(d_i).$$
\end{proof}

\begin{dfn}
A cone-metric $d$ on a surface $S$ is called \emph{balanced} if $$\cosp(d)<\delta(\sys(d)),$$ where $\delta$ is the function from Corollary~\ref{pants}.
\end{dfn}

By ${\mf D}_{bc}(S, V)$ we denote the subset of $\mf D(S, V)$ consisting of balanced convex cone-metrics. Lemma~\ref{sys} and Lemma~\ref{cosp} imply

\begin{lm}
The set ${\mf D}_{bc}(S, V)$ is open.
\end{lm}

For two metrics $d_1$ and $d_2$ on $S$ \emph{the Lipschitz distance} between $d_1$ and $d_2$ is defined as the infimum of the distortions of all homeomorphisms $(S, d_1) \ra (S, d_2)$. We say that a sequence of metrics $\{d_i\}$ on a surface $S$ converges to a metric $d$ \emph{in the Lipschitz sense} if the Lipschitz distance between them tends to zero. The obtained Lipschitz topology is one of the most common to endow the space of all metrics on $S$. Now we are going to prove Proposition~\ref{prop4} stating that the set of convex balanced cone-metrics is open and dense in the space of all convex cone-metrics on $S$ equipped with the Lipschitz topology. It is evident that it is dense, so it remains only to show that it is open. It follows from

\begin{lm}
The function $\cosp$ is upper semicontinuous on the space of all cone-metrics on $S$ in the Lipschitz topology.
\end{lm}

\begin{proof}
The proof is almost identical to the proof of Lemma~\ref{cosp}, but we need to establish one additional fact. Let $\{d_i\}$ be a sequence of hyperbolic cone-metrics converging to a hyperbolic cone-metric $d$ in the Lipschitz sense.

\begin{cla}
For each $v \in V(d)$ there exists a sequence $\{v_i \in V(d_i)\}$ converging to $v$.
\end{cla}

Indeed, otherwise there exists $v \in V(d)$, a small neighborhood $U$ of $v$ not containing other point of $V(d)$, a domain $U' \subset \H^2$ and a sequence of homeomorphisms $f_i: U \rightarrow U'$ with distortions converging to zero. Let $\sigma$ be a circle of small radius $r>0$ around $v$ in $U$. Then $f_i(\sigma)$ converges in the Hausdorff sense to a circle of radius $r$ around $f(u)$ in $U'$. Hence, their lengths must converge. On the other hand, their lengths are clearly different.

The rest of the proof is straightforward.
\end{proof}

\begin{dfn}
For a cone-metric $d$ and $V \supseteq V(d)$ we call $V$ \emph{balancing} $d$ if $$\sp(V, d) <\delta(\sys(d))$$ where $\delta$ is the function from Corollary~\ref{pants}.
\end{dfn}

By $\ol{\mf D}_{bc}(S, V)$ we denote the subset of $\mf D_c(S, V)$ consisting of metrics, for which $V$ is balancing. Note that in contrast with $\ol{\mf D}_{c}(S, V)$, which is the closure of ${\mf D}_{c}(S, V)$, $\ol{\mf D}_{bc}(S, V)$ is not the closure of ${\mf D}_{bc}(S, V)$ in ${\mf D}(S, V)$ as we do not add metrics with $\sp(V, d)=\delta(\sys(d))$. In particular, $\ol{\mf D}_{bc}(S, V)$ is open in $\ol{\mf D}_{c}(S, V)$ endowed with the subspace topology. 

Now we note that the functions $\sys(.)$, $\cosp(.)$ and $\sp(.)$ are constant over $H_0(S, V)$-orbits. Thereby, we can define similar subsets of the spaces $\mc D(S, V)$ and $\mc D^\sharp(S, V)$, for which we use the obvious notation.

Balanced convex cone-metrics play a special role in our paper. Of particular importance are the ``weak'' connectivity and simple-connectivity properties of $\mc D_c(S, V)$, $\mc D_c^\sharp(S, V)$ through balanced convex cone-metrics, Lemmas~\ref{connect1} and~\ref{connect2}. To obtain them, first we need the following statement.

\begin{lm}
\label{increase}
Let $d^\sharp \in \mc D^\sharp(S, V)$, $\mc T$ be its Delaunay triangulation and $t\geq 0$ be a number. For each edge $e \in E(\mc T)$ of length $l_e$ in $d^\sharp$ we set a new length $l_{e,t}$ given by
$$\sinh\left(\frac{l_{e,t}}{2}\right):=e^t\sinh\left(\frac{l_e}{2}\right).$$
Then the new edge-lengths define a metric class $d_t^\sharp \in \mc D^\sharp(S, V)$ with the same Delaunay decomposition as $d^\sharp$ and for every $v \in V$ and $t>0$ we have $\kappa_v(d_t^\sharp)<\kappa_v(d^\sharp)$.
\end{lm}

\begin{proof}
Let us first show that $d_t^\sharp$ is well-defined, i.e., that the obtained edge-lengths satisfy the triangle inequalities for each triangle of $\mc T$. Consider the function $$f(x):=2\arcsinh(e^{t/2}\sinh(x/2)).$$ Its derivatives are given by
$$f'(x)=\frac{e^{t/2}\cosh(x/2)}{\sqrt{e^{t}\sinh^2(x/2)+1}},$$
$$f''(x)=\frac{e^{t/2}\sinh(x/2)(1-e^{t})}{2(e^{t}\sinh^2(x/2)+1)^{3/2}}.$$
Hence, for $t > 0$ and $x>0$ it is strictly increasing and concave. Since $f(0)=0$, $f$ is subadditive under these conditions. Hence, if three edge-lengths of $\mc T$ in $d^\sharp$ satisfy the triangle inequalities, the modified lengths satisfy them too.

It was shown in~\cite[Proposition 10]{GGLSW} that a triangulation $\mc T$ is Delaunay for $d^\sharp$ if and only if for every two adjacent triangles with the length of the common edge $l_0$ and with the other edge-lengths $l_1$ and $l_2$ in one triangle and $l'_0$, $l'_1$ in another, we have
$$\frac{\sinh^2(l_1/2)+\sinh^2(l_2/2)-\sinh^2(l_0/2)}{\sinh(l_1/2)\sinh(l_2/2)}+\frac{\sinh^2(l'_1/2)+\sinh^2(l'_2/2)-\sinh^2(l_0/2)}{\sinh(l'_1/2)\sinh(l'_2/2)} \geq 0.$$
This implies that $\mc T$ remains Delaunay in all $d_t^\sharp$. Moreover, Lemma~\ref{ptolem} implies that $d_t^\sharp$ are independent on the choice of a Delaunay triangulation $\mc T$.

It remains to show that for all $t$ and all $v$ we have $\pt \kappa_v(d_t^\sharp)/\pt t<0$. This was basically done in~\cite[Lemma 5.3.iii]{Pro}. It is enough to consider $t=0$. Let $\ol t \in \R^V$ be a vector close to zero. Define a length $l_{e, \ol t}$ for $e \in E(\mc T)$ with endpoints $v$ and $w$ by 
$$\sinh\left(\frac{l_{e,\ol t}}{2}\right):=\exp\left(\frac{\ol t_v+\ol t_w}{2}\right)\sinh\left(\frac{l_e}{2}\right).$$
If $\ol t$ is sufficiently close to zero, then all triangle inequalities still hold for the new lengths and define a cone-metric class on $(S, V)$, which we denote by $d_{\ol t}^\sharp$. For $v, w \in V$ we denote $\pt \kappa_v(d_{\ol t}^\sharp)/\pt t_w$ by $X_{vw}$. Lemma 5.3.iii in~\cite{Pro} states that for each $v$ we have $\sum_{w \in V}X_{vw}<0$, which implies the desired claim. Note that~\cite{Pro} uses a slightly different setting, but the connection with the present context is explained in~\cite[Section 4.3]{Pro}.
\end{proof}

We employ Lemma~\ref{increase} to prove the following fact.

\begin{lm}
\label{homotop}
Let $D \subset \mc D^\sharp(S, V)$ be a compactum. There exists an isotopy $\iota_t: D \hookrightarrow \mc D^\sharp(S, V)$, $t \in [0,1]$, of the inclusion map $\iota_0: D \hookrightarrow \mc D^\sharp(S, V)$ such that $\iota_1(D) \subset \mc D_c^\sharp(S, V)$ and for each $d^\sharp \in D$ the cone-angle $\kappa_v$ of each $v \in V$ strictly decreases along the respective path $\{\iota_t(d^\sharp): t \in [0,1]\}$. In particular, if $D \subset \ol{\mc D}_c^\sharp(S, V)$, then $\iota_t(D) \subset \mc D_c^\sharp(S, V)$ for all $t>0$.
\end{lm}

\begin{proof}
For each $d^\sharp \in D$ choose a Delaunay triangulation $\mathcal T$ and define $\iota_t(d^\sharp)$ by setting for each edge $e \in E(\mc T)$ and for each $t > 0$ 
$$\sinh\left(\frac{l_e(\iota_t(d))}{2}\right):=e^t\sinh\left(\frac{l_e(d)}{2}\right),$$
where $l_e(d^\sharp)$ is the length of $e$ in metric $d^\sharp$.
Lemma~\ref{increase} shows that this indeed defines a class of metrics in $\mc D^\sharp(S, \mc T)$ and that for each $v \in V$ the cone-angle $\kappa_v$ strictly decreases. Lemma~\ref{ptolem} implies that $\iota_t(d^\sharp)$ is independent on the choice of a Delaunay triangulation. This means that $\iota_t$ is a continuous homotopy. Thus, for some $t_0 >0$ we obtain $\iota_{t_0}(D) \subset \mc D_c^\sharp(S, V)$. It only remains to reparametrize the obtained path. If $D \subset \ol{\mc D}_c^\sharp(S, V)$, then $\iota_t(D) \subset \mc D_c^\sharp(S, V)$ for all $t>0$ as the cone-angles strictly decrease.
\end{proof}

Now we can obtain the desired connectivity properties.

\begin{lm}
\label{connect1}
For every $d_0,~d_1 \in \mathcal D_c(S, V)$ there exist $W \supseteq V$ and a continuous path $\alpha: [0, 1] \rightarrow \ol{\mc D}_{bc}(S, W)$ such that $\alpha(0)$, $\alpha(1)$ are lifts of $d_0$, $d_1$, and for all $t\in (0,1)$ we have $\alpha(t) \in \mathcal D_{bc}(S, W)$.
\end{lm}

\begin{proof}
Let $d_0^\sharp, d_1^\sharp$ be some lifts of $d_0, d_1$ to $\mc D^\sharp(S, V)$ and $\alpha': [0,1] \rightarrow \mathcal D^\sharp(S, V)$ be a path connecting them. Lemma~\ref{homotop} shows that it can be transformed to a path $\alpha'': [0,1] \rightarrow \mc D_c^\sharp(S, V)$ connecting $d_0^\sharp$ and $d_1^\sharp$: just do the respective homotopy and use as $\alpha''$ the concatenation of the upper and the horizontal sides of the homotopy square. 
Lift $\alpha''$ to a path $\mathfrak a'': [0,1] \rightarrow \mathfrak D_c(S, V)$. By compactness of the image of $\mf a''$, we can choose a sufficiently dense set $W \subset S$ such that $W$ balances $\mathfrak a''(t)$ for all $t$. This projects to a path $\alpha''': [0,1] \ra \ol{\mc D}_{bc}^\sharp(S, W)$. Applying once again Lemma~\ref{homotop} and projecting to $\mc D(S, W)$, we obtain the desired path $\alpha$ with the property $\alpha(t) \in \mc D_{bc}(S, W)$ for all $t \in (0,1)$.
\end{proof}

We highlight that while the previous lemma is equally true with $\mc D_c(S, V)$ replaced by $\mc D_c^\sharp(S, V)$, the lemma below is true only for $\mc D_c^\sharp(S, V)$, and it is the main reason for us to introduce $\mc D_c^\sharp(S, V)$.

\begin{lm}
\label{connect2}
Let $\beta: S^1 \rightarrow \overline{\mathcal D}_c^\sharp(S, V)$ be a continuous loop. Then there exist $W \supseteq V$ and a lift $\alpha: S^1 \rightarrow \ol{\mc D}_c^\sharp(S, W)$ of $\beta$ such that $\alpha(S^1) \subset \ol{\mc D}_{bc}^\sharp(S, W)$ and $\alpha$ is contractible in $\overline{\mathcal D}_{bc}^\sharp(S, W)$.
\end{lm}

\begin{proof}
By $\ol D^2$ we denote the closed topological 2-disk with the boundary identified with $S^1$. By Theorem~\ref{topold}, $\mc D^\sharp(S, V)$ is contractible. Let $\xi': \ol D^2 \rightarrow \mc D^\sharp(S, V)$ be a map contracting $\beta$ in $\mc D^\sharp(S, V)$. Lemma~\ref{homotop} allows us to obtain from $\xi'$ a map $\xi'': \ol D^2 \rightarrow \ol{\mc D}_c^\sharp(S, V)$ contracting $\beta$. 
Since $\mf D(S, V)$ is a trivial topological bundle over $\mc D^\sharp(S,V)$ (see~\cite[Section 2.2]{Pro2}), we can lift $\xi''$ to a map $\mf x'': \ol D^2 \rightarrow \mf D(S, V)$. 
By compactness of the image of $\mf x''$, we can choose a sufficiently dense set $W \subset S$ such that $W$ balances $\mf x''(z)$ for all $z \in \ol D^2$. This projects to the desired map $\xi: \ol D^2 \ra \overline{\mathcal D}_{bc}^\sharp(S, W)$ such that $\xi|_{S^1}$ is a lift of $\beta$.
\end{proof}

\subsection{Hyperbolic cone-3-manifolds}
\label{cone3sec}

A hyperbolic cone-3-manifold is a 3-dimensional generalization of a hyperbolic cone-surface. We do not require that much theory about the set of cone-metrics on a given 3-manifold as we do about cone-metrics on surfaces, so we present the former theory in a simplified form.

\begin{dfn}
\emph{A hyperbolic cone-3-manifold} is a pair $(K, \Sigma)$, where $K$ is a 3-manifold, possibly with boundary, equipped with an intrinsic metric and $\Sigma$ is an embedded geodesic graph such that 
\begin{itemize}
\item every interior point belonging to $K \backslash \Sigma$ has a neighborhood isometric to a domain in $\H^3$;
\item every boundary point belonging to $K \backslash \Sigma$ has a neighborhood isometric to a domain of a boundary point in a half-space of $\H^3$;
\item every interior point belonging to $\Sigma$ has a neighborhood isometric to a cone of curvature $-1$ over a 2-dimensional spherical cone-surface homeomorphic to the 2-sphere, but not isometric to the standard sphere;
\item every boundary point belonging to $\Sigma$ has a neighborhood isometric to a cone of curvature $-1$ over a 2-dimensional spherical cone-surface with boundary, homeomorphic to the 2-disk, but not isometric to a half of the standard sphere.
\end{itemize}
\end{dfn}

We refer to~\cite{CHK, HK, MM3, Wei} for an introduction to the theory of closed hyperbolic cone-3-manifolds. If a point of $\Sigma$ has a neighborhood isometric to a cone over the sphere with two conical points (in the case of an interior point), or to a spherical lune (in the case of a boundary point), we call it an \emph{edge point}, otherwise we call it a \emph{vertex}. 

The main source of cone-3-manifolds are gluings of hyperbolic polyhedra along isometric 2-dimensional faces. We say that a cone-3-manifold $K$ is \emph{triangulated} if $\mathcal T$ is its geodesic triangulation such that all vertices of $\mathcal T$ belong to $\Sigma$ and $\Sigma$ is included in the 1-skeleton of $\mathcal T$. We note that we do not necessarily require that all vertices of $\mathcal T$ are vertices of $\Sigma$ or that all edges of $\mathcal T$ belong to $\Sigma$. A triangulated hyperbolic cone-3-manifold $(K, \Sigma, \mc T)$ is uniquely determined by the topology of $(K, \mc T)$ and by the edge lengths of $\mc T$. Let $(K, \mc T)$ be a triangulated 3-manifold and $E(\mc T)$ be the set of edges. There is a domain $\mc C(K, \mc T) \subset \R^{E(\mc T)}$ corresponding to the condition that every $l \in \mc C(K, \mc T)$ turns every tetrahedron of $\mc T$ to a non-degenerate hyperbolic tetrahedron. Hence, we can consider $\mc C(K, \mc T)$ as the space of hyperbolic cone-structures on $(K, \mc T)$. Note, however, that the singular locus of such structures is variable, though it stays belonging to the union of the edges of $\mc T$. For $e \in E(\mc T)$ and $l \in \mc C(K, \mc T)$ by $\nu_e$ we denote the total dihedral angle of $e$ in the cone-structure determined by $l$. This is an analytic function over $\mc C(K, \mc T)$. 

Our proof of the infinitesimal rigidity is based on a comparison of infinitesimal variations of $l \in \mc C(K, \mc T)$ with the standard approach to the infinitesimal variations of hyperbolic cone-structures via the bundle of infinitesimal isometries developed in~\cite{HK, Wei, Wei2}. This occupies Section~\ref{infrigsec}.

In what follows we will need triangulations with few special properties

\begin{dfn}
\label{largedfn}
Let $\mc T$ be a triangulation of a hyperbolic cone-3-manifold $(K, \Sigma)$, and $L:=K \backslash \Sigma$. The triangulation $\mc T$ is called \emph{large} if for every tetrahedron $\tilde T$ of $\mc T$ lifted to $\tilde L$ and every $\gamma \in \pi_1(L)$, $\gamma \neq e$, the tetrahedron $\gamma\tilde T$ does not share with $\tilde T$ any 2-face, or any edge coming from $e \in E(\mc T)$ that is not in $\Sigma$. It is called \emph{super-large} if $\gamma \tilde T$ never shares a vertex with $\tilde T$.
\end{dfn}

\subsection{Convex cocompact hyperbolic 3-manifolds}
\label{cocosec}

This section largely follows Section 2.3 in~\cite{Pro3}. Here we repeat the basic definitions and give a brief summary of the basic facts that we will use, but for additional explanations we refer to~\cite[Section 2.3]{Pro3}.


Recall that $M$ is an oriented smooth compact 3-manifold with non-empty boundary such that its interior $N $ admits a cocompact hyperbolic metric, and that $M$ is distinct from the solid torus. Let $g$ be a hyperbolic metric on $M$ such that $\partial M$ is locally convex (which implies that it is globally convex). It is well-known that

\begin{lm}
\label{extens}
There exists a unique up to isotopy convex cocompact metric $\overline g$ on $N $ such that there is a (unique) isometric embedding $(M, g) \hookrightarrow (N , \overline g)$ such that the complement to the image is a neighborhood of $\pt M$ intersected with $N$, homeomorphic to $\pt M \times (0,1)$, and the induced map $\pi_1(\pt M) \ra \pi_1(\pt M)$ is identity.
\end{lm}

For a proof look the beginning of~\cite[Section 2.3]{Pro3}.
Recall that we are going to use the notation $\ol g$ to denote the metric on $N$ obtained from $g$ by Lemma~\ref{extens}, but sometimes we will just write $\ol g$ to denote a convex cocompact metric on $N$ in the absence of metric $g$. 
For a convex cocompact hyperbolic metric $\overline g$ we denote the hyperbolic manifold $(N, \ol g)$ by $N(\ol g)$. We denote the image of $(M, g)$ in $N(\ol g)$ under the embedding from Lemma~\ref{extens} by $M(g) \subset N(\ol g)$. By $\pt_{\infty} N(\ol g)$ we denote the boundary at infinity of $N(\ol g)$, i.e., $\pt M$ equipped with the respective conformal structure.

By $G$ we denote the group $\rm{Iso}^+(\H^3)$ of orientation-preserving isometries of $\H^3$, which we identify with $PSL(2, \C)$. We will frequently assume that some developing map $\tilde N(\ol g) \rightarrow \H^3$ is fixed, so we identify the universal cover $\tilde N(\ol g)$ with $\H^3$ and $\tilde M(g)$ with a convex subset of $\H^3$. By $\rho_{\ol g}: \pi_1(M) \ra G$ we denote the corresponding holonomy representation. By $\Lambda(\rho_{\ol g})=\Lambda(\overline g)$ we denote its limit set and by $\tilde C(\overline g)$ we denote its convex hull. The latter is $\rho_{\overline g}$-invariant and projects to the \emph{convex core} $C(\overline g) \subset N(\ol g)$, an alternative definition of which was given in the introduction as Definition~\ref{cocodef}. 

We recall that $\ol g$ and $\rho_{\ol g}$ are called \emph{Fuchsian} if $\Lambda(\ol g)$ is contained in a circle. In the Fuchsian case $C(\ol g)$ is 2-dimensional, and we will frequently have to say something exceptional about Fuchsian metrics. We need to distinguish the subcase when $\Lambda(\ol g)$ is a circle itself, so $\ol g$ is called \emph{Fuchsian of the first kind}. It is called \emph{Fuchsian of the second kind} in the remaining case when $\Lambda(\ol g)$ is a proper subset of a circle. In the former case $M$ is the oriented interval bundle over a closed surface and $C(\ol g)$ is a totally geodesic surface. In the latter case $M$ is the oriented interval bundle over a compact surface with non-empty boundary, i.e., a handlebody and $C(\ol g)$ is a totally geodesic surface with non-empty geodesic relative boundary. A well-known observation of Thurston is




%
%

\begin{lm}
\label{cocobound}
Let $\ol g$ be non-Fuchsian. The induced intrinsic metric on $\partial C(\overline g)$ is hyperbolic.
\end{lm} 

See, e.g.,~\cite{EM} and~\cite{Smi} for a detailed proof. Thus, a connected component of $\partial C(\overline g)$ is an embedded hyperbolic surface, which is, however, non-geodesic as it may be bent along a geodesic lamination.

By $\mc{CH}(N)$ we denote the space of convex cocompact hyperbolc metrics on $N$ up to isotopy. As before, we will not introduce a special notation for an equivalence class of metrics, and will write $\ol g$ sometimes for a precise metric, and sometimes for an equivalence class. By $\mc T(\pt M)$ we denote the Teichm\"uller space of $\pt M$. The following version of the generalized Ahlfors--Bers theorem holds

\begin{thm}
The map $\mc {AB}: \mc{CH}(N) \ra \mc T(\pt M)$ sending $\ol g$ to the conformal structure at infinity of $\pt_{\infty} N(\ol g)$ is a bijection.
\end{thm}

We did not find an exact reference with this formulation of this result, but in~\cite[Section 2.3]{Pro3} we explain how to derive it from the existing literature.
We endow $\mc {CH}(N)$ with the smooth topology of $\mc T(\pt M)$. In order to introduce in what follows the topology on the space of bent metrics on $M$ we need to recall few other objects from~\cite[Section 2.3]{Pro3}. By $\mc {CH}_h(N)$ we denote the space of convex cocompact hyperbolic metrics on $N$ up to homotopy. By $MCG_h(M)$ we denote the group of isotopy classes of self-diffeomorphisms of $M$ homotopic to identity. We recal that in the case of incompressible boundary Waldhausen~\cite{Wal} showed that $MCG_h(M)$ is trivial, while the case of compressible boundary is considered in the monograph~\cite{MM2} of McCullough and Miller. The group $MCG_h(M)$ acts freely and properly discontinuously on $\mc {CH}(N)$ and $\mc T(\pt M)$. The quotient of the former is exactly $\mc {CH}_h(N)$, the quotient of the latter is denoted by $\mc T_h(\pt M, M)$. The map $\mc{AB}$ commutes with the action of $MCG_h(M)$, hence $\mc {CH}_h(N)$ and $\mc T_h(\pt M, M)$ are diffeomorphic. Note, however, that there is another more standard way to introduce the topology on $\mc{CH}_h(N)$. By $\mc R(\pi_1(M), G)$ we denote the space of representations of $\pi_1(M)$ in $G$ endowed with the topology of a complex algebraic variety coming from the complex algebraic group structure on $G$, and by $\mc X(\pi_1(M), G)$ we denote the quotient of $\mc R(\pi_1(M), G)$ by the conjugation action of $G$. The developing maps allow us to embed $\mc{CH}_h(N)$ into $\mc X(\pi_1(M), G)$. The global topology of $\mc X(\pi_1(M), G)$ is not well-behaved, but results of Marden~\cite{Mar} imply that the image of $\mc{CH}_h(N)$ belongs to the smooth part of $\mc X(\pi_1(M), G)$, so this embedding endows $\mc{CH}_h(N)$ with the topology of a complex manifold. It is another part of the Ahlfors--Bers theorem that these topologies on $\mc{CH}_h(N)$ coincide. See~\cite[Section 2.3]{Pro3} for more explanations.

We will use the following result on the behavior of limit sets under the convergence in $\mc R(\pi_1(M), G)$.

\begin{thm}
\label{hausd}
Let $\{\rho_i: \pi_1(M) \ra G\}$ be a sequence converging in $\mc R(\pi_1(M), G)$ to $\rho: \pi_1(M) \ra G$ such that all $\rho_i$ and $\rho$ are holonomy maps of convex cocompact hyperbolic metrics on $N$. Then $\Lambda(\rho_i)$ converge to $\Lambda(\rho)$ in the Hausdorff sense.
\end{thm}

See Proposition 4.5.3 in~\cite{Mar2} or Theorem E in~\cite{AC}.

\section{Hyperbolic manifolds with boundary of polyhedral type}

\subsection{Polyhedral and bent manifolds}
\label{bentstruct}

%
%
We start to study bent metrics on $M$.
Denote the set of vertices of a bent metric $g$ by $V(g)$. If $p \in \partial M$ belongs to the relative interiors of two geodesic segments with distinct tangents at $p$, then the convexity implies that $p$ has a neighborhood isometric to a neighborhood of a point in the boundary of a half-space in $\H^3$. Then we say that $p$ is \emph{regular} in $g$. A connected component of the set of regular points is called a \emph{face} of $g$. If $p$ is not a vertex and not regular, it is called an \emph{edge-point} in $g$. Let it belong to the relative interior of a geodesic segment $\psi$. There is a neighborhood of $p$ in the relative interior of $\psi$ consisting only of edge-points. A maximal geodesic segment consisting of edge-points is called an \emph{edge} of $g$. It is straightforward that the edges are disjoint and simple. We can also speak about bent surfaces in hyperbolic 3-manifolds and about their vertices, edges and faces.

We will frequently consider $\tilde M(g) \subset \H^3$ and we denote its set of vertices, i.e. the full preimage of $V(g)$ in $\tilde M(g)$, by $\tilde V(g)$. We will use all our notions interchangeably between $M(g)$ and $\tilde M(g)$. We note that in $\pt \tilde M(g)$ every edge extends from both sides either to a point from $\tilde V(g)$, or to a point in $\pt_{\infty} \H^3$. In the latter case it necessarily belongs to $\Lambda(\ol g)$ as otherwise the distance to $\tilde C(\ol g)$ would grow to infinity along this edge, which contradicts to the compactness of $M$. This provides a classification of the edges of $g$: either a lift of an edge has two endpoints in $\tilde V(g)$, or two endpoints in $\Lambda(\ol g)$, or it is mixed. We will show below that the third case is impossible. In the second case note that by definition such an edge also belongs to $\pt C(\ol g)$.

We recall that for any non-Fuchsian convex cocompact metric $\ol g$ on $N$ a homeomorphism between $C(\ol g)$ and $M$ provides an example of a bent metric on $M$ without vertices. We devote the next subsection to recall the basic facts about the boundary structure of $C(\ol g)$.

\subsubsection{Boundary geometry of convex cores}
\label{cocoresec}

Let $\ol g$ be a convex cocompact hyperbolic metric on $N$. Due to Lemma~\ref{cocobound}, the induced intrinsic metric on $\pt C(\ol g)$ is hyperbolic. The union of edges of $\pt C(\ol g)$ constitutes what is called a \emph{geodesic lamination}. 

\begin{dfn}
Let $(S, d)$ be a closed hyperbolic surface. A \emph{geodesic lamination} $\lambda$ is a closed subset of $S$ foliated by simple disjoint geodesics, which we will call \emph{edges} of $\lambda$.
\end{dfn}

We refer to~\cite{Bon2, CB, CEG, PH} for various aspects of the theory of geodesic laminations. We remark that we step away from the conventional terminology in favor of coherence as the edges of a geodesic lamination are commonly known in the literature as \emph{leaves}.

A geodesic lamination is called \emph{minimal} if it does not contain any proper geodesic sub-laminations. A simple closed geodesic provides an example of a minimal geodesic lamination. Any other minimal geodesic lamination is called \emph{irrational}. Every minimal irrational lamination consists of uncountably many edges. A general geodesic lamination is called \emph{irrational} if it contains a minimal irrational sub-lamination, and is called \emph{rational} otherwise. For examples of irrational laminations we refer to \cite[Chapter 1.3]{Bon2} or~\cite{Bon3}.

\begin{thm}[\cite{CEG}, Theorem I.4.2.8]
\label{core0}
A geodesic lamination $\lambda$ on a closed hyperbolic surface consists of finitely many minimal sub-laminations and finitely many open geodesics, isolated in $\lambda$, spiralling onto a minimal sub-lamination from both ends.
\end{thm}

We will call connected components of $S \backslash \lambda$ \emph{faces} of $\lambda$. A metric completion of a face is a hyperbolic surface of finite type with piecewise-geodesic boundary, whose components may be either closed or open geodesics. Every end of an open boundary-geodesic is asymptotic to another such end forming what we call a \emph{spike}. An edge of $\lambda$ is called \emph{boundary} if it is a boundary component of a face. Note that even though an irrational lamination contains uncountably many leaves, only finitely many of them are boundary. A lamination is called \emph{maximal} if the completion of each face is isometric to an ideal triangle.

Identify the universal cover $(\tilde S, d)$ with $\H^2$ and let $\tilde \lambda$ be the full preimage of a minimal irrational lamination $\lambda$. Then every geodesic $\tilde \psi$ of $\tilde \lambda$ is a limit of other geodesics of $\tilde \lambda$. If $\tilde \psi$ is a lift of a boundary edge, then it is a limit from only one side.


\begin{dfn}
Let $\lambda$ be a geodesic lamination on $(S, d)$. A \emph{transverse arc} to $\lambda$ is a simple regular arc transverse to each edge of $\lambda$ with endpoints in faces. A \emph{transverse measure} $\mu$ on $\lambda$ is a locally finite Borel measure $\mu_{\tau}$ on each transverse arc $\tau$ such that 

(1) for a transverse sub-arc $\tau' \subset \tau$ the measure $\mu_{\tau'}$ is the restriction of $\mu_{\tau}$;

(2) the support of $\mu_{\tau}$ is $\tau \cap \lambda$;

(3) the measure is invariant through isotopies of $S$ preserving $\lambda$.

The pair $(\lambda, \mu)$ is called a \emph{measured (geodesic) lamination}.
\end{dfn}

Note that in particular $\mu_{\tau}(\tau)=0$ if and only if $\tau \cap \lambda = \emptyset$. We will denote $\mu_{\tau}(\tau)$ just by $\mu(\tau)$. We will sometimes abuse the notation and denote the measured lamination $(\lambda, \mu)$ just by $\mu$. We can speak about the measure of each geodesic of $\lambda$ meaning the measure of any point belonging to this geodesic on any transverse arc intersecting this geodesic. Not every geodesic lamination can support a transverse measure. A useful criterion is

\begin{lm}[\cite{Mar3}, Proposition 8.3.8]
\label{core1}
Let $\lambda$ be a geodesic lamination supporting a transverse measure and $\tilde \lambda$ be its full preimage in $\H^2$. Every point $q \in \pt \H^2$ can be adjacent to at most two geodesics from $\tilde \lambda$, and if it is adjacent to two, then both of them are lifts of boundary edges of $\lambda$ belonging to a minimal irrational sub-lamination.
\end{lm}

Theorem~\ref{core0} and Lemma~\ref{core1} imply (see a detailed proof also, e.g., in~\cite[Corollary 1.7.3]{PH})

\begin{lm}
The support of a measured lamination consists of finitely many minimal sub-laminations.
\end{lm}

By identifying the universal cover $\tilde S$ with $\H^2$ a measured lamination on $S$ gives rise to a $\pi_1(S)$-invariant measure on the space of lines in $\H^2$. This allows to consider measured laminations from a metric independent point of view. By $\mc{ML}(S)$ we denote the space of measured laminations on $S$ equipped with the topology of *-weak convergence of measures. There is a natural PL-structure on the space $\mc{ML}(S)$ induced by so-called \emph{train tracks} on $S$. See the details in~\cite{Bon2, PH}. A well-known result is

\begin{thm}[\cite{Bon2, PH}]
The space $\mc{ML}(S)$ is PL-homeomorphic to $\R^{-3\chi(S)}$.
\end{thm}


Now we return to our setting when $\ol g$ is a convex cocompact hyperbolic metric on $N$. The set of edges of $\pt C(\ol g)$ clearly forms a geodesic lamination, which we denote $\lambda_{\ol g}$. There is a transverse measure $\mu_{\ol g}$ measuring how $\pt C(\ol g)$ is bent in $N(\overline g)$. A careful construction is given in~\cite{EM}. We sketch now another construction of $\mu_{\ol g}$. 

Recall that oriented planes in $\H^3$ correspond to point in the de Sitter space $\dS^3$. Let $\tau$ be a transverse arc and $\tilde \tau$ be its lift to $\pt \tilde C(\ol g)$. Consider the set of all pairs $(p, \Pi)$ where $p \in \tilde \tau$ and $\Pi$ is a supporting plane to $\tilde C(\ol g)$ at $p$, which we consider oriented outwards $\tilde C(\ol g)$. Denote the projection of this set to the second coordinate by $\tau^* \subset \dS^3$. It is not hard to see that $\tau^*$ is an arc in $\dS^3$, i.e., the image of a closed interval. One can show that $\tau^*$ is rectifiable. Define $\mu_{\ol g}(\tau)$ as the length of $\tau^*$. 

The pair $(\lambda_{\ol g}, \mu_{\ol g})$ is called \emph{the bending lamination} of $\ol g$. It is possible to extend the notion of geodesic laminations to cone-surfaces (a step towards that was taken, e.g., in~\cite{BS}) and the notion of bending lamination to bent surfaces. This would give alternative proofs to the basic facts about bent surfaces in the next section. We, however, chose to take a simpler, more geometric route. We will need the following easy statements.

\begin{lm}
\label{core2}
Let $\psi \subset \lambda_{\ol g}$ be a geodesic belonging to an irrational component, and $\tilde \psi \subset \pt \tilde C(\ol g)$ be its lift. Then there is a unique supporting plane to $\tilde C(\ol g)$ at $\tilde \psi$. 
\end{lm}

\begin{proof}
The geodesic $\psi$ is open, hence, there is an accumulation point of $\psi$ in $\pt C(\ol g)$. Thereby, there exists a transverse arc intersecting $\psi$ infinitely many times. Thus, the measure of $\psi$ is zero. If there are two distinct support planes to $\tilde C(\ol g)$ at $\tilde \psi$, the measure of $\psi$ is at least the exterior dihedral angle between these planes, which is non-zero.
\end{proof}

\begin{lm}
\label{core3}
Let $\tilde \psi_1, \tilde \psi_2 \subset \pt \tilde C(\ol g)$ be two edges of the same component of $\pt \tilde C(\ol g)$ ending in the same point $q \in \Lambda(\ol g)$. Then $\tilde \psi_1$ and $\tilde \psi_2$ belong to a minimal irrational sub-lamination of $\lambda_{\ol g}$, and there exists a face of $\pt \tilde C(\ol g)$ incident to both of them.
\end{lm}

\begin{proof}
Pick $p \in \tilde \psi_1$. We claim that the intrinsic distance in $\pt \tilde C(\ol g)$ from $p$ to $\tilde \psi_2$ tends to zero as we move $p$ towards $q$. Indeed, there exists a finite sequence $\Pi_1, \ldots, \Pi_r$ of supporting planes to $\tilde C(\ol g)$ such that $\Pi_1$ is tangent to $\tilde \psi_1$, $\Pi_r$ is tangent to $\tilde \psi_2$ and every two subsequent planes intersect. This follows from~\cite[Lemma II.1.8.3]{EM}. By the Busemann-Feller lemma, the intrinsic distance from $p$ to $\tilde \psi_2$ in $\pt \tilde C(\ol g)$ is at least the intrinsic distance from $p$ to $\tilde \psi_2$ in the union of pieces of these support planes forming a convex surface $C$. Take the horosphere at $q$ passing through $p$ and consider in $C$ the curve obtained in the intersection. The length of this curve from $p$ to the intersection with $\tilde \psi_2$ tends to zero as we move $p$ towards $q$, which implies our claim.

Develop the universal cover of this component of $\pt \tilde C(\ol g)$ to $\H^2$. (We remark that as the boundary of $M$ could be compressible, the boundary components of $\tilde M$ could be infinitely connected.) Due to our claim, there are lifts of $\tilde \psi_1$ and $\tilde \psi_2$ ending in the same point $q' \in \pt_{\infty} \H^2$. Due to Lemma~\ref{core1}, they belong to a minimal irrational sub-lamination of $\lambda_{\ol g}$ and there are no other edges of our lifted lamination ending in $q'$. Hence, in the universal cover there is a face between our lifts of $\tilde \psi_1$ and $\tilde \psi_2$, which comes from a face of $\pt \tilde C(\ol g)$ adjacent to $\tilde \psi_1$ and $\tilde \psi_2$.
\end{proof}


\subsubsection{Boundary geometry of bent manifolds}
\label{bentsec}

Let $g$ be a bent metric on $M$. We will be frequently considering the closed set $\partial M(g) \cap \partial C(\overline g)$. Let $p \in \partial M(g) \cap \partial C(\overline g)$. As $\partial C(\overline g)$ does not have vertices, $p$ belongs to the relative interior of a segment $\psi \subset \partial C(\overline g)$. It follows from convexity that $\psi \subset \partial M(g)$. Hence, $\partial M(g) \cap \partial C(\overline g)$ is the union of some edges and faces of $\partial C(\overline g)$. Note however that such an edge or a face of $\partial C(\overline g)$ is not necessarily an edge or a face of $\partial M(g)$. The following lemma is rather trivial

\begin{lm}
\label{bas1}
Let $g$ be a bent metric on $M$ with at least one vertex. Then $M(g)$ is the closed convex hull of its vertices in $N(\overline g)$ and all vertices are at positive distance from $C(\overline g)$. Conversely, let $\overline g$ be a convex cocompact hyperbolic metric on $N $ and $V \subset N(\overline g)$ be a finite set. Then $\clconv(V)$ produces a bent metric on $M$ except the case when $\ol g$ is Fuchsian and $V$ belongs to the complete totally geodesic surface containing $C(\ol g)$. 
\end{lm}

Here $\clconv(V)$ is \emph{the closed convex hull} of $V$, i.e., the inclusion-minimal closed totally convex subset of $N(\ol g)$ containing $V$. We will also apply the notion on closed convex hull to subsets of $\H^3$.

\begin{proof}
Consider $\tilde M(g) \subset \H^3$. A lift of an edge of $g$ is a geodesic segment in $\H^3$ that ends either at $\tilde V(g)$ or at $\Lambda(\overline g)$. A lift of a maximal geodesic segment of $M(g)$ passing through a regular point ends either at $\tilde V(g)$ or at $\Lambda(\overline g)$ or at an edge-point. Thus, $\tilde M(g)=\conv(\tilde V(g) \cup \Lambda(\overline g))$. But each point of $\Lambda(\ol g)$ is an accumulation point of the orbit of every vertex. We get $\tilde M(g)=\clconv(\tilde V(g))$ and $M(g)=\clconv(V(g))$.

If $p \in (\partial \tilde M(g) \cap \partial \tilde C(\overline g))$, then it lies in the relative interior of a geodesic segment $\psi \subset \partial \tilde C(\overline g)$. Then $\psi \subset \partial \tilde M(g)$ and $p$ is not a vertex of $g$. Thus, all vertices of $g$ are at positive distance from $C(\overline g)$.

For the last claim, it is easy to see that if $V \subset N(\overline g)$ is a finite point-set, then $\clconv(V)$ is homeomorphic to $M$ except the case when $\overline g$ is Fuchsian and $V$ belongs to the complete totally geodesic surface containing $C(\ol g)$. 
We need to show that the boundary is bent. Denote $\clconv(V)$ by $M(g)$ and consider $\tilde M(g) \subset \H^3$. The Caratheodory theorem (which we can apply to $\H^3 \cup \partial_{\infty} \H^3$ due to the Klein model) says that each $p \in \tilde M(g)$ is in the convex hull of at most four points from $\tilde V(g) \cup \Lambda(\overline g)$. If a point belongs to the convex hull of four of these points, but no less, then it belongs to the interior of the convex hull. If follows that every point in $\partial \tilde M(g)$ belongs to the relative interior of a geodesic segment except possibly points of $\tilde V$. This finishes the proof.
\end{proof}

We will also make use of the following easy observation

\begin{lm}
\label{isolvertex}
For every $v \in V(g)$ there exists a neighborhood of $v$ in $\partial M$ that does not contain edges of $g$ except those that are incident to $v$.
\end{lm}

\begin{proof}
Let $B$ be a small embedded geodesic ball around $v$ in $N(\overline g)$ that does not contain other vertices. If there is a sequence of edges not incident to $v$ that converge to $v$, then their intersections with $B$ are disjoint geodesic segments with the endpoints in $\partial B$. Then they converge to a geodesic segment in $M(g) \cap B$ containing $v$ in its relative interior. Thus, $v$ is not a vertex. 
\end{proof}

We now describe the intrinsic geometry of bent metrics.

\begin{lm}
\label{indmetric}
If $g$ is a bent metric on $M$, then the induced intrinsic metric $d$ on $\partial M$ is a convex hyperbolic cone-metric.
\end{lm}

This relies on the following proposition, which strengthens the observation of Thurston that the induced intrinsic metric on the boundary of a convex core is intrinsically hyperbolic.

\begin{prop}[\cite{Smi}]
\label{graham}
Let $X$ be a closed subset of $\H^3$. If $\conv(X)$ has non-empty interior, then $\partial \conv (X) \backslash X$ is locally isometric to $\H^2$.
\end{prop}

\begin{proof}[Proof of Lemma~\ref{indmetric}]
Pick a point $p \in \partial M \backslash V(g)$ and lift it to $\tilde p \in \partial \tilde M(g) \subset \H^3$. Consider a small open ball $B$ centered at $\tilde p$ that contains no vertices and does not intersect the components of $\pt \tilde M(g)$ except the one containing $\tilde p$. Set $X = \partial B \cap \partial \tilde M(g)$. Either $X$ is a great circle and $\partial \tilde M(g) \cap B$ is a planar disk or $\conv(X)$ has non-empty interior, $\partial \tilde M(g) \cap B$ is a connected component of $\partial \conv (X) \backslash X$ and we can apply Proposition~\ref{graham}.

We proved that outside the vertices $\partial M$ is locally isometric to the hyperbolic plane. Now pick $v \in V(g)$ and take a small simply connected neighborhood $U$ of $v$ in $\partial M$ that does not intersect neither edges nor faces not incident to $v$, which is guaranteed by Lemma~\ref{isolvertex}. Then each point of $U\backslash v$ lies on a segment emanating from $v$. Consider a lift $\tilde v$ to $\H^3$ and push slightly a supporting geodesic plane to $\tilde M(g)$ at $\tilde v$ so its intersection $\alpha$ with $\partial \tilde M$ is contained in the lift $\tilde U \ni \tilde v$ of $U$. Then the neighborhood of $\tilde v$ bounded by $\alpha$ is the geodesic cone over $\alpha$ from $\tilde v$ and $\alpha$ is a closed planar convex curve. Hence, the metric around $\tilde v$ is locally isometric to a convex hyperbolic cone.    
\end{proof}

As discussed in Section~\ref{cocoresec}, the support $\lambda_{\ol g}$ of the bending lamination of $\pt C(\ol g)$ consists of finitely many disjoint simple closed geodesics and of minimal irrational laminations. Let $\lambda$ be a minimal irrational component of $\lambda_{\ol g}$ and $\psi \subset \lambda$ be a geodesic. If $\psi \subset \pt M(g)$, then since $\pt M(g) \cap \pt C(\ol g)$ is closed and the closure of $\psi$ is $\lambda$, we see that $\lambda \subset \pt M(g)$. Another simple observation is


\begin{lm}
\label{face}
Let $\lambda$ be a minimal irrational sub-lamination of $\lambda_{\ol g}$ such that $\lambda \subset \pt M(g)$, $\psi \subset \lambda$ be a boundary edge and $F \subset \pt C(\ol g)$ be a face adjacent to $\psi$. Then $F \subset \pt M(g)$.
\end{lm}

\begin{proof}
Lift all to the universal cover, let $\tilde F$ and $\tilde \psi$ be the respective lifts. By Lemma~\ref{core1}, $\tilde C(\ol g)$ has a unique supporting plane $\Pi$ at $\tilde \psi$, which contains $\tilde F$. We have $\tilde \psi \subset \pt \tilde M(g)$, and since any supporting plane to $\tilde M(g)$ at $\tilde \psi$ is a supporting plane to $\tilde C(\ol g)$, $\Pi$ is a supporting plane to $\tilde M(g)$ and $(\Pi \cap \pt \tilde C(\ol g))\subset (\Pi \cap \pt \tilde M(g))$. It follows that $\tilde F \subset \pt \tilde M(g)$.
\end{proof}

We use it first to prove


\begin{lm}
\label{halfinf}
Let $g$ be a bent metric on $M$. No edge of $\tilde M(g)$ has one end in $\tilde V(g)$ and another in $\Lambda(\overline g)$.
\end{lm}

\begin{proof}
Suppose that $\tilde \chi$ is such an edge ending at $\tilde v \in \tilde V(g)$ and $q \in \Lambda(\ol g)$. Let $\chi$, $v$ be the projections of $\tilde \chi$, $\tilde v$ to $\partial M(g)$. Note that the distance to $\tilde C(\overline g)$ goes to zero along $\tilde \chi$, but it never attains zero. Indeed, if it is zero at $p \in \tilde \chi$ and $p$ is the closest point to $\tilde v$ with this property (in terms of the intrinsic distance of $\tilde \chi$), then $p$ is an edge-point of $\pt \tilde C(\overline g)$ and belongs to the relative interior of an edge $\tilde \psi$ of $\pt \tilde C(\overline g)$ with a different tangent at $p$ than $\tilde \chi$. But then $\tilde \psi \subset \partial \tilde M(g)$, so $p$ is regular in $\partial \tilde M(g)$ and $\tilde \chi$ is not an edge.

Consider the closure $\lambda_0$ of $\chi$ in $\partial M(g)$. Because the distance to $C(\ol g)$ goes to zero along $\chi$, the set $\lambda=\lambda_0 \backslash \chi$ is at distance zero to $C(\overline g)$, thus it is a subset of $\partial M(g) \cap \partial C(\overline g)$. It is closed as the intersection of two closed subsets, $\lambda_0$ and $\partial M(g) \cap \partial C(\overline g)$. The points of $\lambda$ belong to the boundary of the set $\partial M(g) \cap \partial C(\overline g)$ in $\pt M(g)$ and hence $\lambda$ consists of edge-points of $\pt C(\ol g)$. If $p \in \lambda$, then one can see that the whole edge containing $p$ belongs to $\lambda$. Hence $\lambda$ is a sub-lamination of $\lambda_{\overline g}$. 

Cut $\pt M$ along $\lambda$, consider the connected component containing $\chi$ and let $S$ be its metric completion. Due to Lemma~\ref{isolvertex}, the cone-points of $\pt M(g)$ are isolated from $\lambda$. Hence, the local picture at the boundary points of $S$ is the same as the local picture for metric completions of faces of geodesic laminations on hyperbolic surfaces. Similarly to, e.g.,~\cite[Chapter 1.4]{Bon2}, one can see that $S$ is a convex hyperbolic cone-surface of finite type with geodesic boundary, whose components may be either closed or open geodesics corresponding to edges of $\lambda$. Each boundary component has a convex neighborhood not containing cone-points, so isometric to a convex subset of $\H^2$. It is not hard to see that if $\chi$ gets close enough to a closed component, then it either starts spiraling on it (i.e., in the universal cover of a convex neighborhood developed to $\H^2$ a lift of $\chi$ is asymptotically parallel to the respective lift of the boundary component), or $\chi$ intersects itself. If $\chi$ gets close enough to an open boundary component, then either it is asymptotic to this component (i.e., goes to a spike from one side of the component), or it intersects the boundary. The geodesic $\chi$ is simple, disjoint from the boundary of $S$ and does not have accumulation points in the interior of $\chi$, hence it either spirals onto a closed geodesic component or goes to a spike.




In the first case $\lambda$ is a simple closed geodesic $\psi$. There is its lift $\tilde \psi \subset \pt \tilde M(g)$ that is incident to $q$. Indeed, there is a lift $\tilde \psi$ of $\psi$ such that the intrinsic distance in $\pt \tilde M(g)$ from a point of $\tilde \chi$ to $\tilde \psi$ tends to zero along $\tilde \chi$ towards $q$, hence also the extrinsic distance in $\H^3$ tends to zero and such $\tilde \psi$ must end in $q$. Note that $\tilde \psi$ is an isolated edge of $\pt \tilde C(\ol g)$. Then $\tilde \psi$ is the axis of a pure loxodromic isometry $\gamma \in \rho_{\overline g}(\pi_1(M))$. The $\gamma$-orbit of $\tilde v$ is contained in the plane spanned by $\tilde \chi$ and $\tilde \psi$. But then $\tilde \chi$ belongs to the relative interior of the convex hull of this orbit, thus $\tilde \chi$ is not an edge of $\pt \tilde M(g)$.

In the second case there are edges $\psi_1$ and $\psi_2$ of $\lambda$ corresponding to the boundary components of $S$ ending in the same spike of $S$ as $\chi$. As the whole $\lambda$ is contained in $\pt M(g)$, $\psi_1$ and $\psi_2$ are boundary edges of the lamination $\lambda$ in $\pt \tilde C(\ol g)$. There are their lifts $\tilde \psi_1$, $\tilde \psi_2$ such that the intrinsic distances in $\pt \tilde M(g)$ from a point of $\tilde \chi$ to $\tilde \psi_1$ and $\tilde \psi_2$ tend to zero along $\tilde \chi$ towards $q$. Hence so do the extrinsic distances in $\H^3$ and $\tilde \psi_1, \tilde \psi_2$ end in $q$. Lemma~\ref{core2} says that they belong to the boundary of a face $\tilde F$ of $\pt \tilde C(\ol g)$. By Lemma~\ref{face}, $\tilde F$ belongs to $\pt \tilde M(g)$. The segment $\tilde \chi$ must belong to the support plane to $\tilde F$ and lie between $\tilde \psi_1$ and $\tilde \psi_2$. Thus, it passes through interior points of $\tilde F$ and can not be an edge.
\end{proof}

Now we can show that any bent metric is polyhedral around vertices. This follows from 

\begin{crl}
Let $g$ be a bent metric on $M$ and $v \in V(g)$. Then there are finitely many edges incident to $v$.
\end{crl}

\begin{proof}
Consider the universal cover $\tilde M(g)\subset \H^3$. Let $\tilde v$ be a lift of $v$. Suppose that there are infinitely many edges incident to $\tilde v$. We claim that there is an edge $\tilde \chi$ from $\tilde v$ to $\Lambda(\ol g)$. Indeed, this follows from the fact that $\rho_{\overline g}(\pi_1(M))$ acts properly discontinuously and the fact that the set of all edges and vertices is closed in $\partial M$. Thus, we get a contradiction with Lemma~\ref{halfinf}.
\end{proof}

We see from Lemma~\ref{halfinf} that every edge of $\tilde M(g)$ either has both ends at $\tilde V(g)$ or both ends at $\Lambda(\overline g)$, in which case it belongs to $\tilde C(\overline g) \cap \partial \tilde M(g)$. 
In the latter case it is either a simple closed geodesic from $\pt C(\ol g)$, or belongs to an irrational lamination of $\pt C(\ol g)$. A bent metric is polyhedral if and only if there are no edges of the third type.

It remains to understand how an irrational bending lamination in the boundary of a bent metric can coexist with the presence of vertices.

\begin{lm}
\label{bs2}
Let $F$ be a face of a bent metric $g$ on $M$ such that it is adjacent to a vertex $v \in V(g)$ and to an edge $\psi$ belonging to a minimal irrational sub-lamination $\lambda$ of $\lambda_{\overline g}$. Then there exists a simple closed geodesic $\chi \subset ( \ol F \cap \partial C(g))$, where $\ol F$ is the closure of $F$, separating $v$ from $\psi$.
\end{lm} 

In particular, if $g$ has vertices and the bending lamination $\mu_{\ol g}$ has only irrational components, then $g$ is strictly polyhedral.

\begin{proof}
Let $\tilde F \subset \tilde M(g)$ be a lift of $F$ adjacent to a lift $\tilde v$ of $v$ and to a lift $\tilde \psi$ of $\psi$. By $\tilde \lambda$ we denote the full preimage of $\lambda$ in $\pt \tilde M(g)$ and by $\Pi$ we denote the supporting plane to $\tilde M(g)$ containing $\tilde F$. Note that since $\tilde \psi \subset \Pi$, the plane $\Pi$ is also supporting for $\tilde C(\ol g)$. The geodesic $\tilde \psi$ divides $\pt \tilde C(\ol g)$ in two parts and the nearest-point projection to $\tilde C(\ol g)$ maps $\tilde F$ to one of these parts (it is easy to see that $\tilde F$ cannot be mapped on $\tilde\psi$). We claim that there is a face $\tilde F'$ of $\pt \tilde C(\overline g)$ adjacent to $\tilde \psi$ from this side. Indeed, the only other option would be that there is a sequence of geodesics from $\tilde \lambda$ approaching $\tilde \psi$ from this side. As $\lambda$ is minimal, we have $\tilde \lambda \subset \partial \tilde M(g)$, so all but finitely many of these geodesics belong to $\tilde F \subset \Pi$. As $\Pi$ is a supporting plane to $\tilde C(\ol g)$, all these geodesics that are sufficiently close to $\tilde \psi$ can not be edges of $\pt \tilde C(\ol g)$, which is a contradiction.

Since $\lambda$ is irrational, Lemma~\ref{face} implies that $\tilde F'$ belongs to $\pt \tilde M(g)$. Moreover, it belongs to the same half-plane of $\Pi$ with respect to $\tilde \psi$ as $\tilde v$. As the distance from $\tilde v$ to $\tilde C(\overline g)$ is positive, $\tilde v$ is disjoint from the closure of $\tilde F'$. Thus, there must be a geodesic line $\tilde \chi$ that is a boundary component of $\tilde F'$ and that separates $\tilde v$ from $\tilde \psi$. It remains to show that the projection $\chi$ of $\tilde \chi$ is not from a minimal irrational sub-lamination $\lambda'$ of $\lambda_{\ol g}$. Indeed, otherwise $\lambda' \subset \pt M(g)$ and there is a sequence of geodesics of its full preimage $\tilde \lambda'$ approaching $\tilde \chi$ from the other side than $\tilde F'$. By the previous reasoning, all but finitely many of these geodesics belong to $\Pi$, which is a contradiction. Hence, $\chi$ is a simple closed geodesic that separates $v$ and $\psi$ in~$\ol F$.
\end{proof}

\begin{crl}
\label{bs2c}
The connected components of $\partial M(g) \cap \partial C(g)$ are simple closed geodesics, which are edges of $\partial C(g)$, and compact hyperbolic surfaces with geodesic boundary.
\end{crl}

%

Recall now Definition~\ref{scpdef} of strictly and controllably polyhedral metrics on $M$. We can show

\begin{crl}
\label{balreal}
Let $g$ be a bent metric on $M$ such that the induced cone-metric $d$ on $\pt M$ is balanced. Then $g$ is controllably polyhedral.
\end{crl}

\begin{proof}
Suppose the converse, then $\partial M(g) \cap \partial C(\overline g)$ contains a compact hyperbolic surface with geodesic boundary. Pick a pair of pants $P$ in this surface. Since $\cosp(d)<\delta(\sys(d))$, Corollary~\ref{pants} implies that $V(d) \cap P$ must be non-empty, which is a contradiction.
\end{proof}

%


%

Together with Theorem~\ref{sub1}, Corollary~\ref{balreal} implies Proposition~\ref{prop5}. Now we deal with strictly polyhedral metrics. As we can see, this property has strong implications on the boundary structure of $\partial M(g)$.

\begin{crl}
\label{strpolyh}
Let $g$ be a strictly polyhedral metric on $M$. Then each edge is a segment between vertices and the closures of all faces are isometric to compact hyperbolic polygons. Conversely, if the closures of all faces of a bent metric are compact hyperbolic polygons, then the metric is strictly polyhedral.
\end{crl}

\begin{proof}
Let $g$ be a strictly polyhedral hyperbolic metric on $M$. The claim on the edges follows from our description of the edges of a bent metric. The closure of each face is isometric to a hyperbolic surface with convex boundary. As every edge ends at vertices, the boundary of a face is compact. If a face has more then one boundary component, then it contains a homotopically non-trivial  geodesic loop. Lifting this loop to the universal cover we see that it goes to $\Lambda(\ol g)$, hence it belongs to $\pt C(\ol g)$, which is a contradiction. 

In the other direction, suppose the converse. If $\partial M(g) \cap \partial C(\overline g)$ is non-empty, we saw that there is a face of $\pt M(g)$ containing a simple closed geodesic either in its interior or in its boundary. The closure of this face can not be isometric to a compact hyperbolic polygon, which finishes the proof.
\end{proof}

We end this section by sketching a construction of non-polyhedral bent metrics with vertices. 

\begin{exmp}
We assume that $N$ admits a convex cocompact hyperbolic metric $\ol g$ such that the bending lamination $\lambda_{\ol g}$ contains an irrational component, and that there is a face $F$ of $\pt C(\ol g)$ with boundary consisting of simple closed geodesics. (The works~\cite{BO, Lec} of Bonahon--Otal and Lecuire imply that such $\ol g$ exists on most manifolds under our consideration. In particular, it exists if $M$ has an incompressible boundary component of genus at least three.)

Let $\psi \subset \pt \tilde C(\ol g)$ be a boundary edge of a lift of $F$. By $\Pi_{\psi}$ we denote the supporting plane to $\tilde C(\ol g)$ at $\psi$ that makes equal angles with the both extremal supporting planes to $\tilde C(\ol g)$ at $\psi$. We call a supporting plane $\Pi$ to $\tilde C(\ol g)$ \emph{bad} if it contains a boundary edge $\psi$ of a lift $\tilde F$ of $F$ and if the dihedral angle between $\Pi$ and $\tilde F$ is smaller than the dihedral angle between $\Pi_\psi$ and $\tilde F$. Otherwise, we call a supporting plane \emph{good}. By $\tilde A$ we denote the intersection in $\H^3$ of all closed supporting half-spaces to $\tilde C(\ol g)$ bounded by good supporting planes. We note that the set $\tilde A$ is closed and $\rho_{\ol g}$-invariant, hence it projects to a closed set $A \subset N(\ol g)$. We also remark that
\begin{itemize}
\item every face of $\pt \tilde C(\ol g)$ that is not a lift of $F$ belongs to $\pt \tilde A$;
\item $\tilde A \backslash \tilde C(\ol g)$ is non-empty.
\end{itemize}

Pick a point $\tilde v$ in $\tilde A \backslash \tilde C(\ol g)$, consider its $\rho_{\ol g}$-orbit $\tilde V$ and take the closed convex hull of $\tilde V$ in $\H^3$. This determines a bent metric $g$ on $M$ with vertices. We have $\tilde M(g) \subset \tilde A$, hence $M(g) \subset A$ and every face of $C(\ol g)$ except $F$ belongs to $\pt M(g)$. In particular, $\pt M(g)$ contains the irrational component of $\lambda_{\ol g}$, thus $g$ is not polyhedral.
\end{exmp}

\subsubsection{Every convex realization of a convex cone-metric is bent}

In this subsection we prove Theorem~\ref{sub1}. We need some small preparations. In Section~\ref{relsec} we mentioned that Alexandrov gave a complete intrinsic characterization of the induced intrinsic metrics on the boundaries of general convex sets in $\E^3$. He similarly did this in other model spaces. For convex sets in $\H^3$ such metrics are called \emph{CBB($-1$) metrics}. Again, we are not going to give a definition, but refer to~\cite{Ale3, BBI}. Particularly, every convex hyperbolic cone-metric is CBB($-1$). If $S$ is a surface and $d$ is a CBB($-1$) metric, one can introduce two natural Borel measures: \emph{the area} $\area(.)$ and \emph{the intrinsic curvature} $\curv_I(.)$. The former is the standard Hausdorff measure, while the latter is a non-smooth analogue of the Gaussian curvature measure. We define \emph{the extrinsic curvature measure} $\curv(.):=\curv_I(.)+\area(.)$. If $d$ is a convex hyperbolic cone-metric and $U \subset S$ is a Borel set, then $\curv(U)$ is equal to the sum of the curvatures of all cone-point belonging to $U$. The curvature of a cone-point $v$ is equal to $2\pi$ minus the total angle of $v$.

We use the duality between $\H^3$ and $\dS^3$. If $C \subset \H^3$ is a closed convex set, one can define its \emph{dual} $C^* \subset \dS^3$ as the set of all planes in $\H^3$ that do not separate $C$ and are oriented outwards. If $C$ has a non-empty interior and $U \subset \pt C$ is a Borel set, its dual $U^* \subset \pt C^*$ is defined as the set of all support planes to $G$ at points of $U$. An analogue of the Gauss theorem in this setting states that $\curv(U)$ is equal to the area of $U^*$. See~\cite[Lemma 2.2.5]{Pro3} for a proof. It is easy to deduce from this

\begin{lm}[\cite{Pro2}, Lemma 3.2.2]
\label{nonflat}
Let $C \subset \H^3$ be a compact convex set with non-empty interior and $\Pi$ be a supporting plane such that the set $R:=\pt C \cap \Pi$ has non-empty relative interior in $\pt C$. Define $R^c:=\pt C \backslash R$. Then $\curv(R^c)>0$.
\end{lm}

Now we can prove Theorem~\ref{sub1}.

\begin{proof}[Proof of Theorem~\ref{sub1}]
Consider $\tilde M(g)$ embedded in $\H^3$ in the Klein model. By $\tilde V$ we denote the full preimage of $V(d)$. By $B$ we denote $\tilde M(g)\cup \Lambda(\ol g)$, which is a compact convex set in $\R^3$ with non-empty interior. We recall few standard notions from convex geometry. A point $p \in \pt B$ is called \emph{extreme} if it does not belong to the relative interior of a segment with endpoints in $\pt B$. A point $p \in \pt B$ is called \emph{exposed} if there exists a plane $\Pi \subset \R^3$ such that $\Pi \cap B=\{p\}$. Obviously, exposed points are extreme, but not necessarily vice versa. All points from $\Lambda(\ol g)$ are exposed. It follows directly from the definition that $B$ is the convex hull of all its extreme points.

The Straszewicz theorem~\cite{Str} (see also~\cite[Theorem 1.4.7]{Schn}) says that the set of extreme points is contained in the closure of the set of exposed points. We are going to show that if $p \in \pt B$ is exposed, then $p \in \tilde V \cup \Lambda(\ol g)$. Since the latter set is closed, this also is true for every extreme point. In turn, this implies that every $p \in \pt \tilde M(g)$ such that $p \notin \tilde V$ belongs to the relative interior of a segment. Due to the equivariance of $\tilde M(g)$, this proves that $g$ is bent.

Suppose the converse that there is $p \in \pt \tilde M(g)$, $p \notin \tilde V$ such that $p$ is exposed. Then $p$ has a neighborhood $U$ in $\pt B$ distinct from $\tilde V \cup \Lambda(\ol g)$. Let $\Pi$ be a plane such that $\Pi \cap B=\{p\}$. Push $\Pi$ slightly towards $B$ such that the curve $\Pi \cap \pt B$ stays in $U$, by $U'$ we denote the open set bounded by this curve. By the choice of $U$ we have $\curv(U)=0$, thus also $\curv(U')=0$. On the other hand, Lemma~\ref{nonflat} implies $\curv(U')>0$, which is a contradiction.
\end{proof}

\subsection{Triangulations}

For our further investigations we need a geodesic triangulation of $M(g)$, where $g$ is polyhedral, with vertices at some subset of $\partial M(g)$ containing $V(g)$ (the examples of non-strictly polyhedral metrics show that just $V(g)$ might be not enough). Unfortunately, we can not prove the existence of such a triangulation in full generality, but we can show two weaker things, which are enough for our purposes: (1) there exists such a decomposition of $M(g)$ into convex hyperbolic polyhedra; (2) in a finite cover of $M(g)$ this decomposition can be triangulated. This situation is very similar to the treatment of hyperbolic cusp-3-manifolds, where the existence of an ideal geodesic triangulation would also be desirable for many applications. In that case the analogue of (1) is the well-known Epstein--Penner decomposition~\cite{EP}, and the analogue of (2) is shown in the work~\cite{LST} of Luo--Schleimer--Tillmann. A similar treatment was considered by Schlenker in~\cite{Sch3} for polyhedral hyperbolic 3-manifolds with ideal vertices. We will follow the mentioned works and will show that their approaches are also applicable in our situation. Note that the construction of the Epstein--Penner decomposition is itself very similar to a well-known convex-hull construction of the Delaunay decomposition of a point set in the Euclidean space via lifting the set to a surface in one dimension higher. In~\cite{DeB} DeBlois considers generalizations of this construction to discrete subsets of $\H^n$ invariant with respect to torsion-free lattices of finite covolume. 

In the above-mentioned situations the first step is to prove that the visible part of the convex hull in the Minkowski space has only space-like supporting planes. From this and from the discretness of the vertex set it follows that the convex hull provides a locally finite polyhedral decomposition. In our setting we can have also light-like and time-like supporting planes, so these conclusions become slightly more involved to make. Before we start, let us prove an auxiliary result.

\begin{lm}
\label{extremedual}
Let $g$ be a polyhedral metric on $M$, $s \in \Lambda(\ol g)$ and $\Pi$ be a plane passing through $s$, but not intersecting $\tilde M(g)$. Then there exists a supporting plane $\Pi_0$ to $\tilde M(g)$ passing through $s$ and asymptotically parallel to $\Pi$. (By saying that a plane is supporting to $\tilde M(g)$ we particularly mean that it contains points of $\pt\tilde M(g)$.)
\end{lm}

\begin{proof}
Consider the dual convex set $\tilde M(g)^*\subset \dS^3$. Its boundary points can be divided in two types: (1) dual to supporting planes to $\tilde M(g)$ and (2) dual to planes $\Pi$ such that $\Pi \cap \tilde M(g)=\emptyset$, but $\pt_\infty \Pi \cap \Lambda(\ol g)\neq \emptyset$. 
%
We use the notions of exposed and extreme points from the previous section. We make a convention that $\pt \tilde M(g)^*$ means $\pt(\tilde M(g)^*)$. (Note that it is not the same as $(\pt \tilde M(g))^*$. The latter is exactly the set of points of type (1).)

Let us show that points of type (2) cannot be exposed for $\tilde M(g)^*$. Suppose that $\Pi^* \in \pt \tilde M(g)^*$ is such a point, dual to a plane $\Pi \subset \H^3$. Note that $\pt_\infty \Pi \cap \Lambda(\ol g)$ consists of a single point $s$. Indeed, otherwise $\Pi$ contains a line between two points of $\Lambda(\ol g)$, hence $\Pi \cap \tilde M(g) \neq \emptyset$ and $\Pi^*$ is not of type (2). Assume that there is a plane $p^* \subset \dS^3$ such that $p^* \cap \tilde M(g)^*=\Pi^*$. Note that $p^*$ cannot be time-like, thus the dual point $p$ is in $\H^3 \cup \pt_\infty\H^3$. If there is another plane in $\H^3$ passing through $p$ and not intersecting $\tilde M(g)$, then $p^*$ contains another point of $\tilde M(g)^*$, which contradicts to our choice. If $p \neq s$, then because $\tilde M(g) \cup \Lambda(\ol g)$ is closed in $\H^3 \cup \pt_\infty \H^3$ and $(\Pi \cup \pt_\infty \Pi) \cap (\tilde M(g)\cup \Lambda(\ol g))=s$, we can rotate $\Pi$ slightly around an axis in $\Pi$ through $p$, but not through $s$, so that the obtained plane does not intersect $\tilde M(g)\cup \Lambda(\ol g)$. If $p=s$, then we can act on $\Pi$ by a parabolic isometry preserving $s$ and sending $\Pi$ to an asymptotically parallel plane in the direction opposite from $\tilde M(g)$. The resulting plane again does not intersect $\tilde M(g)$. Hence, $\Pi^*$ cannot be exposed.

One can see that when $g$ is polyhedral, a point $\Pi^* \in \pt \tilde M(g)^*$ of type (1) is exposed if and only if the dual plane $\Pi$ contains a regular point of $\pt \tilde M(g)$. The set of such points $\Pi^*$ is discrete. The Straszewicz theorem applied to $\tilde M(g)^*\cup\Lambda(\ol g)$ says that the extreme points of $\tilde M(g)^*\cup\Lambda(\ol g)$ are in the closure of the set of exposed points. Then all exposed points of $\tilde M(g)^*$ are extreme. In particular, no point of type (2) is extreme.

Now consider the situation in the statement of the lemma. The point $\Pi^*$ is of the type (2). The segment from $s$ to $\Pi^*$ belongs to $\pt \tilde M(g)^*$. Let us extend it behind $\Pi^*$. It can not be extended infinitely, thus it ends at a point $\Pi_0^* \in \pt \tilde M(g)^*$. It remains to show that $\Pi_0^*$ is of type (1). Indeed, otherwise it is not extreme and belongs to the relative interior of a segment $I$ in $\pt\tilde M(g)^*$ such that each endpoint of $I$ either is in $\Lambda(\ol g)$, or is an extreme point of $\pt \tilde M(g)^*$. If one of the endpoints is $s_0 \in \Lambda(\ol g)$, then by construction $s_0 \neq s$, hence $\pt_\infty \Pi_0 \cap \Lambda(\ol g)$ contains $s$ and $s_0$, so $\Pi^*_0$ is of type (1). If both endpoints are extreme, then they are of type (1). We will show that so is $\Pi^*_0$. 

Denote the supporting planes to $\tilde M(g)$, dual to the endpoints of $I$, by $\Pi_1$ and $\Pi_2$. They must be distinct. Note that because they pass through $s$ and are supporting to $\tilde M(g)$, from Lemma~\ref{halfinf} they contain points of $\tilde C(\ol g)$, hence are supporting to it. Suppose that one of them, say $\Pi_1$, supports a face $\tilde F$ of $\tilde C(\ol g)$ such that there is no edge of this face ending in $s$. Then $s$ is a limit point for points of $\Lambda(\ol g) \cap \pt_\infty \Pi_1$ from both sides. Thus, every geodesic ray emanating from $s$ in $\Pi_1$ contains points of $\tilde F$. Then $\Pi_1$ is the only supporting plane to $\tilde C(\ol g)$ through $s$, which is a contradiction with that $\Pi_2$ is distinct from $\Pi_1$. Hence, there is an edge of $\pt\tilde C(\ol g)$ in $\Pi_1$ adjacent to $s$. Similarly, there is an edge adjacent to $s$ in $\Pi_2$. If these edges are distinct, then by Lemma~\ref{core3} they belong to an irrational component of $\lambda_{\ol g}$, which is contained in $\pt \tilde M(g)$. This contradicts to $g$ being polyhedral. Thus, it is the same edge, i.e., $\Pi_1$ and $\Pi_2$ intersects and the intersection line between $\Pi_1$ and $\Pi_2$ is an edge of $\pt \tilde M(g)$. Hence, $I$ consists of points of type (1), which finishes the proof.

\end{proof}

\begin{lm}
\label{triang1}
Let $g$ be a polyhedral hyperbolic metric on $M$. 
Let $W \subset \partial M(g)$ be a finite set of points containing $V(g)$ and containing at least one point in each edge that is a closed geodesic. Then there exists a decomposition of $M(g)$ into finitely many convex hyperbolic polyhedra with vertices in $W$.
\end{lm}

Here by a decomposition we mean that every point in $M(g)$ belongs to at least one polyhedron, the interiors of the polyhedra are disjoint, and if the intersection of two polyhedra is non-empty, then this intersection is a face of both polyhedra. In this section by a face of a polyhedron we sometimes mean a face of any dimension.

\begin{proof}
Consider the universal cover $\tilde M(g) \subset \H^3$ in the hyperboloid model, let $\tilde W$ be the full preimage of $W$. It is easy to see that after we added to $W$ a point on each edge that is a closed geodesic, $\tilde M(g)$ is the convex hull of $\tilde W$ in the usual sense, without taking the closure. 
So each point of $\tilde M(g)$ is a convex combination of finitely many points of $\tilde W$.

Let $C'$ be the convex hull of $\tilde W$ in $\R^{3,1}$ and $C$ be the set of points $p$ of $\partial C'$ visible from the origin $o$, i.e., such that the segment $op$ in $\R^{3,1}$ does not intersect $C'$ except $p$. Take $p_0 \in \tilde M(g)$. As it is a convex combination of finitely many points of $\tilde W$, the ray $op_0$ intersects $C'$. By definition, the first intersection point $p$ belongs to $C$. One can see that $C$ is homeomorphic to $\tilde M(g)$ via the central projection from the origin, and is closed in $\R^{3,1}$.

Consider a supporting plane $\Pi$ to $C$ that does not pass though the origin, let $H:=\Pi \cap \H^3$. Depending if $\Pi$ is space-like, light-like or time-like, $H$ is a sphere, a horosphere or a hypersphere respectively. We prove that in the latter two cases 

\begin{cla}
\label{noinfin}
$$\partial_{\infty} H \cap \Lambda(\overline g) = \emptyset.$$
\end{cla}

\vskip+0.2cm
Note that as $\Pi \cap C \neq \emptyset$, we have $\Pi \cap \tilde W \neq \emptyset$. The set $H$ bounds an open convex set $H_+$, which does not contain points of $\tilde W$ since $\Pi$ is a supporting plane for~$C$.

In order to prove Claim~\ref{noinfin} we need to introduce a special notion. A point $s \in \partial_{\infty}\H^3$ is called a \emph{conical limit point} for a Kleinian group $\Gamma$ if for every geodesic line $\psi$ through $s$ and every $p \in \H^3$ there exists a sequence $\gamma_i \in \Gamma$ such that $\gamma_i(p)$ converges to $s$ and stays at a bounded distance from $\psi$. A theorem of Beardon--Maskit~\cite{BM} says that if $\Gamma$ is convex cocompact, then every limit point of $\Gamma$ is conical. See also a discussion in~\cite[Chapter 2]{Nic}, particularly Theorem 2.7.2 there.

Suppose that $H$ is a horoball and that there is a point $s \in\partial_{\infty} H \cap \Lambda(\overline g)$. We have a point $s \in \Lambda(\overline g)$ with an open horoball $H_+$ centered at it, which does not contain points of $\tilde W$. In particular it does not contain points from the $\rho_{\overline g}$-orbit of a point in $\H^3$. This contradicts to the claim that $s$ is a conical limit point for $\rho_{\ol g}(\pi_1(M))$.

Suppose that $H$ is a hypersphere and that there is a point $s \in\partial_{\infty} H \cap \Lambda(\overline g)$. The set $H$ is equidistant from a plane $H_0 \subset H_+$ with $\partial_{\infty} H=\partial_{\infty} H_0$. As $H_+$ does not contain points of $\tilde W$, $H_0$ is disjoint from $\tilde M(g)$. We remark that $\partial_{\infty} H \cap \Lambda(\overline g)=s$. Indeed, if there were more than one point, then $H_0$ would contain a line from $\tilde C(\overline g) \subset \tilde M(g)$, which is a contradiction. 

By Lemma~\ref{extremedual}, there exists a a supporting plane $\Pi_0$ to $\tilde M(g)$ passing through $s$ and asymptotically parallel to $\Pi$. Because of the conditions on $\Pi_0$, the set $\Pi_0 \cap \tilde M(g)$ is either an edge, or a face of $\pt \tilde M(g)$.
In the first case denote the edge by $\tilde \psi$. Since $g$ is polyhedral, $\tilde \psi$ is the axis of a purely hyperbolic isometry $\gamma \in \rho_{\ol g}(\pi_1(M))$. Either $\gamma$ or $\gamma^{-1}$ sends $H$ to $H_+$. This contradicts with $H_+ \cap \tilde W=\emptyset$. 

Otherwise denote the face $\Pi_0 \cap \tilde M(g)$ by $\tilde F$ and remark that $H \cap \Pi_0$ is a horocycle based at $s$. Indeed, there exists a unique one-parametric parabolic subgroup of ${\rm Iso}^+(\H^3)$ fixing $s$ and fixing $\Pi$ and $\Pi_0$ as sets. Thus, it also fixes as sets $H$ and $H \cap \Pi_0$. Hence, the orbit of every point from $H \cap \Pi_0$ is the whole $H \cap \Pi_0$. On the other hand, this orbit must be a horocycle. 

Now consider a subgroup $\Gamma$ of $\rho_{\ol g}(\pi_1(M))$ consisting of all elements fixing $\Pi_0$. This is a convex cocompact Kleinian group, and $s$ must be a conical limit point for $\Gamma$. Thus, the $\Gamma$-orbit of any point of $\tilde W \cap \tilde F$ must enter $H_+ \cap \tilde \Pi_0$. This again contradicts with $H_+ \cap \tilde W=\emptyset$ and finishes the proof of Claim~\ref{noinfin}.
\vskip+0.2cm

Now we claim that $\Pi \cap \tilde W$ is finite. If $\Pi$ is space-like, then $H:=\Pi \cap \H^3$ is compact, so the claim follows from the discreteness of $\tilde W$ in $\H^3$. If $\Pi$ is time-like or light-like and $\Pi \cap \tilde W$ is infinite, then the discreteness of $\tilde W$ implies that there is $s \in\partial_{\infty} H \cap \Lambda(\overline g)$, which is a contradiction with Claim~\ref{noinfin}. 

Let $p$ be in the relative interior of $C$ and $\Pi$ be a support plane at $p$. The point $p$ is a convex combination of at least two points of $\tilde W$, which are contained in $\Pi$. From the previous argument, $\Pi \cap \tilde W$ is finite, so $\Pi \cap C$ is the convex hull of finitely many points of $\tilde W$. We claim that there exists a supporting plane through $p$ containing at least four points of $\tilde W$. Indeed, if it is less than four and it is not possible to rotate $\Pi$ to increase $\Pi \cap \tilde W$, then there exists a sequence $\{\tilde w_i\} \subset \tilde W$ such that the distance from $\tilde w_i$ to $\Pi$ tends to zero. Due to the discreteness of $\tilde W$, up to taking a subsequence we assume that $\{\tilde w_i\}$ converges to $s \in \Lambda(\overline g)$. It follows that $\Pi$ is not space-like and $s \in \partial_{\infty} \Pi$, which is a contradiction. 


We see that $C$ decomposes into polyhedral 3-dimensional faces. We want to show that this decomposition is locally finite. It is enough to show that for each point $\tilde w \in \tilde W$ there are only finitely many segments from $\tilde w$ to other points of $\tilde W$ belonging to $C$. Indeed, otherwise from the discreteness of $\tilde W$ and the closeness of $C$ we obtain that there exists a light-like segment from $\tilde w$ belonging to $C$ pointing to $\Lambda(\overline g)$. A supporting plane $\Pi$ containing this segment is light-like and for $H=\Pi \cap \H^3$ we have $\partial_{\infty} H \cap \Lambda(\overline g) \neq \emptyset$, which is a contradiction with Claim~\ref{noinfin}.

By the central projection from the origin, this decomposition projects to a locally finite decomposition of $\tilde M(g)$ into convex hyperbolic polyhedra with vertices in $\tilde W$. As $C$ is invariant with respect to $\rho_{\overline g}$, so is the decomposition, and it produces a decomposition of $M(g)$ into convex polyhedra with vertices in $W$ as desired.

\end{proof}

\begin{lm}
\label{triang2}
Let $g$ be a polyhedral hyperbolic metric on $M$ and let $M(g)$ be decomposed into finitely many convex hyperbolic polyhedra with vertices in a set $W$ such that $V(g) \subseteq W \subset \partial M(g)$. Then there exists a finite cover $M'(g')$ of $M(g)$, where $g'$ is the lift of the metric $g$, such that the lift of the decomposition to $M'(g')$ can be subdivided to a triangulation without adding new vertices, and the obtained triangulation is super-large.
\end{lm}

\begin{proof}
\begin{cla}
\label{large}
There exists a finite cover $M'(g')$ of $M(g)$ such that the polyhedral decomposition lifted to $M'(g')$ is super-large. 
\end{cla}
Here the notion of super-largeness from Definition~\ref{largedfn} is applied to polyhedral decompositions in a straightforward way.
\vskip+0.2cm
Indeed, let $P_1, \ldots, P_r$ be an enumeration of all polyhedra of the decomposition of $M$. Choose a lift $\tilde P_i$ for each $P_i$. Consider the set $\Gamma$ of all $\gamma \in \pi_1(M)$ such that for some vertex $\tilde w$ of some $\tilde P_i$, the vertex $\gamma\tilde w$ is a vertex of the same $\tilde P_i$. As the stabilizer of each $\tilde w$ is trivial, the set $\Gamma$ is finite. As $\pi_1(M)$ is linear, it is residually finite~\cite[Theorem III.7.11]{LS}, and there exists a normal subgroup $S \triangleleft \pi_1(M)$ of finite index not containing any element of $\Gamma$. 

Consider the cover $M'$ of $M$ corresponding to $S$. Suppose that it does not satisfy our claim. Then there exists $\tilde P'$ of $\tilde M(g)$, a vertex $\tilde w'$ of $\tilde P'$ and an element $\gamma' \in \pi_1(M')$ such that $\gamma'\tilde w'$ is a vertex of $\tilde P'$. Note that $\tilde P'=\gamma \tilde P_i$ for some $\gamma \in \pi_1(M)$. Then $\tilde w:=\gamma^{-1}\tilde w'$ is a vertex of $\tilde P_i$. Hence, $\gamma^{-1}\gamma'\gamma \tilde w$ is again a vertex of $\tilde P_i$. Thus, $\gamma^{-1}\gamma'\gamma \in \Gamma$. On the other hand, as $S$ is normal, $\gamma^{-1}\gamma'\gamma \in S$, which contradicts our construction of $S$. This finishes the proof of the claim.
\vskip+0.2cm

Now we consider the cover $M'$ of $M$ as from Claim~\ref{large}. The lift of the polyhedral decomposition to $M'$ is super-large and it is well-known that any such decomposition is triangulable. We give a quick proof of this fact. 

Let us construct a correspondence $f: \mathcal C \rightarrow W$ from the set of polyhedra $\mathcal C$ to the set of vertices with the following properties. \\
(1) For every $P$ in $\mathcal C$, $f(P)$ is a vertex of $P$.  \\
(2) If $P_1$ and $P_2$ are adjacent by a 2-face and both $f(P_1)$, $f(P_2)$ belongs to this 2-face, then $f(P_1)=f(P_2)$. 

To construct this correspondence, start from $P_1$ and let $f(P_1)$ be any its vertex $w$. Next, for each $P_i$ adjacent to $w$ set $f(P_i)=w$. Now at each next step take $P_j$ for which $f(P_j)$ is still not defined. Take any its vertex $w$ that is not in the image of $f$ and set $f(P_i)=w$ for each polyhedron $P_i$ adjacent to $w$. Note that for some of them $f(P_i)$ could have been already defined, but we redefine it now. After finitely many steps we obtain the map $f$ defined over all $\mathcal C$ satisfying the required properties. 

Now for each $P \in \mathcal C$ consider all 2-faces of $P$ adjacent to $f(P)$ and triangulate them by coning from the vertex $f(P)$. Clearly, there are no contradictions due to the condition (2). Next, triangulate all remaining 2-faces arbitrarily. For each $P$ the obtained triangulation of $\partial P$ can be easily extended to a triangulation of $P$ by coning it from $f(P)$. This provides the desired triangulation $\mathcal T$.
\end{proof}

\subsection{Topology of the space of bent metrics}

Let $V $ be a finite set. 
By $\mathcal{CH}(N, V)$ denote the space of equivalence classes of pairs $(\overline g, f)$, where $\overline g$ is a convex cocompact hyperbolic metric on $N$, ${f: V \hookrightarrow N}$ is an injective map, called \emph{marking}, and two pairs are equivalent if they differ by an isometry preserving the marking and isotopic to identity. Here we do not assume that the isotopy should fix the marked points.
As in Section~\ref{cocosec}, we similarly define the space $\mathcal{CH}_h(N, V)$ with isotopy replaced by homotopy. We are going to endow $\mathcal{CH}(N, V)$ with a natural smooth structure such that the forgetful map $\mathcal{CH}(N, V) \rightarrow \mathcal{CH}(N)$ is a smooth bundle. This is very similar to the treatment of the space of metrics with marked oriented planes performed by us in~\cite[Section 2.5]{Pro3}.

Again, it is easier to define a smooth topology on $\mathcal{CH}_h(N, V)$. By $(\H^3)^{V}$ we denote the space of maps $\tilde f:V \rightarrow \H^3$. Consider the direct product $$\mathcal R(\pi_1(M), G) \times (\H^3)^{V}.$$ Let $G$ act on it from the left, by conjugation on the first factor and by isometries of $\H^3$ on the second factor. Additionally, $\pi_1(M)$ acts from the left on each fiber by $\gamma(\rho, \tilde f)=(\rho, \rho(\gamma)\tilde f)$, where $(\rho, \tilde f) \in \mathcal R(\pi_1(M), G) \times (\H^3)^{V}$. One can check that these actions commute. By taking the quotient over both actions, we obtain a fibration over $\mathcal X(\pi_1(M), G)$. We denote the total space of the fibration by $\mc X(\pi_1(M), G, V)$. With the help of developing and holonomy maps the set $\mathcal{CH}_h(N, V)$ injects in $\mc X(\pi_1(M), G, V)$ as an open smooth subset so that the forgetful map $\mathcal{CH}_h(N, V) \rightarrow \mathcal{CH}_h(N)$ is a smooth bundle with fibers homeomorphic to the space of injective maps $f: V \hookrightarrow N$. The group $MCG_h(M)$ acts on $\mathcal{CH}(N)$ freely and properly discontinuously determining a covering map $\mathcal{CH}(N) \rightarrow \mathcal{CH}_h(N).$ We have the commutative diagram
\begin{center}
\begin{tikzcd}
\mathcal{CH}(N, V) \arrow[r] \arrow[d]
& \mathcal{CH}_h(N, V) \arrow[d] \\
\mathcal{CH}(N) \arrow[r]
& \mathcal{CH}_h(N)
\end{tikzcd}
\end{center}
from which we can locally pull back the topology on $\mathcal{CH}_h(N, V)$ to the topology on $\mathcal{CH}(N, V)$ preserving the bundle structure. This endows $\mathcal{CH}(N, V)$ with the topology of a smooth manifold of dimension $3(n-k)$ where $k$ is the Euler characteristic of $\partial M$.


Assume that $V \subset \pt M$. By $\mathcal P_b(M, V)$ we denote the space of bent hyperbolic metrics $g$ on $M$ with $V(g)=V$ considered up to isometries preserving $V$ pointwise and isotopic to identity. We will frequently abuse the notation and write $g \in \mathcal P_b(M, V)$ meaning that $g$ is a bent hyperbolic metric on $M$ representing a class in $\mathcal P_b(M, V)$.
A class in $\mathcal P_b(M, V)$ determines a class in $\mathcal{CH}(N, V)$ and this defines an embedding of $\mathcal P_b(M, V)$ to $\mathcal{CH}(N, V)$, so we will consider $\mathcal P_b(M, V)$ as a subset of $\mathcal{CH}(N, V)$ and endow it with the induced topology. 


\begin{lm}
$\mathcal P_b(M, V)$ is open in $\mathcal{CH}(N, V)$.
\end{lm}

\begin{proof}
A pair $(\rho, \tilde f) \in \mathcal R(\pi_1(M), G) \times (\H^3)^{V}$ projects to an element of $\mathcal P_b(M, V)$ if $\rho$ is a holonomy of a convex cocompact hyperbolic metric on $N$ and if $\Lambda(\rho)$ and $\tilde V$, where $\tilde V$ is the $\rho$-orbit of $\tilde f(V)$, are in a convex position. Take $g \in \mathcal P_b(M, V)$ and lift it to $(\rho, \tilde f) \in \mathcal R(\pi_1(M), G) \times (\H^3)^{V}$. Lemma~\ref{hausd} says that when convex cocompact representations converge in $\mathcal R(\pi_1(M), G)$, their limit sets converge in the Hausdorff sense. Hence, if $g$ does not have a neighborhood in $\mathcal{CH}(N, V)$ belonging to $\mathcal P_b(M, V)$, then for some $\tilde v \in \tilde V$ there exists a sequence of 4-tuples $(p^1_i, p^2_i, p^3_i, p^4_i)$, where $p^j_i$ is either in $\tilde V$ or in $\Lambda(\rho)$, such that the distance from $\tilde v$ to the convex span of the $i$-th 4-tuple tends to zero. Up to passing to subsequences, we can assume that all four sequences converge to points $p^1, p^2, p^3, p^4 \in \Lambda(\rho)$. Then $\tilde v$ belongs to the convex span of $p^1, p^2, p^3, p^4$, which is a contradiction with that $(\rho, \tilde f)$ projects to $\mathcal P_b(M, V)$.
\end{proof}

By $\ol{\mathcal P}_b(M, V)$ we denote the space of bent hyperbolic metrics $g$ on $M$ with $V(g)\subseteq V$ considered up to isometries preserving $V$ pointwise and isotopic to identity. Similarly to $\mathcal P_b(M, V)$, the set $\ol{\mathcal P}_b(M, V)$ embeds in $\mc{CH}(N, V)$ and is endowed with the induced topology. It is easy to see that $\ol{\mc P}_b(M, V)$ is exactly the closure of $\mc P_b(M, V)$ in $\mc{CH}(N, V)$.
From the description of the topology on $\overline{\mathcal P}_b(M, V)$ it follows that if a sequence $\{g_i\}$ converges to $g$, then one can choose the embeddings of $\tilde M(g_i)$ and $\tilde M(g)$ to $\H^3$ such that the respective holonomy maps $\{\rho_{\overline g_i}\}$ converge to $\rho_{\overline g}$. 

\begin{lm}
\label{hausdd}
Let a sequence $\{g_i\}$ converge to $g$ in $\overline{\mathcal P}_b(M, V)$. Choose the embeddings of $\tilde M(g_i)$ and $\tilde M(g)$ to $\H^3$ as above. Then $\{\tilde M(g_i)\}$ converges to $\tilde M(g)$ in the Hausdorff sense.
\end{lm}

\begin{proof}
Let $\tilde V_i$ and $\tilde V$ be the lifts of $V$ in $\tilde M(g_i)$ and $\tilde M(g)$ to $\H^3$. From the description of the topology on $\overline{\mathcal P}_b(M, V)$ the discrete sets $\tilde V(g_i)$ converge to $\tilde V$ in the Hausdorff sense. Consider the Klein model of $\H^3$. Recall that the Hausdorff convergence of closed subsets of a metric space depends only on the topology of the space. By Lemma~\ref{hausd}, $\Lambda(\overline g_i)$ converge to $\Lambda(\overline g)$ in the Hausdorff sense. As $\tilde M(g_i)\cup\Lambda(\overline g_i)=\conv(\Lambda(\overline g_i)\cup \tilde V_i)$ and $\tilde M(g)\cup\Lambda(\overline g)=\conv(\Lambda(\overline g)\cup \tilde V)$, we get that the closed subsets $\tilde M(g_i)\cup\Lambda(\overline g_i)$ of the unit ball converge in the Hausdorff sense to $\tilde M(g)\cup\Lambda(\overline g)$. It is easy to check that then the same do their intersections with the interior of the unit ball, i.e., $\tilde M(g_i)$ converge in the Hausdorff sense to $\tilde M(g)$.
\end{proof}

It is easy to conclude from this

\begin{crl}
\label{hausdb}
Let a sequence $\{g_i\}$ converge to $g$ in $\overline{\mathcal P}_b(M, V)$. Choose the embeddings of $\tilde M(g_i)$ and $\tilde M(g)$ to $\H^3$ as above. Then $\{\partial \tilde M(g_i)\}$ converges to $\partial \tilde M(g)$ in the Hausdorff sense.
\end{crl}

We will make use from the following standard fact

\begin{lm}[\cite{Ale2}, Theorem 4]
\label{distconv}
Let $S_i \subset \H^3$ be a sequence of complete convex surfaces converging to a complete convex surface $S \subset \H^3$ in the Hausdorff sense, and $p_i, q_i \in S_i$ be two sequences of points converging to $p, q \in S$. Then the intrinsic distances between $p_i$ and $q_i$ on $S_i$ converge to the intrinsic distance between $p$ and $q$ on $S$.
\end{lm}

For every $g \in \ol{\mathcal P}_b(M,V)$ the induced intrinsic metric on $\partial M(g)$ is a convex hyperbolic cone-metric $d$ with $V(d)\subseteq V$ defined up to isometries preserving $V$ pointwise and isotopic to identity. Hence, this defines the realization map
$$\mathcal I_V: \overline {\mathcal P}_b(M, V) \rightarrow \overline {\mathcal D}_c(\partial M,V).$$
(note that in Section~\ref{ressec} in order to simplify the exposition we reduced the domain and the codomain of $\mc I_V$). We can deduce

\begin{lm}
\label{contin}
The map ${\mathcal I}_V$ is continuous.
\end{lm}

\begin{proof}
Let $\{g_i\} \subset \overline {\mathcal P}_b(M, V)$ be a sequence converging to $g \in \overline {\mathcal P}_b(M, V)$. We note that here by $g_i$ and $g$ we mean some concrete metrics on $M$ such that their classes converge in $\overline {\mathcal P}_b(M, V)$. Choose the embeddings of $\tilde M(g_i)$ and $\tilde M(g)$ to $\H^3$ such that $\rho_{\overline g_i}$ converge to $\rho_{\overline g}$. Corollary~\ref{hausdb} implies that $\partial \tilde M(g_i)$ converges to $\partial \tilde M(g)$ in the Hausdorff sense. Let $d_i$ and $d$ be the induced intrinsic metrics on $\partial M(g_i)$ and $\partial M(g)$. 

We claim that there exists a finite set $W \subset \pt M$, $W \supseteq V$, and a geodesic triangulation $\mathcal T$ of $(\partial M, W, d)$ such that each edge is the unique shortest arc between its endpoints. Indeed, first we claim that there exists a finite set $W' \subset \pt M$, $W' \supseteq V$ and a geodesic triangulation $\mc T'$ of $(\pt M, W', d)$ such that each edge is a shortest arc between its endpoints, but possibly not unique. To this purpose it is enough to start with any geodesic triangulation of $(\pt M, V, d)$ and do the barycentric subdivision sufficiently many times. Next, it is easy to see that every proper sub-arc of a shortest arc is the unique shortest arc between the endpoints. Thus, if we divide every triangle of $\mc T'$ into six triangles by drawing all the medians, we will obtain the desired set $W$ and the triangulation $\mc T$.

By $\tilde W_i$ and $\tilde W$ we denote the full preimages of $W$ in $\pt M(g_i)$ and $\pt M(g)$ respectively. We make an isotopy of $g_i$ (and respectively of $d_i$) fixing $V$ such that for each $w \in W$ the orbit of $w$ under $\rho_{\overline g_i}$ converge to the orbit under $\rho_{\overline g}$.

We prove that after this the classes of $\{d_i\}$ converge to the class $d$ in $\overline{\mathcal D}_c(\partial M, W)$, which implies also that their classes converge in $\overline{\mathcal D}_c(\partial M, V)$ and proves the lemma. To this purpose we show that for all sufficiently large $i$ the triangulation $\mathcal T$ is realized in $d_i$ and that the lengths of the edges converge to those in $d$. Let $e$ be an edge of $\mathcal T$ of length $l_e$ in $d$, $\tilde e$ be a lift of $e$ with endpoints $\tilde v,\tilde w \in \tilde W$, let $\tilde v_i, \tilde w_i \in \tilde W_i$ be the respective points of $\tilde W_i$ converging to $\tilde v$, $\tilde w$. Lemma~\ref{distconv} shows that the distance between $\tilde v_i$ and $\tilde w_i$ converges to the distance between $\tilde v$ and $\tilde w$. Connect $\tilde v_i$ and $\tilde w_i$ by an intrinsic shortest arc $\tilde e_i$ in $\partial \tilde M(g_i)$, so the lengths of $\tilde e_i$ converge to the length of $\tilde e$. Parametrize the curves $\tilde e_i$ and $\tilde e$ by the arc length. We claim that $\tilde e_i$ converge to $\tilde e$ as parametrized curves. By~\cite[Chapter II.1, Theorems 4 and 5]{Ale3} from any subsequence of $\{\tilde e_i\}$ one can extract a further subsequence converging to a curve $\tilde e' \subset \partial \tilde M(g)$ of length at most $l_e$. Because $e$ is the unique shortest arc between its endpoints, the same is true for $\tilde e$, and we get that every arc $\tilde e'$ obtained this way is actually $\tilde e$. In particular, this argument implies that a subsequence of $\{\tilde e_i\}$ converges to $\tilde e$. If, however, $\{\tilde e_i\}$ does not converge to $\tilde e$ itself, by the Hausdorff maximal principle we can choose a maximal subsequence $\{\tilde e_{i_j}\}$ of $\{\tilde e_i\}$ converging to $\tilde e$. Then the complement to $\{\tilde e_{i_j}\}$ in $\{\tilde e_i\}$ should be infinite, but this means that we can again extract from the complement a subsequence converging to $\tilde e$, which contradicts to the maximality of $\{\tilde e_{i_j}\}$.

Let $e_i$ be the projection of $\tilde e_i$ to $\partial M$. It follows from the construction that for all sufficiently large $i$ the arc $e_i$ belongs to the same isotopy class of arcs in $\partial M\backslash W$ as~$e$. Doing the described procedure for all edges of $\mathcal T$ we obtain a collection of intrinsic geodesic arcs in $d_i$ constituting a triangulation equivalent to $\mathcal T$ for all sufficiently large~$i$. The lengths of the respective arcs converge to the lengths of the arcs of $\mathcal T$. This finishes the proof.
\end{proof}

%
%

Denote the space of classes of controllably polyhedral metrics by $\mathcal P_{cp}(M, V)$ and the space of classes of strictly polyhedral metrics by $\mathcal P_{sp}(M, V)$. We have
$$\mathcal P_{sp}(M, V) \subset \mathcal P_{cp}(M, V)\subset \mathcal P_{p}(M, V)\subset \mathcal P_{b}(M, V).$$
We also denote by $\overline{\mathcal P}_p(M,V)$, $\overline{\mathcal P}_{cp}(M,V)$, $\overline{\mathcal P}_{sp}(M,V)$ the respective subsets of metrics of $\overline{\mathcal P}_b(M,V)$.

\begin{dfn}
A hyperbolic metric $g$ on $M$ in $\mathcal P_{cp}(M,V)\backslash \mathcal P_{sp}(M, V)$ is called \emph{peculiar}.
\end{dfn}

\begin{lm}
\label{open1}
$\overline{\mathcal P}_{sp}(M, V)$ is open in $\overline{\mathcal P}_{b}(M, V)$.
\end{lm}

\begin{proof}
One needs to check that the function $g(\partial M(g), \partial C(\overline g))=d_{\H^3}(\partial \tilde M(g), \partial \tilde C(\overline g))$ is continuous at zero over $\overline{\mathcal P}_b(M, V)$. 
Consider $g \in \mathcal \overline{\mathcal P}_b(M, V)$ such that $$g(\partial M(g), \partial C(\overline g))=0.$$ Hence, there is a common point $p \in \partial \tilde M(g)\cap \partial \tilde C(\overline g)$. Let $\{g_i\}$ be a sequence converging to $g$ in $\overline{\mathcal P}_b(M, V)$. Due to Corollary~\ref{hausdb}, we can choose the embeddings of $\tilde M(g_i)$ and $\tilde M(g)$ to $\H^3$ such that $\{\partial \tilde M(g_i)\}$ converges to $\partial \tilde M(g)$ in the Hausdorff sense. Corollary~\ref{hausdb} also implies $\partial \tilde C(\overline g_i)$ converges to $\partial \tilde C(\overline g)$ in the Hausdorff sense. Hence, there exists a sequence of points $\{p_i \in \partial \tilde M(g_i)\}$ and $\{p'_i \in \partial \tilde C(\overline g_i)\}$ converging to $p$. Thus, $$g_i(\partial M(g_i), \partial C(\overline g_i)) \rightarrow 0.$$
\end{proof}

\begin{lm}
\label{open2}
$\overline{\mathcal P}_{cp}(M, V)$ is open in $\overline{\mathcal P}_{b}(M, V)$.
\end{lm}

\begin{proof}
Suppose the converse that there exists $g \in \overline{\mathcal P}_{cp}(M, V)$ and a sequence $$\{g_i\} \subset (\overline{\mathcal P}_{b}(M, V) \backslash \overline{\mathcal P}_{cp}(M, V))$$ converging to $g$ in $\overline{\mathcal P}_{b}(M, V)$. By $d_i$ and $d$ we denote the induced intrinsic metrics on $\partial M(g_i)$ and $\partial M(g)$. As ${\mathcal I}_V$ is continuous, the metrics $d_i$ converge to $d$ in the weak Lipschitz sense. Then Lemma~\ref{sys} implies that the systoles of $d_i$ are uniformly bounded from below by some constant~$\Delta>0$.

As $g_i \notin \overline{\mathcal P}_{cp}(M, V)$, Corollary~\ref{bs2c} implies that $\partial M(g_i)\cap\partial C(\overline g_i)$ contains at least one component, which is a hyperbolic surface with geodesic boundary. We can pick a hyperbolic pair of pants $P_i$ in this component. The lengths of the boundaries are bounded from below by $\Delta$. Hence, Corollary~\ref{pants} implies that there exists $\delta>0$ and a point $p_i$ in $P_i$ such that the intrinsic distance from $p_i$ to $\partial P_i$ is at least $\delta$. We use this to choose an intrinsic geodesic triangle $T_i \subset P_i$ containing $p_i$ with the areas of all $T_i$ uniformly bounded from below and the diameters uniformly bounded from above.

Consider the embeddings of $\tilde M(g_i)$ and $\tilde M(g)$ to $\H^3$ as in Corollary~\ref{hausdb}, so $\{\partial \tilde M(g_i)\}$ converges to $\partial \tilde M(g)$ in the Hausdorff sense and $\{\partial \tilde C(\overline g_i)\}$ converges to $\partial \tilde C(\overline g)$ in the Hausdorff sense. We can choose lifts $\tilde p_i \in \partial \tilde M(g_i)$ of $p_i$ such that the sequence $\{\tilde p_i\}$ is bounded in $\H^3$. Hence, up to passing to a subsequence we can assume that $\tilde p_i$ converge to a point $\tilde p \in \partial \tilde M(g)$. Now let $\tilde T_i \subset (\partial \tilde M(g_i)\cap \partial \tilde C(\overline g_i))$ be the lift of $T_i$ containing $\tilde p_i$. As $\tilde T_i$ are compact subsets of $\H^3$ staying in a bounded domain, up to taking a subsequence they must converge in the Hausdorff sense to a compact connected subset $\tilde T$ of $\partial \tilde M(g)\cap \partial \tilde C(\overline g)$. The latter consists only of disjoint geodesic lines, hence $\tilde T$ is a compact geodesic segment (possibly a point). However, Alexandrov showed that because $\tilde T_i$ and $\tilde T$ belong to complete convex surfaces $\partial \tilde M(g_i)$ and $\partial \tilde M(g)$, the areas of $\tilde T_i$ must converge to the area of $\tilde T$, which is zero. In the Euclidean case this is~\cite[Chapter X.2, Theorem 2]{Ale3}, the changes for the hyperbolic case are discussed in~\cite[Chapter XII.1]{Ale3}. This contradicts to the uniform lower bound on the areas of $\tilde T_i$.
\end{proof}

Note that Lemmas~\ref{open1} and~\ref{open2} particularly imply that $\mathcal P_{sp}(M, V)$ and $\mathcal P_{cp}(M, V)$ are open in $\mathcal P_{b}(M, V)$.

\subsection{Differentiability of the realization map}
\label{diffsec}

Our current aim is to study the differentiability properties of the realization map $\mathcal I_V$. In full generality this is a complicated task. We will restrict ourselves to controllably polyhedral metrics, and will prove that $\mathcal I_V$ is $C^1$ at strictly polyhedral metrics and differentiable at peculiar metrics.

\subsubsection{Differentiability for strictly polyhedral metrics}
\label{dssec}

\begin{lm}
\label{c1}
The realization map $\mathcal I_V$ is $C^1$ over $\mathcal P_{sp}(M, V)$.
\end{lm}

\begin{proof}
Pick $g \in \mathcal P_{sp}(M, V)$ and its neighborhood $U$ in $\mathcal P_{sp}(M, V)$. Lift $g$ to $(\rho, \tilde f) \in \mathcal R(\pi_1(M), G) \times (\H^3)^{V}$. Provided $U$ is small enough, there exists a neighborhood $\tilde U$ of $(\rho, \tilde f)$ projecting onto $U$ such that the action of $\pi_1(M)$ sends $\tilde U$ to disjoint domains. The action of $G$ on $\tilde U$ defines a smooth submersion $\tilde U \rightarrow U$. We will show that $\mathcal I_V$ is $C^1$ as a function over $\tilde U$. 

If there is an edge of $\partial \tilde M(g)$ between $\tilde v, \tilde w \in \tilde V(g)$, one can see that they are connected by an edge for all $(\rho', \tilde f')$ in $\tilde U$ provided that $\tilde U$ is sufficiently small. Indeed, the key is that the distance from any point of the edge to $\tilde C(\overline g)$ is positive and that the limit sets of converging representations converge in the Hausdorff sense. If there exists a sequence $(\rho_i, \tilde f_i)$ converging to $(\rho, \tilde f)$ such that $\tilde v$ and $\tilde w$ are not connected by an edge in all $\partial \tilde M(g_i)$, then there exists a sequence of 4-tuples of points $(p^1_i, p^2_i, p^3_i, p^4_i)$, where all $p^j_i$ are from $(\Lambda(\rho)\cup \tilde V(g))\backslash \{\tilde v, \tilde w\}$, such that the distance from the convex hull of each 4-tuple to the midpoint of the segment between $\tilde v$ and $\tilde w$ tends to zero. By passing to a subsequence we assume that each $p^j_i$ converge to $p^j \in \Lambda(\rho)$, which means that the distance from the midpoint to $\tilde C(\overline g)$ is zero, which is a contradiction. Note that the length of such an edge is a smooth function in $\tilde U$.

As $g\in \mathcal P_{sp}(M,V)$, all edges of $\partial \tilde M(g)$ are between pairs of vertices and all faces are convex compact hyperbolic polygons. In particular, the face decomposition of all $(\rho', \tilde f')$ in a sufficiently small neighborhood of $(\rho, \tilde f)$ is a subdivision of the face decomposition of $\partial \tilde M(g)$. Let $d$ be the induced intrinsic metric on $\partial M(g)$. Consider a geodesic triangulation $\mathcal T$ of $\partial M(g)$ subdividing the face decomposition. As $\mathcal T$ is realized in $(d, V)$, we have $d \in \mathcal D(V, \mathcal T)$. The lengths of all edges of $\mathcal T$ that are edges of $g$ are smooth functions around $g$. We need to show that the lengths of the other edges of $\mathcal T$ are $C^1$-functions.

Consider a non-triangular face $Q$ of $\partial \tilde M(g)$, let $V_Q$ be the set of its vertices and $e$ be a geodesic segment in the relative interior of $Q$ corresponding to an edge of $\mathcal T$ that is not an edge of $g$. When we move the vertices of $V_Q$, $\conv(V_Q)$ may become a convex polyhedron. If we move them sufficiently close to the initial positions, all boundary edges of $Q$ remain among the edges of $\conv(V_Q)$. They divide $\partial \conv(V_Q)$ into two parts. One of them, determined with the help of the orientation of $Q$ with respect to $\tilde M(g)$, is considered as a deformation of $Q$. We call this part \emph{upper}. We also require that we perturb the vertices slightly enough so that the intrinsic angles of all the vertices in the upper part stay less than $\pi$. The segment $e$ may become a broken line (i.e., after we move the vertices, we identify $e$ with the intrinsic shortest arc in the upper part between the vertices of $Q$). Consider a small neighborhood $U_Q \subset (\H^3)^{V_Q}$ of the initial position of $V_Q$ for which the boundary edges of $Q$ remain to be in the boundary of $\conv(V_Q)$ and for which the intrinsic angles of all the vertices in the upper part stay less than $\pi$. We can consider the length $l_e$ as a continuous function over $U_Q$. To finish the proof of the lemma, it is enough to show 

\begin{cla}
\label{differential}
The function $l_e$ is $C^1$ over $U_Q$ and its differential in the initial position coincides with the differential of the $\H^3$-distance between the endpoints of $e$.
\end{cla}

Let $\mathcal T_1, \ldots, \mathcal T_r$ be all triangulations of $Q$. Pick one of them, $\mathcal T_i$, and now, as we move the vertices, consider the deformations of $Q$ as of a surface triangulated with $\mathcal T_i$ (i.e., the edges of $\mathcal T_i$ remain geodesic segments of $\H^3$). It may become a non-convex surface, but near the initial position the edge $e$ is still defined in the obtained surface as the intrinsic shortest arc between its endpoints. Its length $l_{i,e}$ is a smooth function because the edge lengths of $\mathcal T_i$ change smoothly and the length of $e$ is a smooth function of them. We assume that $U_Q$ is small enough so that for all such polyhedral surfaces the intrinsic angles of all vertices stay less than $\pi$. Then $e$ is defined in all these surfaces over the whole $U_Q$. Let $U_i$ be the subset of $U_Q$ corresponding to the configurations where $\mathcal T_i$ subdivides the face decomposition of the upper part of $\conv(V_Q)$ is triangulated with $\mathcal T_i$. One can see that this produces a decomposition of $U_Q$ into cells $U_1, \ldots, U_r$ with non-empty disjoint interiors and piecewise-analytic boundaries. The function $l_e$ is equal to $l_{i,e}$ over $U_i$. In particular, it is smooth in the interiors of $U_i$. We need to check that if $U_i$, $U_j$ have a common boundary, then the first derivatives of $l_{i,e}$, $l_{j, e}$ coincide there.

Let $\Pi$ be the plane containing $Q$, we orient it arbitrarily. Set a local coordinate system in $U_Q$: for each $v \in V_Q$ choose two coordinates that are coordinates of the projections of $v$ to $\Pi$ and the third coordinate is the oriented distance from $v$ to $\Pi$. Call the first two coordinates \emph{horizontal} and the third \emph{vertical}. Note that in the initial moment all the horizontal derivatives of all $l_{i,e}$ coincide and are equal to the respective derivatives of the distance between the endpoints of~$e$. 

We now claim that all the vertical derivatives are zero. Indeed, let $ABC$ be a hyperbolic triangle with the corresponding sides equal to $a,b,c$ and with the right angle at $A$. Then we have
$$\cosh(a)=\cosh(b)\cosh(c).$$
Hence, for fixed $b>0$ as we move the vertex $B$ along the line $AB$ we have $\partial a/\partial c=0$ at $c=0$. This means that when we consider the bending of the polyhedral surface triangulated along $\mathcal T_i$, the vertical derivatives of all the edge lengths of $\mathcal T_i$ are zero at the initial moment. Hence, so are the vertical derivatives of $l_{i,e}$. This implies that all first derivatives of all $l_{i,e}$ coincide at the initial moment. The same argument implies that all first derivatives of any two $l_{i,e}$ and $l_{j,e}$ coincide at any point of the common boundary of $U_i$ and $U_j$. This shows that $l_e$ is $C^1$ in $U_Q$. In turn this finishes the proof.
\end{proof}

\subsubsection{Deformations of hyperbolic strips}
\label{stripssec}

In order to investigate the differentiability of the realization map for controllably polyhedral metrics, we need first to restrict ourselves to the case of closed solid hyperbolic tori with convex polyhedral boundary. 

Only in this subsection by $M$ we denote the closed solid torus and by $N$ we denote its interior. It is easy to see that the space $\mathcal{CH}(N)$ of convex cocompact hyperbolic metrics on $N$ up to isotopy is in one-to-one correspondence with the set of loxodromic isometries of $\H^3$ up to conjugation. Any such isometry $\gamma$ is determined up to conjugacy by two parameters: the shift $a>0$ along the geodesic line $\tilde \psi \subset \H^3$ fixed by $\gamma$ and the rotation angle $\alpha \in \R/2\pi\Z$ around $\tilde \psi$, measured clockwise to the shift direction. We make the convention that $\tilde \psi$ is oriented in the shift direction. Let $\Gamma < {\rm Iso}^+(\H^3)$ be the cyclic subgroup generated by $\gamma$. The quotient hyperbolic manifold $\H^3 / \Gamma$ is convex cocompact and homeomorphic to $N$. Its convex core is the closed geodesic $\psi$ that is the image of $\tilde \psi$. Note that for every $\overline g \in \mathcal{CH}(N)$ the hyperbolic manifold $N(\overline g)$ has a connected two-dimensional isometry group. Hence, for $V \neq \emptyset$ the space $\mathcal{CH}(N, V)$ has dimension $3n$, and so is the dimension of $\mathcal P_b(M, V)$. It is evident that every bent hyperbolic metric on $M$ is controllably polyhedral, so $\mathcal P_b(M, V)=\mathcal P_p(M, V)=\mathcal P_{cp}(M, V)$.


Let $g_0 \in \mathcal P_{cp}(M, V)$. We are interested in the case when $g_0$ is peculiar. Clearly, this is equivalent to the situation when all points of $\tilde V(g_0) \subset \H^3$ belong to a closed half-space containing $\tilde \psi$ in its boundary plane. Particularly, $\alpha_0=0$, where $\alpha_0$ is the rotation parameter for $g_0$. Also, if $\alpha_0=0$ and $\clconv(V)$ is homeomorphic to $M$, one can see that $n=|V| \geq 2$. 
By $\Sigma$ we denote the union of faces whose closures contain $\psi$ (there are either two or one of them), and by $\tilde \Sigma \subset \partial \tilde M(g)$ denote its universal cover. In this section we study deformations of the induced metric on $\Sigma$ when we vary $g_0$, so we expel from $V$ all vertices except the vertices of $\Sigma$. Of special interest are the deformations when $\alpha$ becomes nonzero. 
Our aim is to show that the induced metric is differentiable at $g_0$ and that the partial derivative of the metric by $\alpha$ is zero. Fix an intrinsically geodesic triangulation $\mathcal T$ of~$\Sigma$. There are two types of edges in $\mathcal T$: some connect vertices from different components of~$\partial \Sigma$ and some connect vertices from the same component. In order to consider less cases we assume that all interior edges of $\Sigma$ are of the first type. It is easy to see that such $\mathcal T$ exists. Let $e$ be an interior edge of~$\mathcal T$.

At the initial moment all vertices of $\tilde \Sigma$ belong to two half-planes intersecting at $\tilde \psi$, which we denote by $\Pi^{1}$ and $\Pi^{2}$. We denote the vertices with the lifts belonging to $\Pi^{1}$ by $v^1_1, \ldots, v^1_{n_1}$ and to $\Pi^{2}$ by $v^2_1, \ldots, v^2_{n_2},$ where $n_1+n_2=n$. For each $v^j_i$, where $j=1,2$, we fix a lift $\tilde v^j_i \in \partial \tilde \Sigma$. By $u^j_i$ we denote the projection of $\tilde v^j_i$ to $\tilde \psi$. By $x^j_{i, 0}$ we denote the distance from $\tilde v^j_i$ to $\tilde \psi$. By $y^j_{i, 0}$ we denote the oriented distance from $u^1_1$ to $u^j_i$. We assume that $v^1_1$ and $v^2_1$ are the endpoints of $e$ and consider its lift $\tilde e$, which is the unique intrinsic shortest arc in $\tilde \Sigma$ between $\tilde v^1_1$ and $\tilde v^2_1$. By possibly changing the labels, we assume that $\tilde \psi$ is oriented in the direction from $u^1_1$ to $u^2_1$. By $\phi_0 \geq \pi$ we denote the angle from $\Pi^1$ to $\Pi^2$ clockwise. 

Note that when we vary $g_0$ so that $\alpha$ becomes nonzero, the strip $\Sigma$ becomes triangulated. There are infinitely many triangulations of $\Sigma$ realized by metrics in any neighborhood of~$g_0$. This is the reason, why we can not really apply the methods of Section~\ref{dssec} to show the differentiability of the boundary metric, and have to look for another way.

Take $g \in \mathcal P_{cp}(M,V)$ sufficiently close to $g_0$ so that all edges of $g_0$ except possibly $\psi$ remain to be edges of~$g$. Then the boundary edges of $\Sigma$ divide $\partial M(g)$ into two strips. One of them is a deformation of $\Sigma$ and we will continue to denote it~$\Sigma$. We assume that $g$ is sufficiently close to $g_0$ so that the intrinsic angles of the vertices in $\Sigma$ remain to be less than $\pi$.
We choose the same lifts of all $v^j_i$, so $u^j_i$ remain defined also for $g$, and we define $x^j_i$ and $y^j_i$ in a similar fashion to $x^j_{i, 0}$ and $y^j_{i,0}$. Hence, $x^j_i$ and $y^j_i$ are smooth functions in $\mathcal P_{cp}(M,V)$. By $\Pi^j_i$ we denote the half-plane spanned by $\tilde \psi$ and $\tilde v^j_i$. 
By $\phi^j_i$ we denote the angle between the segments $u^j_i\tilde v^j_i$ and the half-plane $\Pi^1_1$. Since for $g_0$ we have $\phi^1_{i,0}=0$, $\phi^2_{i,0}=\phi_0$, we consider $\phi^1_i$ as a real number valuing in a neighborhood of 0, and $\phi^2_i$ as a real number valuing in a neighborhood of $\phi_0$.
One can see that the $3n$ variables $$x^1_1, \ldots, x^1_{n_1}, x^2_1, \ldots, x^2_{n_2}, y^1_2, \ldots, y^1_{n_1}, y^2_1, \ldots, y^2_{n_2}, \phi^1_2, \ldots, \phi^1_{n_1}, \phi^2_1, \ldots, \phi^2_{n_2}, a, \alpha$$ provide local coordinates in a neighborhood of $g_0$ in $\mathcal P_{cp}(M, V)$. Particularly we use these coordinates to consider a neighborhood of $g_0$ as a domain in a $(3n)$-dimensional real vector space, and by $\|.\|$ we denote the standard Euclidean norm there.

Since now and until the end of this section we are going to work only in the universal cover embedded in $\H^3$. To simplify the reading, until the end of this section we drop tilde in $\tilde M(g)$, $\tilde \psi$, $\tilde \Sigma$, $\tilde v^j_i$, writing just $M(g)$, $\psi$, $\Sigma$, $v^j_i$ instead.

We will continue to denote the intrinsic shortest arc in $\Sigma$ with the metric induced by $g$ between $v^1_1$ and $v^2_1$ by $e$. By $l_e$ we denote its length, so we consider $l_e$ as a function in a neighborhood of $g_0$, particularly a function of our $3n$ variables. The results of Section~\ref{dssec} apply also to the situation we have now provided that $\alpha \neq 0$. Hence, $l_e$ is $C^1$ when $\alpha\neq 0$. By $e'$ we denote the intrinsic shortest arc between $v^1_1$ and $v^2_1$ in $\Pi^1_1 \cup \Pi^2_1$. By $l'_e$ we denote the length of $e'$. The function $l'_e$ is a smooth function of three variables: $x^1_1$, $x^2_1$, $y^2_1$. Our aim is to show 

\begin{lm}
\label{est1}
For every $0<\mu<1$ the difference $l_e-l'_e$ is $O(\|\Delta g\|^{1+\mu})$ at $g_0$. 
\end{lm}

Here $\Delta g=g-g_0$ is an abstract vector in our coordinates. 

\begin{crl}
\label{diff}
The function $l_e$ is differentiable at $g_0$. Moreover, its differential coincides with the differential of the function $l'_e$. In particular, its derivatives with respect to $\phi^j_i$ and to $\alpha$ are zero.
\end{crl}

We have to consider separately the cases $\phi_0 > \pi$ and $\phi_0=\pi$. In the case $\phi_0>\pi$ we assume that the deformation is small enough so that all differences $\phi^2_{i_2}-\phi^1_{i_1}$ are greater than $\pi$. Then the realization of $e$ in $\partial M(g)$ belongs to the concave part of $\H^3 \backslash (\Pi^1_1 \cup \Pi^2_1)$, so by the Busemann--Feller lemma we have $l'_e \leq l_e$ and the non-trivial part of Lemma~\ref{est1} is to provide an upper bound on $l_e-l'_e$. The case $\phi=\pi$ is slightly more involved, but the upper bound still remains the main issue of the proof.

Let $p \in \partial M(g)$ and $q$ be its projection to $ \psi$. By $\phi_p \in \R /2\pi\Z$ we denote the angle between the segment $pq$ and $\Pi^1_1$ clockwise. By $t_p$ we denote the signed distance $u^1_1q$. This provides a parametrization of $\partial M(g)$. However, we are more interested in the coordinate $\phi_p$ extended to the universal cover of $\partial M(g)$ (recall that $\partial M(g)$ is topologically a cylinder). That is, we consider the universal cover of $\partial M$, fix a lift $\hat v^1_1$ of $v^1_1$ and for $p \in \partial M(g)$ by $\theta_p \in \R$ we denote the extended oriented angle from $\hat v^1_1$ to the lift $\hat p$ of $p$ such that $\hat p$ belongs to the same lift of $\Sigma$ or of $\partial M(g)\backslash \Sigma$ as~$\hat v^1_1$.


Fix $\mu \in \R$ such that $0<\mu<1$ and suppose that $\phi_0>\pi$. Set $\mu'=(1+\mu)/2.$ Consider in $ \Sigma$ the curve $ e''$ consisting of the following three arcs. Denote the intersection point of $e'$ with $\psi$ by $u$. Let $p_1 \in \Sigma$ be a point such that 
\\
(1) the plane spanned by $ v^1_1$, $u$ and $p_1$ is orthogonal to $\Pi^1_1$;
\\
(2) $\theta_{p_1}=\max_{1 \leq i \leq n_1}\phi^1_i+|\alpha|^{\mu'}$ 
\\
Similarly, let $p_2$ be the point such that 
\\
(1) the plane spanned by $ v^2_1$, $u$ and $p_2$ is orthogonal to $\Pi^2_1$; 
\\
(2)~$\theta_{p_2}=\min_{1 \leq i \leq n_2}\phi^2_i-|\alpha|^{\mu'}$.
\\
Now let $ e''_1$ be the curve connecting $ v^1_1$ with $p_1$ in the intersection of $\partial M(g)$ with the plane $ v^1_1up_1$ in the clockwise direction. Similarly, let $ e''_2$ be the curve connecting $ v^2_1$ with $p_2$ in the intersection of $\partial M(g)$ with the plane $ v^2_1up_2$ in the counterclockwise direction. Finally, let $ e''_3$ be the curve connecting $p_1$ with $p_2$ in the intersection of $\partial M(g)$ with the plane spanned by $u, p_1, p_2$ in the clockwise direction. Altogether the arcs $ e''_1$, $ e''_3$, $ e''_2$ form a curve in $\Sigma$ connecting $ v ^1_1$ with $ v^2_1$. By $l''_e$ we denote its length. We have $l''_e \geq l_e \geq l'_e$. Thus, Lemma~\ref{est1} for $\phi_0 > \pi$ follows from 


\begin{lm}
\label{est2}
Let $\phi_0 > \pi$. For every $0<\mu<1$ the difference $l''_e-l'_e$ is $O(\|\Delta g\|^{1+\mu})$ at~$g_0$. 
\end{lm}

Now consider the case $\phi_0=\pi$. In this case we propose a different construction of the auxiliary curve $e''$. Emit a geodesic ray from $u$ such that it is orthogonal to $\psi$, directed towards $\Sigma$, and the plane spanned by this ray and $\psi$ bisects the dihedral angle between $\Pi^1_1$ and $\Pi^2_1$. Let $p_0$ be the intersection point of this ray with $\Sigma$. Let $e''_1$ be the curve connecting $v^1_1$ with $p_0$ in the intersection of $\pt M(g)$ with the plane $v^1_1up_0$ in the clockwise direction, and $e''_2$ be the curve connecting $p_0$ with $v^2_1$ in the intersection of $\pt M(g)$ with the plane $v^2_1up_0$. Altogether the arcs $e''_1$ and $e''_2$ form a curve in $\Sigma$ connecting $v^1_1$ with $v^2_1$, which we again denote by $e''$. By $l''_e$ we denote its length. We have $l''_e \geq l_e$. We will similarly show

\begin{lm}
\label{est3}
Let $\phi_0 = \pi$ and $\phi^2_1 \geq \pi$. For every $0<\mu<1$ the difference $l''_e-l'_e$ is $O(\|\Delta g\|^{1+\mu})$ at $g_0$. 
\end{lm}

In the case $\phi^2_1 \geq \pi$, by the Busemann--Feller lemma we have $l_e \geq l'_e$ and Lemma~\ref{est1} follows from Lemma~\ref{est3}. When $\phi^2_1 < \pi$, a simple additional argument is required, which we provide below. 

Denote the hyperbolic metric of $\H^3$ by $d_{\H^3}$. To prove our claims we need to make the following estimates on the geometry of $\Sigma$. 

\begin{lm}
\label{gest1}
Let $\phi_0>\pi$, $\alpha \neq 0$. For every $0<\mu<1$ and every $T>0$ we have
$$\sup\{d_{\H^3}(p,  \psi): p \in \Sigma, -T \leq t_p \leq T, \max_{1 \leq i \leq n_1}\phi^1_i+|\alpha|^\mu \leq  \theta_p \leq \min_{1 \leq i \leq n_2}\phi^2_i-|\alpha|^\mu \} = O(\|\Delta g\|^2).$$
\end{lm}

\begin{lm}
\label{gest2}
Let $\alpha \neq 0$. For every $0<\mu<1$ and every $T>0$ we have
$$\sup\{d_{\H^3}(p, \Pi^1_1): p \in \Sigma, -T \leq t_p \leq T, 0 \leq  \theta_p \leq \frac{\pi}{2}  \} = O(\|\Delta g\|^\mu),$$
$$\sup\{d_{\H^3}(p, \Pi^2_1): p \in \Sigma, -T \leq t_p \leq T,  \phi^2_1 - \frac{\pi}{2} \leq  \theta_p \leq  \phi^2_1 \} = O(\|\Delta g\|^\mu).$$
\end{lm}

We postpone proofs of Lemmas~\ref{gest1}--\ref{gest2} and first deduce Lemmas~\ref{est2} and~\ref{est3} from them. To this purpose we also need few less involved estimates

\begin{lm}
\label{compare}
Fix $0<\mu\leq 1$ and let $\xi>0$ be a parameter. Consider the points $p, p' \in \H^2$ connected by a geodesic segment $\psi$. Let $\chi$ and $\chi'$ be geodesic rays from $p$ and $p'$ at the same side of $\psi$, and $\sigma$ be a convex arc starting at $p$ tangent to $\chi$ and ending at a point $p_0 \in \chi'$. By $r_0$ denote the foot of the perpendicular from $p_0$ to $\psi$. Assume that
\\
(1) the angle between $\psi$ and $\chi$ is $O(\xi^\mu)$;
\\
(2) the angle between $\psi$ and $\chi'$ is $O(\xi^\mu)$;
\\
(3) the length of $\psi$ is $O(1)$.
\\
Here all asymptotics are as $\xi$ tends to zero. Then $$\len(\sigma)-pr_0=O(\xi^{2\mu}).$$
\end{lm}

Here and in what follows we will abuse the notation and sometimes write $pr_0$ for the length of the geodesic segment $pr_0$.

\begin{proof}
We have $\len(\sigma)-pr_0 \geq 0$. Clearly, the rays $\chi$ and $\chi'$ intersect provided that $\xi$ is sufficiently small, denote the intersection point by $p''$. Let $\sigma''$ be the broken line consisting of the segments $pp''$ and $p''p_0$. By the Busemann--Feller lemma we have $\len(\sigma'') \geq \len(\sigma)$. Thus, it is enough to show that $\len(\sigma'')-pr_0 = O(\xi^{2\mu}).$

Let $r$ be the foot of the perpendicular from $p''$ to $\psi$. We have
$$\tanh(pr)=\tanh(pp'')\cos(\angle p''pr).$$
As $\xi \rightarrow 0$ we get from (1)
$$1-\cos(\angle p''pr)=O(\xi^{2\mu}).$$
Also clearly $pp''=O(1)$, $pr=O(1)$. Thus, we can estimate
\begin{multline}
\label{comp1}
0\leq pp''-pr \leq \sinh(pp''-pr) = \cosh(pp'')\cosh(pr)(\tanh(pp'')-\tanh(pr))= \\
=O(1)\tanh(pp'')(1-\cos(\angle p''pr))=O(\xi^2).
\end{multline}

Now for the trapezoid $p''p_0r_0r$ we have from Lemma~\ref{sinlaw}
\begin{equation}
\label{comp2}
\frac{\sinh(p''p_0)}{\sinh(rr_0)}=\frac{\cosh(p''r)}{\sin(\ang p''p_0r_0)}.
\end{equation}

From
$$\tan(\ang p_0p'r_0)\cosh(p_0p')=\cot(\ang p'p_0r_0)=\tan\left(\ang p''p_0r_0-\frac{\pi}{2}\right)$$
we get $0 \leq \ang p''p_0r_0-\frac{\pi}{2} = O(\xi^\mu).$ Note that
$$\cosh(p''r)-1=O(\xi^{2\mu}),~~~~~1-\cos(\angle p_0p'r_0)=O(\xi^{2\mu}).$$

Using this and~(\ref{comp2}) we obtain
\begin{multline}
\label{comp3}
0\leq p''p_0 - rr_0 \leq 2\sinh\left(\frac{p''p_0 - rr_0}{2}\right)\leq 2\sinh\left(\frac{p''p_0 - rr_0}{2}\right)\cosh\left(\frac{p''p_0 + rr_0}{2}\right)=\\
=\sinh(p''p_0)-\sinh(rr_0)=\sinh(p''p_0)\left(1-\cos\left(\ang p''p_0r_0-\frac{\pi}{2}\right)\right)+\\+\sinh(rr_0)(\cosh(p''r)-1)
=O(\xi^{2\mu}).
\end{multline}

Altogether~(\ref{comp1}) and~(\ref{comp3}) give us
$$0\leq \len(\sigma'')-pr_0=pp''+p'p''-pr-rr_0 = O(\xi^{2\mu}).$$
\end{proof}

\begin{lm}
\label{compare2}
Fix $0<\mu<1$ and let $\xi>0$ be a parameter. Let $ \psi$ be a geodesic line in $\H^3$, $C$ be the surface consisting of all points at distance $x$ from $ \psi$ where $x=O(\xi^2)$ as $\xi$ tends to zero. Consider two points $p,p' \in C$, let $q,q'$ be their projections to $ \psi$ and $\sigma \subset C$ be an intrinsic shortest arc connecting them. Assume that $qq'=O(1)$. Then
$$0\leq \len(\sigma)-qq'=O(\xi^2).$$
\end{lm}

\begin{proof}
Let $p''$ be the point at the intersection of $C$, the plane spanned by $p$ and $ \psi$ and of the plane orthogonal to $ \psi$ at $q'$. Let $\sigma'' \subset C$ be the intrinsic shortest arc connecting $p$ to $p''$ and $\sigma' \subset C$ be the intrinsic shortest arc connecting $p'$ to $p''$. We then have $\len(\sigma'')=qq'\cosh(x)$. Thus, $0 \leq \len(\sigma'')-qq'=O(\xi^2)$. Also $\sigma'$ is an arc of a circle centered at $q'$ with radius $x$, hence $\len(\sigma')=O(\xi^2)$. Finally, $$\len(\sigma)-qq'\leq \len(\sigma')+\len(\sigma'')-qq'=O(\xi^2).$$
\end{proof}

Now we are ready to prove Lemma~\ref{est2}.

\begin{proof}
Suppose that $\alpha \neq 0$. Consider the arc $ e''_1$. It belongs to the orthogonal plane to $\Pi ^1_1$ passing through $ v ^1_1$ and $u$. Let $r_1$ be the foot of the perpendicular from $p_1$ to $ v ^1_1u$ and $q_1$ be the foot of the perpendicular from $p_1$ to $\psi$. We want to apply Lemma~\ref{compare} to $ e''_1$ and to the segment $ v ^1_1u$. Considering the spherical link at the point $u$ we get
$$\tan(\ang p_1u v ^1_1)=\tan(\theta_{p_1})\sin(\ang q_1u v ^1_1).$$
Thus, from $\theta_{p_1}=O(\|\Delta g\|^{\mu'})$ we get $\ang p_1u v ^1_1=O(\|\Delta g\|^{\mu'})$. It remains to show that the angle between the tangent ray to $ e''_1$ and $ v ^1_1u$ at $ v ^1_1$ is $O(\|\Delta g\|^{\mu'})$. Denote this angle by~$\eta_1$.

Let $v, w$ be the neighbouring vertices of $v ^1_1$ in $\pt \Sigma$. By $z_1$ denote the intrinsic distance from $v ^1_1$ to the intrinsic geodesic segment $vw$ in $\Sigma$. Let $p'_1$ be the intersection point of $ e''_2$ with $vw$. Note that the part of $e''_1$ between $ v ^1_1$ and $p'_1$ is a geodesic segment in $\H^3$ and $v^1_1p'_1 \geq z_1$. Due to Lemma~\ref{gest2}, we have $d_{\H^3}(p'_1, \Pi ^1_1)=O(|g|^{\mu'})$. Thus, from
$$\sinh(d_{\H^3}(p'_1, \Pi ^1_1))=\sin(\eta_1)\sinh( v ^1_1p'_1)$$
we get $\eta_1=O(\|\Delta g\|^{\mu'})$.

Now we can apply Lemma~\ref{compare}. We obtain
$$\len( e''_1)- v^1_1r_1 =O(\|\Delta g\|^{2\mu'})=O(\|\Delta g\|^{1+\mu}).$$

We deal just the same with the arc $ e''_2$ of $ e''$ between $ v^2_1$ and $p_2$. Let $r_2$ be the foot of the perpendicular from $p_2$ to $ v^2_1u$. We get
$$\len( e''_2)- v^2_1r_2 =O(\|\Delta g\|^{1+\mu}).$$

Now we need to bound $\len( e''_3)-r_1u-r_2u$. By $u_1$, $u_2$ denote the projections of $r_1$, $r_2$ to $ \psi$. We have $r_1u \geq u_1u$, $r_2u \geq u_2u$. 
Lemma~\ref{gest2} shows that $$\sup_{p \in  e''_3} d_{\H^3}(p, \psi) = O(\|\Delta g\|^2).$$
Hence we can choose a radius $x=O(\|\Delta g\|^2)$ such that $ e''_3$ belongs to the set of points at distance at most $x$ from $ \psi$. Let $C$ be the boundary of this set, $p''_1$ and $p''_2$ be the intersections of $C$ with the rays $u_1p_1$, $u_2p_2$, let $\sigma \subset C$ be the intrinsic shortest arc connecting $p''_1$ and $p''_2$. Lemma~\ref{compare2} implies that $\len(\sigma)-u_1u_2 = O(\|\Delta g\|^2)$. The Busemann--Feller lemma shows that 
$$\len( e''_3) \leq p_1p''_1+\len(\sigma)+p''_2p_2=\len(\sigma)+O(\|\Delta g\|^2).$$
Thus, $\len( e''_3)-u_1u_2=O(\|\Delta g\|^2).$
Hence, also $\len( e''_3)-r_1u-r_2u=O(\|\Delta g\|^2).$

Summing this up we get
$$l''_e-l'_e=O(\|\Delta g\|^{1+\mu}).$$

Now suppose that $\alpha=0$. Then the arc $e''_3$ is the union of two geodesic segments, $p_1u$ and $p_2u$. Denote similarly the points $p'_1$, $p'_2$ and the angles $\eta_1$, $\eta_2$. We have
$$d_{\H^3}(p'_1, \Pi^1_1)\leq \sinh(d_{\H^3}(p'_1, \Pi^1_1)) \leq \max\{\sin(|\phi^1_i|)\}\max\{\sinh(x^1_i)\}=O(\|\Delta g\|).$$
Hence, also $\eta_1=O(\|\Delta g\|)$ and from Lemma~\ref{compare} we get
$$\len(e''_1)+p_1u-v^1_1u=O(\|\Delta g\|^2).$$
Similarly,
$$\len(e''_2)+p_2u-v^2_1u=O(\|\Delta g\|^2)$$
and altogether
$$l''_e-l'_e=\len(e''_1)+p_1u+p_2u+\len(e''_2)-v^1_1u-v^2_1u=O(\|\Delta g\|^2),$$
which finishes the proof.
\end{proof}

For Lemma~\ref{est3} we also require

\begin{lm}
\label{compare3}
Fix $0<\mu\leq 1$ and let $\xi>0$ be a parameter. Let $p_1p_0p_2q$ be a hyperbolic quadrilateral with $\ang p_1qp_2 \geq \pi$, but with the other angles $<\pi$, such that \\
(1) the angle $\ang p_1qp_2-\pi=O(\xi)$;\\
(2) the angles $\ang p_0p_1q$ and $\ang p_0p_2q$ are $O(\xi^\mu)$;\\
(3) the lengths $p_1q$ and $p_2q$ are $O(1)$.\\
Here all asymptotics are as $\xi$ tends to zero. Let $\sigma$ be a convex arc connecting $p_1$ and $p_2$ such that $q$ belongs to the convex side. Then
$$\len(\sigma)-p_1q-p_2q=O(\xi^{2\mu}).$$
\end{lm}

\begin{proof}
We have $\len(\sigma)-p_1q-p_2q \geq 0$. Let $\sigma'$ be the broken line consisting of the segments $p_1p_0$ and $p_2p_0$. By the Busemann--Feller lemma we have $\len(\sigma')\geq \len (\sigma)$. Thus, it is enough to show $$\len(\sigma')-p_1q-p_2q=p_1p_0+p_2p_0-p_1q-p_2q = O(\xi^{2\mu}).$$

Let the ray $p_1q$ intersect the segment $p_0p_2$ in a point $p'_0$. It is easy to check that
$$\ang p_0p'_0p_1 = O(\xi^\mu).$$
Applying Lemma~\ref{compare} to the triangle $p_1p_0p'_0$ we get
$$p_1p_0+p_0p'_0-p'_0p_1=O(\xi^{2\mu}).$$
Applying Lemma~\ref{compare} to the triangle $p_2qp'_0$ we obtain
$$qp'_0+p'_0p_2-p_2q=O(\xi^{2\mu}).$$
Altogether this gives
$$p_1p_0+p_2p_0-p_1q-p_2q = O(\xi^{2\mu}).$$
\end{proof}

Now we prove Lemma~\ref{est3}.

\begin{proof}
Suppose that $\alpha \neq 0$. Let $\eta_1$ be the angle between the tangent ray to $e''_1$ and $v^1_1u$ at $v^1_1$, $\eta_2$ be the angle between the tangent ray to $e''_2$ and $v^2_1u$ at $v^2_1$. Similarly to the proof of Lemma~\ref{est2}, from Lemma~\ref{gest2} we get that $\eta_1$ and $\eta_2$ are $O(\|\Delta g\|^{\mu'})$. Consider the union of the half-planes of the planes $v^1_1p_0u$ and $v^2_1p_0u$ containing the arc $e''_1$. Develop them to $\H^2$. Clearly, when $\|\Delta g\|$ is sufficiently small, the tangent rays to $e''_2$ at its endpoints intersect in $\H^2$. Let $p'_0$ be the intersection point. Our result follows from applying Lemma~\ref{compare3} to the quadrilateral $v^1_1p'_0v^2_1u$ and the arc $e''_2$.

For $\alpha=0$ we argue like in the proof of Lemma~\ref{est2} to obtain that $\eta_1$ and $\eta_2$ are $O(\|\Delta g\|)$. Then Lemma~\ref{compare3} gives us 
$$l''_e-l'_e=O(\|\Delta g\|^2),$$
which finishes the proof.
\end{proof}

We can finish the proof of Lemma~\ref{est1}. 

\begin{proof}[Proof of Lemma~\ref{est1}.]
Due to the discussion above, Lemma~\ref{est1} follows from Lemmas~\ref{est2} and~\ref{est3} except the case $\phi_0=0$, $\phi^2_1<\pi$, so we deal with this case here. Suppose that $\alpha \neq 0$. Let $\Pi$ be the plane containing the half-plane $\Pi^1_1$ and $v_0$ be the orthogonal projection of $v^2_1$ to $\Pi$. By $e_0$ we denote the union of the segments $v^1_1u$ and $uv_0$, by $l_0$ we denote its length. We act now similarly to the proof of Lemma~\ref{est3}. By $\eta_1$, $\eta_2$ we denote the angles between the tangent rays to $e''_2$ at its endpoints and between the segments connecting the endpoints with $u$. We get that $\eta_1$, $\eta_2$ are $O(\|\Delta g\|^{\mu'})$. Since the dihedral angle between $\Pi^2_1$ and $\Pi$ is $O(\|\Delta g\|)$, it is easy to see that when $\|\Delta g\|$ is small enough, the extension of the tangent ray to $e''_2$ at $v^2_1$ in the converse direction intersects $\Pi$ and the intersection angle is $O(\|\Delta g\|^{\mu'})$. Applying Lemma~\ref{compare} we get
$$l''_e-l_0=O(\|\Delta g\|^{2\mu'}).$$

By $l'_0$ we denote the length $v^1_1v_0$. It is easy to check that since the dihedral angle between $\Pi^2_1$ and $\Pi$ is $O(\|\Delta g\|)$, we get $l'_0-l_0=O(\|\Delta g\|^2)$ and $l'_e-l'_0=O(\|\Delta g\|^2)$. Thus, $l''_e-l'_0=O(\|\Delta g\|^{2\mu'}).$ Since $l''_e \geq l_e \geq l'_0$, we obtain $l_e-l'_0=O(\|\Delta g\|^{2\mu'})$, and, finally, $$l_e-l'_e=O(\|\Delta g\|^{2\mu'}).$$

For $\alpha=0$, similarly to the proofs of Lemmas~\ref{est2},~\ref{est3} we can obtain that $\eta_1$, $\eta_2$ are $O(\|\Delta g\|)$. The remaining part is identical to the argument above.
\end{proof}

Now we start to prove Lemmas~\ref{gest1}--\ref{gest2}. To this purpose we introduce an auxiliary surface. Since now we assume $\alpha \neq 0$. Let $\Gamma^* < {\rm Iso}^+(\H^3)$ be the smooth 1-parametric subgroup consisting of isometries with the oriented shift $t \in \R$ along $ \psi$ in the positive direction and rotating by $\frac{\alpha t}{a}$ clockwise. Note that $\Gamma < \Gamma^*$. By $\tau^j_i$ we denote the $\Gamma^*$-orbit of $ v^j_i$, which is a smooth curve, by $ M^*(g)$ we denote the closed convex hull of all $\tau^j_i$. We have $ M(g) \subset  M^*(g)$.

\begin{lm}
\label{strips}
If $g$ is sufficiently close to $g_0$, then the boundary $\partial  M^*(g)$ contains a curve $\tau^1_{i_1}$ for some $i_1$ and a curve $\tau^2_{i_2}$ for some $i_2$.
\end{lm}

\begin{proof}
Clearly $\partial M^*(g)$ must contain at least one curve $\tau^j_i$. Suppose that there exists a sequence $g_k$ converging to $g$ in $\mathcal P_b(M, V)$ such that $\partial  M^*(g)$ does not contain curves of the type $\tau^2_i$. Similarly to Corollary~\ref{hausdb}, one can see that $\partial M^*(g_k)$ must converge in the Hausdorff sense to $\partial M^*(g)$. However, it is easy to see that $\clconv(\{\tau^1_i\})$ converge in the Hausdorff sense to a subset of $\Pi^1$ bounded by $\psi$ and by a curve $\tau^1_{i_0} \subset \Pi^1$ at the greatest distance from $\psi$ in $g_0$.
\end{proof}

The union of the curves $\tau^j_i$ belonging to $\partial  M^*(g)$ divide $\partial  M^*(g)$ into strips. When $g$ is sufficiently close to $g_0$, there are exactly two strips bounded by a curve $\tau^1_{i_1}$ for some $i_1$ and a curve $\tau^2_{i_2}$ for some $i_2$. We choose the one, for which the curve $\tau^2_{i_2}$ is clockwise with respect to the curve $\tau^1_{i_1}$. We denote this strip by $\Sigma^*$. From now on we denote the boundary curves of $\Sigma^*$ just by $\tau_1$ and $\tau_2$. Similarly, we denote $v^1_{i_1}$, $v^2_{i_2}$, $x^1_{i_1}$, $x^2_{i_2}$, $y^1_{i_1}$, $y^2_{i_2}$, $\phi^1_{i_1}$, $\phi^2_{i_2}$ just by $v_1$, $v_2$, $x_1$, $x_2$, $y_1$, $y_2$, $\phi_1$, $\phi_2$. Define $y:=y_2-y_1$, $\phi:=\phi_2-\phi_1$.

The surface $\Sigma^*$ is invariant with respect to $\Gamma^*$. Each point $p \in \Sigma^* \backslash (\tau_1 \cup \tau_2)$ either belongs to a segment contained in $\Sigma^*$ with vertices on $\tau_1$, $\tau_2$, or belongs to a triangle contained in $\Sigma^*$ with vertices on $\tau_1$, $\tau_2$. As there is no plane invariant with respect to $\Gamma^*$, the second option is not possible. Thus, $\Sigma^*$ is a ruled surface decomposing into the union of segments so that the decomposition is $\Gamma^*$-invariant. 

The main step is to prove analogues of Lemmas~\ref{gest1} and~\ref{gest2} for the surface $\Sigma^*$.

\begin{lm}
\label{gest1*}
Let $\phi_0>\pi$. For every $0<\mu<1$ and every $T>0$ we have
$$\sup\{d_{\H^3}(p,  \psi): p \in \Sigma^*, -T \leq t_p \leq T, \max_{1 \leq i \leq n_1}\phi^1_i+|\alpha|^\mu \leq  \theta_p \leq \min_{1 \leq i \leq n_2}\phi^2_i-|\alpha|^\mu \} = O(\|\Delta g\|^2).$$
\end{lm}

\begin{lm}
\label{gest2*}
For every $0<\mu<1$ and every $T>0$ we have
$$\sup\{d_{\H^3}(p, \Pi^1_1): p \in \Sigma^*, -T \leq t_p \leq T, \max_{1 \leq i \leq n_1}\phi^1_i+|\alpha|^\mu \leq  \theta_p \leq \frac{\pi}{2}  \} = O(\|\Delta g\|^\mu),$$
$$\sup\{d_{\H^3}(p, \Pi^2_1): p \in \Sigma^*, -T \leq t_p \leq T,  \phi^2_1 - \frac{\pi}{2} \leq  \theta_p \leq  \min_{1 \leq i \leq n_2}\phi^2_i-|\alpha|^\mu \} = O(\|\Delta g\|^\mu).$$ 
\end{lm}


Note that there is a slight difference in the bounds on $\theta_p$ between Lemma~\ref{gest2} and Lemma~\ref{gest2*}. This is done on purpose to simplify our argument.
Since $ M(g) \subset  M^*(g)$, Lemmas~\ref{gest1} and~\ref{gest2} follow from Lemmas~\ref{gest1*} and~\ref{gest2*} quite easily:

\begin{proof}[Proof of Lemma~\ref{gest1}.]
Note that if $p \in \tau^j_i$, then 
$$\theta_p=\phi^j_i+\frac{\alpha}{a}(t_p-y^j_i).$$
Thus, for every $T>0$ we have
$$\sup\{\theta_p: p \in \tau^1_i, 1 \leq i \leq n_1, -T\leq t_p \leq T\} \leq \max_{1 \leq i \leq n_1} \phi^1_i+O(|\alpha|),$$
$$\inf\{\theta_p: p \in \tau^2_i, 1 \leq i \leq n_2, -T\leq t_p \leq T\} \geq \min_{1 \leq i \leq n_2} \phi^2_i+O(|\alpha|).$$


Let $p \in \Sigma$ and $p^* \in \partial  M^*(g)$ be the point with the same coordinates $( \theta_p, t_p)$.
Due to the bounds above, for sufficiently small $\alpha$, if $p \in \Sigma$, $-T \leq t_p \leq T$ and 
$$\max_{1 \leq i \leq n_1}\phi^1_i+|\alpha|^\mu \leq  \theta_p \leq \min_{1 \leq i \leq n_2}\phi^2_i-|\alpha|^\mu,$$ 
then $p^* \in \Sigma^*$. Thus, Lemma~\ref{gest1} follows from Lemma~\ref{gest1*}.
\end{proof}

\begin{proof}[Proof of Lemma~\ref{gest2}.]
As in the proof of Lemma~\ref{gest1}, for $p \in \Sigma$ let $p^* \in \partial  M^*(g)$ be the point with the same coordinates $( \theta_p, t_p)$. For sufficiently small $\alpha$, if $p \in \Sigma$, $-T \leq t_p \leq T$ and 
$$\max_{1 \leq i \leq n_1}\phi^1_i+|\alpha|^\mu \leq  \theta_p \leq \min_{1 \leq i \leq n_2}\phi^2_i-|\alpha|^\mu,$$ 
then $p^* \in \Sigma^*$, and for such $p$ Lemma~\ref{gest2} follows from Lemma~\ref{gest2*}.

But for $p \in \Sigma$ satisfying $0\leq  \theta_p \leq \max_{1 \leq i \leq n_1}\phi^1_i+|\alpha|^\mu$ we have
$$d_{\H^3}(p, \Pi ^1_1)\leq \sinh(d_{\H^3}(p, \Pi ^1_1))=\sin( \theta_p)\sinh(d_{\H^3}(p,  \psi))\leq \theta_p \sinh(\max(x^j_i))=O(\|\Delta g\|^\mu).$$
Similarly for $p$ such that $\min_{1 \leq i \leq n_2}\phi^2_i-|\alpha|^\mu\leq  \theta_p \leq \phi^2_1$ we have
$$d_{\H^3}(p, \Pi^2_1)\leq\sinh(d_{\H^3}(p, \Pi^2_1))=\sin(\phi^2_1- \theta_p)\sinh(d_{\H^3}(p,  \psi))\leq$$ $$\leq(\phi^2_1-\theta_p) \sinh(\max(x^j_i))=O(\|\Delta g\|^\mu).$$
This finishes the proof.
\end{proof}




Now we are going to prove Lemmas~\ref{gest1*}--\ref{gest2*}. We will need the following preliminaries

\begin{lm}
\label{asymptrap}
Let $p_1p_2q_1q_2$ be a hyperbolic quadrilateral, possibly self-intersecting, such that $\ang p_1q_1q_2=\ang p_2q_2q_1=\pi/2$. Consider a point $p \in p_1p_2$, let $q \in q_1q_2$ be its orthogonal projection to the line $q_1q_2$. Then $$pq \leq 2\max_{j=1,2}\{p_jq_j \exp(-qq_j)\}.$$
\end{lm}

\begin{proof}
%
Let $q'$ be the projective intersection point of the lines $p_1p_2$ and $q_1q_2$, possibly ideal or hyperideal. Assume that $q$ lies between $q'$ and $q_1$. We will show that $$pq \leq 2p_1q_1\exp(-qq_1).$$

If $q'$ is not hyperideal, we can rotate the line $p_1p_2$ around the point $p_1$ until we make $q'$ hyperideal so that $pq$ increases and $q$ stays between $q'$ and $q_1$. Here we mean that the point $q$ stays fixed and the point $p$ is defined as the intersection point of the new line with the perpendicular to $q_1q_2$ at $q$. Hence, we can assume that $q'$ is hyperideal, so the lines $p_1p_2$ and $q_1q_2$ are ultra-parallel. Let $p_0$, $q_0$ be the respective closest points on these lines. By construction, the point $q$ lies between $q_0$ and $q_1$. From Lemmas~\ref{coslaw} and~\ref{sinlaw} we get
$$\tanh(p_0q_0)=\frac{\tanh(p_1q_1)}{\cosh(q_0q_1)}.$$
And also
$$\tanh(pq)=\tanh(p_0q_0)\cosh(qq_0)=\frac{\tanh(p_1q_1)\cosh(qq_0)}{\cosh(q_0q_1)}=$$ 
$$=\tanh(p_1q_1)\exp(-qq_1)\frac{1+\exp(-2qq_0)}{1+\exp(-2q_0q_1)}\leq 2\tanh(p_1q_1)\exp(-qq_1).$$
It remains to note that $pq\leq p_1q_1$, thus $pq\leq \frac{p_1q_1}{\tanh(p_1q_1)}\tanh(pq).$
\end{proof}


\begin{lm}
\label{gener}
Let $\alpha \neq 0$, let $(\phi-\pi) a-\alpha y>0$ and let $p_1p_2$ be a generatrix of $\Sigma^*$, where $p_1 \in \tau_1$, $p_2 \in \tau_2$. Then $\sgn(t_{p_1}-t_{p_2})=\sgn(\alpha)$, $0< \theta_{p_2}- \theta_{p_1} < \pi$ and $$\left|\frac{t_{p_1}-t_{p_2}}{ \theta_{p_2}- \theta_{p_1}}\right| \geq \left|\frac{(\phi-\pi) a-\alpha y}{\alpha\pi}\right|.$$
\end{lm}

\begin{proof}
It is clear that $-\pi<  \theta_{p_2}- \theta_{p_1}<\pi$. Let $p'_1 \in \tau_1$ be a point with $\theta_{p'_1}= \theta_{p_1}-\pi$. Then we have $\sgn(t_{p'_1}-t_{p_2})=\sgn(\alpha)$. Hence, the same holds for $\sgn(t_{p_1}-t_{p_2})$.

We now show that $ \theta_{p_2}> \theta_{p_1}$. Consider the equidistant surfaces from $\psi$. The induced intrinsic metrics on such surfaces are locally Euclidean. Consider now the supporting plane $\Pi$ to $\Sigma^*$ containing the generatrix $p_1p_2$. Its intersection with each equidistant surface to $ \psi$ is an ellipse in the intrinsic metric of this surface. In $\Pi$ all such intersections are concentric figures around a point $p_0 \in \Pi$, which is the closest point from $\Pi$ to $ \psi$. For each such ellipse we call its ``left side'' all the points $p$ with $ \theta_p <  \theta_{p_0}$ and its ``right side'' the set of points with $ \theta_p >  \theta_{p_0}$. It is easy to see that by the choice of $\Sigma^*$ the curve $\tau_1$ is tangent to the corresponding ellipse in its left side, and $\tau_2$ tangents the corresponding ellipse in its right side. Hence, $ \theta_{p_1} >  \theta_{p_0} >  \theta_{p_2}$.

Due to the $\Gamma^*$-invariance of $\Sigma^*$, we can assume that $p_2= v_2$, so $t_{p_2}=y_2$ and $ \theta_{p_2}=\phi_2$. As above, let $p'_1 \in \tau_1$ be the point with $ \theta_{p_2}- \theta_{p_1'}=\pi$. We get $ \theta_{p'_1}=\phi_2-\pi$. Hence, $t_{p_1'}-t_{p_2}=\frac{(\phi-\pi) a}{\alpha}-y$. The chord $p'_1p_2$ is tangent to $\psi$, and in order for the chord $p_1p_2$ to be in the right location with respect to $\psi$, the following bound must hold
$$\left|\frac{t_{p_1}-t_{p_2}}{ \theta_{p_2}- \theta_{p_1}}\right| \geq \left|\frac{t_{p'_1}-t_{p_2}}{ \theta_{p_2}- \theta_{p'_1}}\right|.$$
Our last claim follows. 
\end{proof}

\begin{lm}
\label{gener1}
Let $\alpha \neq 0$, let $(\phi-\pi) a-\alpha y>0$ and let $p \in \Sigma^*$ be such that 
$-T \leq t_p \leq T$, $p_1 \in \tau_1$, $p_2 \in \tau_2$ be the endpoints of the generatrix containing $p$. Then there are constants $c_1, c_2 >0,$ $c_3, c_4 \geq 0$ such that in a neighborhood of $g_0$ we have
$$|t_{p_1}-t_p|\geq \frac{c_1 (\theta_p-\phi_1)}{|\alpha|}-c_2,$$
$$|t_{p}-t_{p_2}|\geq \frac{c_3(\phi_2- \theta_p)}{|\alpha|}-c_4.$$
\end{lm}

\begin{proof}
We have $\theta_{p_1}=\phi_1+\frac{\alpha (t_{p_1}-y_1)}{a}$. On the other hand, from Lemma~\ref{gener} we get
$$|t_{p_1}-t_p|\geq (\theta_p - \theta_{p_1})\frac{(\phi-\pi) a-\alpha y}{\pi|\alpha|}.$$
Substituting the first into the second and rearranging we get
$$|t_{p_1}-t_p|\frac{\phi a-\alpha y}{\pi a}\geq (\theta_p-\phi_1)\frac{(\phi-\pi) a-\alpha y}{\pi|\alpha|}-\sgn(\alpha)(t_p-y_1)\frac{(\phi-\pi) a-\alpha y}{\pi a},$$
from which we obtain
$$|t_{p_1}-t_p|\geq \frac{a (\theta_{p}-\phi_1)}{|\alpha|}\cdot\frac{(\phi-\pi) a-\alpha y}{\phi a-\alpha y}-(T+y_1)\cdot \frac{(\phi-\pi) a-\alpha y}{\phi a-\alpha y}.$$
The first claim follows.

Similarly we have 
$$\theta_{p_2}=\phi_2+\frac{\alpha (t_{p_2}-y_2)}{a}$$
and from Lemma~\ref{gener}
$$|t_p-t_{p_2}|\geq (\theta_{p_2}-\theta_p)\frac{(\phi-\pi) a-\alpha y}{\pi|\alpha|}.$$
From this we get
$$|t_p-t_{p_2}|\frac{\phi a-\alpha y}{\pi a}\geq (\phi_2- \theta_p)\frac{(\phi-\pi) a-\alpha y}{\pi|\alpha|}+\sgn(\alpha)(t_p-y_2)\frac{(\phi-\pi) a-\alpha y}{\pi a}$$
and in turn
$$|t_{p}-t_{p_2}| \geq \frac{a(\phi_2- \theta_p)}{|\alpha|}\cdot\frac{(\phi-\pi)a-\alpha y}{\phi a-\alpha y}-(T+y_2)\cdot \frac{(\phi-\pi)a-\alpha y}{\phi a-\alpha y}.$$
The second claim follows.
\end{proof}

Now we can prove Lemma~\ref{gest1*}.

\begin{proof}[Proof of Lemma~\ref{gest1*}.]

We may assume that $\|\Delta g\|$ is small enough so that $\phi>\pi$ and $$(\phi-\pi) a-\alpha y>0.$$
Let $p$ be a point satisfying the conditions of the lemma. It belongs to a generatrix of $\Sigma^*$, with endpoints $p_1 \in \tau_1$, $p_2 \in \tau_2$. Let $q_1, q_2, q$ be the orthogonal projections of $p_1, p_2, p$ to $ \psi$.  Due to the restrictions on $\theta_p$, from Lemma~\ref{gener1} we get that in a neighborhood of $g_0$, there exist constants $c_1, c_2>0$ such that $$q_1q=|t_{p_1}-t_p|\geq \frac{c_1}{|\alpha|^{1-\mu}},$$ $$qq_2=|t_p-t_{p_2}| \geq \frac{c_2}{|\alpha|^{1-\mu}}.$$ 

Consider the tetrahedron $p_1p_2q_2q_1$. Rotate the segment $p_2q_2$ in the plane orthogonal to $q_1q_2$. We keep $q$ fixed and by $p$ we mean the point at $p_1p_2$ such that its orthogonal projection to $q_1q_2$ is $q$. We see that the length $pq$ maximizes when the tetrahedron degenerates to a planar trapezoid. Indeed, the trajectory of $p$ is a circle in the plane orthogonal to $q_1q_2$ at $q$ centered at the intersection of this plane with the line $p_1q_2$. The farthest point from this circle to $q$ corresponds exactly to the situation when $pq$ becomes coplanar to $p_1q_2$ and lying on the same side of $ \psi$. 

Let $p'_1p'_2q'_2q'_1$ be the trapezoid with $p'_1q'_1=p_1q_1$, $p'_2q'_2=p_2q_2$ and $q_1q_2=q'_1q'_2$, $p'$ be the point at the segment $p'_1p'_2$ with the orthogonal projection $q$ to the line $q'_1q'_2$ such that $q'q'_1=qq_1$, $q'q'_2=qq_2$. From the argument above we have $pq \leq p'q'$. Applying Lemma~\ref{asymptrap} we get 
$$p'q' \leq 2\max_{j=1,2}\left\{x_j\exp \left(-\frac{c_j}{|\alpha|^{1-\mu}}\right)\right\}\leq 2\max_{j=1,2}\left\{\frac{x_j}{c_j^{\frac{2}{1-\mu}}}\right\}|\alpha|^2=O(\|\Delta g\|^2).$$
\end{proof}

\begin{proof}[Proof of Lemma~\ref{gest2*}.]
Consider the first claim of the lemma. Let $p$ be a point satisfying the conditions of the claim. We treat separately the cases $\phi_0>\pi$ and $\phi_0=\pi$. Suppose that $\phi_0>\pi$.
We may assume that $\|\Delta g\|$ is small enough so that $\phi>\pi$ and $$(\phi-\pi) a-\alpha y>0.$$
We have 
$$d_{\H^3}(p, \Pi^1_1)\leq \sinh (d_{\H^3}(p, \Pi^1_1))=\sin (\theta_p) \sinh(d_{\H^3}(p,  \psi))\leq$$ $$\leq \max_{j=1,2}\left\{\frac{\sinh(x_j)}{x_j}\right\}  \theta_p d_{\H^3}(p,  \psi).$$
The point $p$ belongs to a generatrix of $\Sigma^*$, with endpoints $p_1 \in \tau_1$, $p_2 \in \tau_2$. Let $q_1, q_2, q$ be the orthogonal projections of $p_1, p_2, p$ to $ \psi$.  From Lemma~\ref{gener1} we get that in a neighborhood of $g_0$ there exist constants $c_1, c_2>0$ such that $$q_1q=|t_{p_1}-t_p|\geq \frac{c_1}{|\alpha|^{1-\mu}},$$ $$qq_2=|t_p-t_{p_2}| \geq \frac{c_2}{|\alpha|^{1-\mu}}.$$ 

Thus, similarly to the proof of Lemma~\ref{gest1*}, from Lemma~\ref{compare} in this case we can bound
$$d_{\H^3}(p,  \psi)=O(\|\Delta g\|^2),~~~~~d_{\H^3}(p, \Pi^1_1)=O(\|\Delta g\|^2).$$

Now we deal with the case $\phi_0=\pi$. Let $\Pi$ be the plane containing the half-plane $\Pi^1_1$, and $r_1$, $r_2$, $r$ be the orthogonal projections of $p_1$, $p_2$, $p$ to $\Pi$. From Lemma~\ref{asymptrap} we get
\begin{equation}
\label{comp4}
d_{\H^3}(p, \Pi)=pr\leq 2\max_{j=1,2}\{p_jr_j \exp(-rr_j)\}.
\end{equation}

Suppose that $rr_j \geq |\alpha|^{\mu-1}$. Then
$$p_jr_j\exp(-rr_j)\leq x_j|\alpha|^2=O(|\alpha|^2).$$

Suppose that $rr_j \leq |\alpha|^{\mu-1}$. Then also $qq_j \leq |\alpha|^{\mu-1}$ and
$$|t_{p_j}-y_j| \leq T+qq_j+|y_j|\leq T+|\alpha|^{\mu-1}+|y_j| \leq c|\alpha|^{\mu-1}$$
for some constant $c>1$. We have
$$\theta_{p_j}=\phi_j+\frac{\alpha(t_{p_j}-y_j)}{a}.$$
In the case $j=1$ we get
$$|\theta_{p_1}|\leq|\phi_1|+\frac{c|\alpha|^\mu}{a} =O(\|\Delta g\|^{\mu}),$$
$$p_1r_1=d_{\H^3}(p_1, \Pi)\leq \sinh (p_1r_1)=\sin (|\theta_{p_1}|) \sinh(p_1q_1)\leq$$ $$\leq   |\theta_{p_1}| \sinh(x_1)=O(\|\Delta g\|^{\mu}).$$
In the case $j=2$ we get
$$|\theta_{p_2}-\pi|\leq |\phi_2-\pi|+\frac{c|\alpha|^{\mu}}{a}=O(\|\Delta g\|^{\mu}),$$
$$p_2r_2=d_{\H^3}(p_2, \Pi)\leq \sinh (p_2r_2)=\sin (|\theta_{p_2}-\pi|) \sinh(p_2q_2)\leq$$ $$\leq   |\theta_{p_2}-\pi| \sinh(x_2)=O(\|\Delta g\|^{\mu}).$$
In any case we see that
$$p_jr_j\exp(-qq_j)=O(\|\Delta g\|^\mu).$$
Due to~(\ref{comp4}), this finishes the proof. The second claim of the lemma is obtained similarly.
\end{proof}


\subsubsection{Differentiability for controllably polyhedral metrics}

We use the results of the previous section to show

\begin{lm}
\label{diff2}
Let $g \in \mathcal P_{cp}(M, V)$. Then $\mathcal I_V$ is differentiable at $g$.
\end{lm}

\begin{proof}
We start as in the proof of Lemma~\ref{c1}. Pick $g \in \mathcal P_{cp}(M, V)$, lift it to $(\rho, \tilde f) \in \mathcal R(\pi_1(M), G) \times (\H^3)^{V}$, choose a neighborhood $\tilde U$ of $(\rho, \tilde f)$ submersing smoothly onto a neighborhood $U$ of $g$ in $\mathcal P_{cp}(M, V)$ and consider $\mathcal I_V$ as a function over $\tilde U$. Let $d$ be the induced intrinsic metric on $\partial M(g)$.

Let $\psi_1, \ldots, \psi_r$ be all components of $\partial M(g) \cap \partial C(\overline g)$. If $\psi_i$ is an edge of $g$, then by $\Sigma_i$ we denote the union of $\psi_i$ and the two faces adjacent to $\psi_i$. If $\psi_i$ is not an edge of $g$, then by $\Sigma_i$ we denote the face containing $\psi_i$. Pick any triangulation $\mathcal T$ realized in $(d, V)$ such that (1) every edge of $g$ that is not one of $\psi_i$ is an edge of $\mathcal T$; (2) all $\Sigma_i$ are subdivided so that all the interior edges of $\mathcal T$ in $\Sigma_i$ connect two different boundary components of $\Sigma_i$. Then if $U$ is small enough, we get $\mathcal I_V(U) \subset \mathcal D_c(V, \mathcal T)$.

The lengths of the edges of $\mathcal T$ that are edges of $g$ are smooth functions around $g$. For each face of $g$ that is a convex hyperbolic polygon, $\mathcal T$ subdivides it, and the lengths of the edges of $\mathcal T$ in the interior of the faces are $C^1$-functions around $g$ by the same argument as in Lemma~\ref{c1}. The only remaining edges are the edges of $\mathcal T$ belonging to interiors of $\Sigma_i$. For every such edge Corollary~\ref{diff} shows that its length is a differentiable function at~$g$. This finishes the proof.
\end{proof}

\subsubsection{Differentiability for convex cores}
\label{diffcoresec}

Let $V =\emptyset$, so for each $g \in \mc P_b(M, V)$ the set $M(g)$ coincides with $C(\ol g)$. The set $\mc P_b(M, V)$ is identified with a subset of $\mc{CH}(N)$ corresponding to non-Fuchsian metrics on $N$ (recall that it is the whole $\mc{CH}(N)$ except the case when $M$ is an interval bundle over a surface). The map $\mc I_V$ coincides with a map $\mc I$ defined in the introduction, which naturally extends to the whole $\mc{CH}(N)$ when there are Fuchsian metrics.
It was shown by Bonahon in~\cite{Bon}

\begin{lm}
\label{diffcore}
The map $\mc I$ is continuously differentiable over $\mc{CH}(N)$. 
\end{lm}

We will need some additional knowledge on the differential of $\mc I$. 
We recall from the introduction that there is another map $\mc I^*: \mc{CH}(N) \ra \mc{ML}(\pt M)$ sending a convex cocompact metric on $N$ to the bending lamination of its convex core. The space $\mc{ML}(\pt M)$ does not carry a natural smooth structure, but carries a PL-structure. In~\cite{Bon} Bonahon proposed a natural relaxation of the notion of differentiability for PL-structures. Let $U \subset \R^n$ be a domain and $f: U \ra \R^k$ be a map. \emph{The tangent map} to $f$ at $x \in U$, which we denote $df_x$, is a map $\R^n \ra \R^k$ such that for every continuous curve $\tau: [0, \e] \ra U$ with $\tau(0)=x$, differentiable at $0$ with $\tau'(0)=v$, the curve $f\circ\tau$ is differentiable at zero and $df_x(v)=(f\circ\tau)'(0)$. We say that $f$ is \emph{tangentiable} at $x$ if it admits a tangent map at $x$, and $f$ is tangetiable over $U$ if it is tangentiable at every $x \in U$.

By the same construction as for differential manifolds, PL-manifolds admit tangent spaces at each point, which, however, lack the addition structure on the tangent vectors, and allow only the multiplication by non-negative scalars. For an $n$-dimensional PL-manifold, the tangent space at a point is isomorphic to $\R^n$ endowed with such a structure. One can naturally speak about tangentiable maps between PL-manifolds. For a bit more information we refer to~\cite[Section 1]{Bon}. Bonahon proved in~\cite{Bon} that the map $\mc I^*$ is tangentiable.

Let us recall from Section~\ref{cocosec} that $\mc{CH}(N)$ can be thought as a covering space over a smooth open subset of the character space $\mc X(\pi_1(M), G)$, where $G={\rm Iso}^+(\H^3)$. We need to slightly expand this construction. Let $S_i$, ${i=1 \ldots r}$, be the components of $\pt M$. Define $\mc R(\pi_1(S_i), G)$ and $\mc X(\pi_1(S_i), G)$ as in Section~\ref{cocosec} and denote $\prod_{i=1}^r \mc X(\pi_1(S_i), G)$ by $\mc X(\pt M, G)$. Let $[\rho_i] \in \mc X(\pi_1(S_i), G)$ be the characters induced by a representation $\rho_{\ol g} \in \mc R(\pi_1(M), G)$ for $\ol g \in \mc{CH}(N)$. It is known that even when the boundary of $M$ is compressible, $[\rho_i]$ belongs to the smooth part of $\mc X(\pi_1(S_i), G)$, which is locally a complex manifold of complex dimension $-3\chi(S_i)$, see, e.g.,~\cite{Gol2}. Hence, also $\mc X(\pt M, G)$ is locally a complex manifold around the induced characters, and the smooth part of $\mc X(\pi_1(M), G)$ is immersed to $\mc X(\pt M, G)$ as a complex submanifold of half dimension. 

Thurston proposed a well-known parametrization of the space of complex projective structures, which defines a local homeomorphism
$$\mc J: \mc T(\pt M) \times \mc{ML}(\pt M) \ra \mc X(\pt M, G),$$ 
see details in~\cite{KT}. The points of $\mc X(\pt M, G)$ coming from convex cocompact metrics on $M$ are in the image of $\mc J$, and $\mc I \times \mc I^*$ coincides with a local inverse of $\mc J$ on such points. Bonahon showed in~\cite{Bon} that $\mc J$ is tangentiable and that its tangent map is an isomorphism of the tangent spaces at each point. Moreover, he proposed a convenient description of the tangent map.

We need to introduce few definitions from another paper~\cite{Bon4} of Bonahon. An \emph{(abstract) pleated surface} with topological type $S$ (for a closed connected surface $S$) is a pair $(\tilde \iota, \rho)$, where $\tilde\iota: \tilde S \ra \H^3$ is a map from the universal cover $\tilde S$ of $S$ and $\rho: \pi_1(S) \ra G$ is a homomorphism, such that\\
(1) for every $\gamma \in \pi_1(S)$ we have $\tilde \iota\circ\gamma=\rho(\gamma)\circ\tilde\iota$;\\
(2) the pull-back of the intrinsic metric on $\tilde\iota(\tilde S)$ induces a hyperbolic metric $d$ on $S$;\\
(3) there exists a geodesic lamination $\lambda$ in metric $d$ such that $\tilde\iota$ sends each edge of its preimage $\tilde\lambda$ in $\tilde S$ to a geodesic in $\H^3$ and $\tilde\iota$ is totally geodesic on the complement to $\tilde\lambda$.

This notion is different from the common notion of a pleating surface in that the holonomy representation $\rho$ is not required to be faithful or to act freely or properly discontinuously on $\H^3$.

Now let us step back and consider a geodesic lamination on a hyperbolic surface $(S, d)$. An \emph{$A$-valued transverse cocycle} for $\lambda$, where $A$ is an abelian group, is an $A$-valued transverse finitely additive measure for $\lambda$. By $\mc H_A(S, \lambda)$ we denote the space of transverse cocycles for $\lambda$. For $A=\R$ this is a finite-dimensional vector space. In particular, measured laminations and tangent vectors to $\mc{ML}(\pt M)$ with the supports contained in $\lambda$ are $\R$-valued transverse cocycles for $\lambda$. (For the description of tangent vectors to $\mc{ML}(\pt M)$ we refer to~\cite{Bon2, PH}.) In~\cite{Bon4} Bonahon showed that one can associate to a pleated surface an $\R/2\pi\Z$-valued transverse cocycle measuring the bending of the surface. Moreover, he described a converse process of constructing a pleated surface from a hyperbolic surface $(S, d)$ together with an $\R/2\pi\Z$-valued transverse cocycle for $\lambda$. This determines a map
$$\mc T(S) \times \mc H_\R(S, \lambda) \ra \mc X(S, G).$$
(Here we associate to $\R$-valued transverse cocycles their $\R/2\pi\Z$-reductions.) The space $\mc T(S) \times \mc H_\R(S, \lambda)$ can be endowed with a complex structure such that $iT\mc T(S)=T\mc H_\R(S, \lambda)$, where $T\mc T(S)$ and $T\mc H_\R(S, \lambda)$ are considered as sub-bundles of the tangent bundle to $\mc T(S) \times \mc H_\R(S, \lambda)$.
Bonahon showed in~\cite{Bon4} that when $\lambda$ is maximal, the map above is a local biholomorphic homeomorphism. We consider the product of these maps for different components of $\pt M$
$$\mc J_\lambda: \mc T(\pt M) \times \mc H_\R(\pt M, \lambda) \ra \mc X(\pt M, G).$$
In~\cite{Bon} Bonahon proved


\begin{lm}
\label{diffcore2}
Let $(d, \mu) \in \mc T(\pt M) \times \mc{ML}(\pt M)$, $\dot d \in T_{d}\mc T(\pt M)$, $\dot \mu  \in T_\mu\mc{ML}(\pt M)$ and let $\lambda$ be a maximal geodesic lamination containing the supports of $\mu$ and $\dot \mu$. Then $$(d\mc J)_{(d, \mu)}(\dot d, \dot \mu)=(d\mc J_\lambda)_{(d, \mu)}(\dot d, \dot \mu).$$
\end{lm}

\section{Infinitesimal rigidity}
\label{infrigsec}

This section is devoted to a proof of the following infinitesimal rigidity result.

\begin{lm}
\label{infrig}
Let either $g \in \mathcal P_{cp}(M, V)$ or $V=\emptyset$ and $g \in \mc P_p(M, V)$. Then $d\mathcal I_V$ has the full rank at $g$.
\end{lm}

We now show that it implies the corresponding local rigidity results

\begin{crl}
\label{locrig}
Let $g \in {\mathcal P}_{cp}(M, V)$. Then $\mathcal I_V$ is a local homeomorphism at $g$. 
\end{crl}

\begin{proof}
The manifolds $\mathcal P_b(M, V)$ and $\mc D_c(\pt M, V)$ have the same dimension $3(n-k)$ and $\mc P_{cp}(M, V)$ is an open subset of the former. Since $\mc I_V$ is differentiable over $\mc P_{cp}(M, V)$ and the differential is everywhere non-degenerate, by the inverse function theorem for differentiable maps (see, e.g.,~\cite{Ray}) we get that $\mathcal I_V$ is locally injective. By the Brower invariance of the domain theorem, $\mathcal I_V$ is a local homeomorphism around $g$. 
\end{proof}

\begin{crl}
\label{locrig2}
Let $V =\emptyset$ and $g \in {\mathcal P}_{p}(M, V)$. Then $\mathcal I_V$ is a local $C^1$-diffeomorphism at~$g$. 
\end{crl}

\begin{proof}
The proof is almost the same, but here we use that $\mc I_V$ is $C^1$ over $\mc P_b(M, V)$. Hence, it is non-degenerate in a neighborhood of $g$ and we can use the standard inverse function theorem to make the desired conclusion.
\end{proof}

Particularly, we get the first claim of Part (1) of Corollary~\ref{main2}. We can also deduce Part (2).

\begin{crl}
\label{locrig3}
Let $\ol g \in \mc{CH}(N)$ have polyhedral convex core. Then the tangent map $d\mc I^*_{\ol g'}$ is an isomorphism for every $\ol g'$ sufficiently close to $\ol g$. 
\end{crl}

\begin{proof}
Due to the invariance of domain, it is enough to show that $d\mc I^*_{\ol g'}$ is injective. We will rely only on that $d\mc I_{\ol g'}$ is non-degenerate, which is true in a neighborhood of $\ol g$ since $\mc I$ is $C^1$. Thus, it is enough to prove the lemma for $\ol g'=\ol g$. In what follows we won't be using the subscript $_{\ol g}$ for tangent maps meaning that this is the case.

Denote $\mc I \times \mc I^*(\ol g)$ by $(d, \mu)$. First let us show that $d\mc I^*$ is non-degenerate.
Suppose the converse, that there exists $\dot {\ol g}\neq 0$ such that $d\mc I^*(\dot {\ol g})=0$. 
Define 
$$P:=d\mc J(T_d\mc T(\pt M) \times\{0\}) \subset T_{\ol g}\mc X(\pt M, G).$$
Due to Lemma~\ref{diffcore2}, $P$ is a linear subspace and $iP$ is transversal to $P$. We have $\dot{\ol g} \in P$. Let $\mc J^{-1}$ be a local inverse for $\mc J$ sending $\ol g$ to $(d, \mu)$. Then $d\mc J^{-1}(i\dot{\ol g}) \in  \{0\} \times T_\mu\mc{ML}(\pt M)$. As $\mc X(\pi_1(M), G)$ is a complex submanifold of $\mc X(\pt M, G)$, we have $i\dot{\ol g}\in T_{\ol g}\mc X(\pi_1(M), G)$. We get $d\mc I(i\dot{\ol g})=0$, which contradicts to Corollary~\ref{locrig2}.



Now suppose that for $\dot{\ol g}_1, \dot{\ol g}_2 \in T_{\ol g}\mc{CH}(N)$ we have $d\mc I^*(\dot{\ol g}_1)=d\mc I^*(\dot{\ol g}_2)=\dot \mu$. Let $\lambda$ be a maximal geodesic lamination containing the supports of $\mu$ and $\dot \mu$, and let $\mc J_\lambda^{-1}$ be a local inverse for the map $\mc J_{\lambda}$ sending $\ol g$ to $(d, \mu)$. Due to Lemma~\ref{diffcore2}, we have $d\mc J_\lambda^{-1}(\dot{\ol g}_1)=(\dot d_1, \dot \mu)$, $d\mc J_\lambda^{-1}(\dot{\ol g}_2)=(\dot d_2, \dot \mu)$ for $\dot d_1, \dot d_2 \in T_d\mc T(\pt M)$. Hence, $d\mc J_\lambda^{-1}(\dot{\ol g}_1-\dot{\ol g}_2)=(\dot d_1-\dot d_2, 0)$. Due to Lemma~\ref{diffcore2}, we get $d\mc I^*(\dot{\ol g}_1-\dot{\ol g}_2)=0$, which contradicts to our first claim. 
\end{proof}

In order to pass from local rigidity to global rigidity we will need the following

\begin{prop}
\label{prop6}
Let $g_0 \in \mathcal P_{b}(M, V)$ be such that $\mc I_{V}$ is a local homeomorphism at~$g_0$. Let $V \subseteq W$, $g \in \ol{\mathcal P}_{b}(M, W)$ be a lift of $g_0$ and let $d:=\mc I_W(g)$. Then $g$ is isolated in~$\mc I_W^{-1}(d)$. 
\end{prop}

\begin{proof}
If $V=W$, then the claim is straightforward, so suppose that $V\subsetneq W$. By $\ol {\mc P}_b(M, V, W)$ we denote the subset of $\ol {\mc P}_b(M, W)$ of metrics $g'$ with $V(g')=V$, by $\ol {\mc D}_c(M, V, W)$ we denote the subset of $\ol {\mc D}_c(M, W)$ of metrics $d$ with $V(d)=V$. Consider the commutative diagram
\begin{center}
\begin{tikzcd}
\ol {\mc P}_b(M, V, W) \arrow[r, "\mc I_W"] \arrow[d]
& \ol {\mc D}_c(M, V, W) \arrow[d] \\
{\mathcal P}_{cp}(M, V) \arrow[r, "\mc I_{V}"]
& {\mathcal D}_{cp}(M, V)
\end{tikzcd}
\end{center}
Since the bottom arrow is a local homeomorphism around $g_0$, which is the vertical image of $g$, and by definition $\mc I_{W}$ is a homeomorphism when restricted to a fiber of the vertical map, it is easy to deduce that the restriction of $\mc I_W$ to $\ol{\mc P}_b(M, V, W)$ is locally injective. Because $\mc I_W^{-1}(d) \subset \ol{\mc P}_b(M, V, W)$, the claim follows.
\end{proof}

\subsection{Infinitesimal substitution}
\label{subst}

Let $g \in \mathcal P_{p}(M, V)$ and $\Psi$ be the set of all edges $\psi \subset \partial M(g)$ that are simple closed geodesics (note that there are finitely many of them). We orient all these geodesics arbitrarily. Let $V'$ be a finite set in bijection with $\Psi$. Consider a smooth section $\sigma$ of the bundle $\mc{CH}(N, V') \ra \mc{CH}(N)$ such that for every $\ol g \in \mc{CH}(N)$ and every $v \in V'$ the point marked by $v$ in $N(\ol g)$ belongs to the closed geodesic homotopic to $\psi$, where $\psi$ is the element of $\Psi$ corresponding to $v$. It is not hard to see that such a section indeed can be chosen smooth since for every $\gamma \in \pi_1(M)$ the axis fixed by $\rho(\gamma)$ varies smoothly as $\rho$ varies in $\mc R(\pi_1(M), G)$. By $W$ we denote the union $V \cup V'$.

By Lemmas~\ref{triang1} and~\ref{triang2} there exists a finite cover $M^c(g^c)$ of $M(g)$ admitting a geodesic triangulation $\mathcal T$ with vertices in $W^c$, where $W^c$ is the preimage of $W$.
Taking a cover provides a smooth immersion of $\mathcal P_{b}(M,V)$ to $\mathcal P_{b}(M^c, V^c)$, and if we take $\dot g \in T_g \mathcal P_{b}(M,V)$ and consider the corresponding vector $\dot g^c \in T_{g^c} \mathcal P_{b}(M^c,V^c)$, then $d\mathcal I_V(\dot g)=0$ is equivalent to $d\mathcal I_{V^c}(\dot g^c)=0$. Hence, if we show that $d\mathcal I_{V^c}(\dot g^c)$ implies $\dot g^c=0$, we also obtain that $d\mathcal I_V(\dot g)$ implies $\dot g=0$. We extend the section $\sigma$ to a smooth section $\sigma^c$ of $\mc{CH}(N^c, W^c) \ra \mc{CH}(N^c)$ with the same property that every point marked by an element of $W^c \backslash V^c$ belongs to the corresponding geodesic. Until the end of this section we will work with $M^c(g^c)$, so to simplify the reading we now omit to write $^c$ and assume that $M(g)$ has a large triangulation $\mathcal T$ with vertices in $W$ (now there can be several new vertices at each $\psi \in \Psi$, but we continue to denote by $V'$ the difference $W \backslash V$). We write $E(\mathcal T)$ for the set of edges of $\mathcal T$ and divide them in two groups: $E_{in}(\mathcal T)$ are the interior edges and $E_{bd}(\mathcal T)$ are the boundary edges.

We take the direct product $\mc P_b(M, V) \times \R^{V'}$ and immerse it in $\mc{CH}(N, W)$ as follows. Denote the elements of of $\R^{V'}$ by $z$ and their components by $z_v$ for $v \in V'$. For $\ol g \in \mc{CH}(N)$ denote $\sigma(\ol g)\in \mc{CH}(N, V')$ by $[g_{\sigma(\ol g)}, f_{\sigma(\ol g)}].$ For every $([g, f], z) \in \mc P_b(M, V) \times \R^{V'}$ we map it to $[g', f'] \in \mc{CH}(N, W)$ such that $[g', f']$ is a lift of $[g, f]$ and for every $v \in V'$ the marked point $f'(v)$ is the point on the corresponding geodesic at oriented distance $t_v$ from $f_{\sigma(\ol g)}(v)$. We identify $\mc P_b(M, V) \times \R^{V'}$ with its image and write for simplicity its elements as $(g,z)$. For each $(g, z) \in \mc P_b(M, V) \times \R^{V'}$ and for each edge $e$ of $\mathcal T$ we consider the length of the unique geodesic segment in $M(g)$ with the same endpoints as $e$ and in the same homotopy class (with fixed endpoints). This defines a smooth map
$$l: \mc P_b(M, V) \times \R^{V'} \rightarrow \R^{E(\mathcal T)}.$$
We now return to our initial metric $g$ and consider the differential $dl$ of the above map at the point $(g, 0)$.

\begin{lm}
\label{subst1}
If there exists $\dot z \in T_0 \R^{V'}$ such that $dl(\dot g, \dot z)=0$, then $\dot g=0$.
\end{lm}

\begin{proof}
It is easy to see that $dl(\dot g, \dot z)=0$ implies that the induced deformation of $\overline g$ is trivial. Indeed, pick arbitrary $v \in W$. Every element $\gamma \in \pi_1(M)$ can be represented by a simplicial loop based at $v$ in the 1-skeleton of $\mathcal T$. Pick a cyclic sequence of tetrahedra along this loop based on a tetrahedron $T$ containing $v$ so that any two subsequent tetrahedra have a 2-face in common, and develop this sequence in $\H^3$. Let $\tilde v_1, \tilde v_2 \in \H^3$ be the lifts of $v$ that are the endpoints of the lifted loop, $\tilde T_1$ and $\tilde T_2$ be the starting and the ending lifts of $T$ so that $\tilde T_i$ contains $\tilde v_i$. The image $\rho_{\overline g}(\gamma)$ can be reconstructed by its action on any point in $\H^3$ together with a basis in the tangent space of this point. Take $\tilde v_1$ and the basis given by the edges of $\tilde T_1$. As $dl(\dot g, \dot z)=0$, we get that $\tilde v_2$ and $\tilde T_2$ infinitesimally do not move and so the induced change on $\rho_{\overline g}(\gamma)$ is trivial.

Hence, $\dot g$ is vertical with respect to the fibration $\mathcal{CH}(N, V) \rightarrow \mathcal{CH}(N)$. Thus, it represents an infinitesimal change of the locations of the points of $V$. Consider the universal cover $\tilde M(g) \subset \H^3$. Then $\dot g$ can be represented by a vector field on $\tilde V=\tilde V(g)$ equivariant with respect to $\rho_{\overline g}$. We extend this vector field to the full preimage $\tilde W$ of $W$ using $\dot z$. As $dl(\dot g, \dot z)=0$, it extends to a Killing field on each 2-face of $\mathcal T$ belonging to the boundary of $\tilde M(g)$ such that the Killing fields coincide on the edges. Moreover, as $dl(\dot g, \dot z)=0$, this vector field preserves the link of each vertex, thus it is the same Killing field on every connected component of $\partial \tilde M(g)$. Because it needs to be equivariant, it must be zero. Hence, $\dot g=0$.
\end{proof}

We now show

\begin{lm}
\label{subst2}
Let $g \in \mc P_{cp}(M, V)$. If $d\mathcal I_V(\dot g)=0$, then there exists $\dot z \in T_0 \R^{V'}$ such that for each $e \in E_{bd}(\mc T)$ we have $dl_e(\dot g, \dot z)=0$
\end{lm}

\begin{proof}
If both endpoints of $e$ are in $V$ and $e$ is an edge of $g$, then the statement is clearly true for any $\dot z$. If $e$ is not an edge and belongs to a simply connected face, the statement is again true for any $\dot z$, which follows from Claim~\ref{differential} that the differential of the length of the realization of $e$ in the intrinsic metric of the boundary coincides with the differential of the extrinsic distance between its endpoints.

Consider a simple closed geodesic $\psi \in \pt M(g)$, let $\gamma \in \pi_1(M)$ be a generating element for a subgroup of $\pi_1(M)$ corresponding to $\psi$, let $F_1$ and $F_2$ be two faces of $\partial M(g)$ adjacent to $\psi$ (possibly degenerating to one) and let $\Sigma$ be their union. Develop the universal cover $\tilde \Sigma$ of $\Sigma$ to $\H^3$ and consider the coordinate system from Subsection~\ref{stripssec}. By $\tilde \psi$ denote the development of $\psi$. Let $g_t$, $t \in [0,\e]$ for some small $\e>0$, be a path tangent to $\dot g$. We consider a corresponding path of $\rho_{\ol g_t}(\gamma)$ and the trajectories of the vertices adjacent to $\Sigma$ and develop them in our coordinate system. By $V'_\psi$ denote the subset of $V'$ belonging to $\psi$. For each $v \in V'_\psi$ choose a lift $\tilde v \in \tilde \psi$ and by $T_{\tilde \psi}$ denote the subspace of $\Pi_{v \in V'_\psi} T_{\tilde v}\H^3$ consisting of tuples of tangent vectors parallel to $\tilde \psi$. The section $\sigma$ determines trajectories for points $\tilde v$. Considering the derivatives at $t=0$ we obtain an element of $T_{\tilde \psi}$. It is easy to see that $T_0\R^{V'_\psi}$ is in one-to-one correspondence with $T_{\tilde\psi}$ (more exactly, there is a natural non-degenerate affine map between them). Let $\dot z_\psi \in T_0\R^{V'_\psi}$ be the preimage of the zero vector.

The deformation $\dot g$ induces the deformation of the initial coordinates such that the deformation of the induced metric on $\Sigma$ is zero. Corollary~\ref{diff} shows that if we set all $\dot\phi^j_i=0$ and $\dot \alpha=0$, then the obtained deformation of $\Sigma$ is still zero. But this implies that all $\dot x^j_i=0$, all $\dot y^j_i=0$ and $\dot a=0$. Together with zero vectors chosen at each $\tilde v$, this  induces zero infinitesimal deformation on every geodesic segment of $\mathcal T$ in $\Sigma$. Let $\dot z \in T_0\R^{V'}$ be the vector obtained from gluing all $\dot z_\psi$. We have $dl_e(\dot g, \dot z)=0$ for all $e \in E_{bd}(\mc T)$.
\end{proof}

\begin{lm}
\label{subst2.5}
Let $V=\emptyset$ and $g \in \mc P_{p}(M, V)$. If $d\mathcal I_V(\dot g)=0$, then there exists $\dot z \in T_0 \R^{V'}$ such that for each $e \in E_{bd}(\mc T)$ we have $dl_e(\dot g, \dot z)=0$.
\end{lm}

\begin{proof}
As described in Section~\ref{diffcoresec}, there is an immersion from $\mc P_b(M, V)$ to the smooth part of $\mc X(\pt M, G)$ and we identify $g$ with its image. Let $S_i$, $i=1\ldots r$, be the components of $\pt M$. By $V'_i$ we denote the subset of $V'$ corresponding to geodesics belonging to $S_i$ in $M(g)$. As in Section~\ref{bentsec}, we define the fibrations 
$$\mc X(\pi_1(S_i), G, V'_i) \ra \mc X(\pi_1(S_i), G)$$
and by $\mc X(\pt M, G, V')$ we denote the product $\prod_{i=1}^r\mc X(\pi_1(S_i), G, V'_i)$, which fibers over $\mc X(\pt M, G)$. Let $U_0$ be a small enough neighborhood of $g$ in $\mc X(\pt M, G)$ and $U'_0$ be its preimage in  $\mc X(\pt M, G, V')$. We consider a smooth section $\sigma'$ of the fibration $U'_0 \ra U_0$ extending the section $\sigma$ such that again the marked points belong to the respective geodesics. We consider the product $U_0 \times \R^{V'}$ and use $\sigma'$ to immerse it in $U'_0$ in the same way as above we immersed $\mc P_b(M, V) \times \R^{V'}$ in $\mc{CH}(N, W)$. We also define a map $l_{bd}: U_0 \times \R^{V'} \ra \R^{E_{bd}(\mc T)}$ in the same way as we defined it above, though now we care only about the boundary edges of $\mc T$.

We now use the maps $\mc J$ and $\mc J_\lambda$ from Section~\ref{diffcoresec}. Let $\mc J^{-1}$ be a local inverse to $\mc J$ defined on $U_0$, provided that $U_0$ is sufficiently small, and coinciding with $\mc I \times \mc I^*$ on $U_0 \cap \mc X(\pi_1(M), G)$. Define 
$(d, \mu):=\mc J^{-1}(g) \in \mc T(\pt M) \times \mc{ML}(\pt M)$. As $d\mc J$ at $(d,\mu)$ is an isomorphism of the tangent spaces and $d\mc I_V(\dot g)=0$, we have $$d\mc J^{-1}(\dot g) \in \{0\} \times T_\mu\mc{ML}(\pt M).$$ 
We are going to prove that for every $\dot \mu \in T_\mu \mc{ML}(\pt M)$ there exists $\dot z \in T_0 \R^{V'}$ such that $dl_{bd}(d \mc J(0, \dot \mu), \dot z)=0$. This will imply the claim of the lemma.

We first suppose that the support of $\dot\mu$ is rational. Let $\lambda$ be a maximal rational geodesic lamination containing the supports of $\mu$ and $\dot \mu$. Consider a local inverse $\mc J_\lambda^{-1}$ to the map $\mc J_\lambda$ defined over $U_0$ and sending $g$ to $(d,\mu)$.
We identify the vector space $\mc H_\R(\pt M, \lambda)$ with its tangent space and consider a curve $\mu+t\dot\mu \in \mc H_\R(\pt M, \lambda)$ for $t \in [0, \e]$ and some $\e>0$. 
Take a component $S_i$ endowed with the hyperbolic metric $d$ induced from $M(g)$, we consider it now as abstract hyperbolic surface triangulated by the restriction of $\mc T$ to $S_i$.

Look at the pleated surfaces $(\tilde\iota_t, \rho_t)$ in $\H^3$ obtained from $(S_i, d)$ by pairs $(d, \mu+ t\dot\mu)$. The starting surface is a boundary component of $\pt \tilde C(\ol g)$. Let us fix a base point on it, together with a basis in the tangent space to the base point. We can choose $(\tilde\iota_t, \rho_t)$ so that they fix the base point and the chosen basis. Since $\lambda$ is rational, it is easy to understand what happens with it: the surface starts to bend along the geodesics in $\tilde\iota_0(\tilde\lambda)$, which is a discrete set, with the speeds of the bending angles prescribed by $\dot \mu$. 

Pick an edge $e$ of $\mc T$ in $S_i$, let $\tilde e$ be its lift to $(\tilde S_i, d)$ and $p_0, \ldots, p_s$ be the intersection points of $\tilde e$ with $\tilde \lambda$ including the endpoints. The curves $\tilde\iota_t(p_j)$ are differentiable at zero. We claim that the derivative of the hyperbolic distance between $\tilde\iota_t(p_1)$ and $\tilde\iota_t(p_s)$ is zero at $t=0$. Indeed, let $x_j$ be the tangent vector at $\tilde\iota_0(p_j)$ to its trajectory, and $x^k_j$ be the result of the parallel translation of $x_j$ to $\tilde\iota_0(p_k)$ along $\tilde\iota_0(\tilde e)$ with respect to the metric connection of $\H^3$. Since the length of each segment between $\tilde\iota_0(p_j)$ and $\tilde \iota_0(p_{j+1})$ does not change infinitesimally, we have that $x^{j+1}_j-x_j$ is orthogonal to $\tilde\iota_0(\tilde e)$. Hence $x^s_1-x_s$ is orthogonal to $\tilde\iota_0(\tilde e)$ and the claim follows.

The trajectories $\tilde\iota_t(p_1)$ and $\tilde\iota_t(p_s)$ determine tangent vectors $\dot z_{v_1}, \dot z_{v_2} \in T_0\R^{V'}$ in the components corresponding to the endpoints $v_1$ and $v_2$ of $e$. The obtained vectors are independent on the choice of $\tilde e$ or $e$. By determining the tangent vector for every element of $V'$ we obtain a tangent vector $\dot z \in T_0\R^{V'}$. For each edge $e \in E_{bd}(\mc T)$ the deformation $dl_{bd, e}(d\mc J_\lambda(0, \dot \mu), \dot z)$ coincides with the deformation of the distance between $\tilde\iota_0(p_1)$ and $\tilde\iota_0(p_s)$ in the construction above. Hence, $dl_{bd}(d\mc J_\lambda(0, \dot \mu), \dot z)=0$. By Lemma~\ref{diffcore2} we get $dl_{bd}(d\mc J(0, \dot \mu), \dot z)=0$ as desired. 

If the support of $\dot \mu$ is irrational, we can approximate it by tangent vectors $\dot \mu_j$ whose supports are rational. By the previous argument, for every $j$ there exists $z_j \in T_0 \R^{V'}$ such that $dl_{bd}(d\mc J(0, \dot \mu_j), \dot z_j)=0$. This means that for every $\dot z \in T_0\R^{V'}$ the vector $dl_{bd}(d\mc J(0, \dot \mu_j), \dot z)$ belongs to the subspace $dl_{bd}(\{0\}\times T_0\R^{V'})$. By continuity of the tangent map, the same is true for $dl_{bd}(d\mc J(0, \dot \mu), \dot z)$. Hence, there exists $\dot z \in T_0\R^{V'}$ such that $dl_{bd}(d\mc J(0, \dot \mu), \dot z)=0$. This finishes the proof.
\end{proof}

There is a neighborhood of $l=l(g, 0)$ in $\R^{E(\mathcal T)}$ where the edge lengths induce a hyperbolic cone-3-manifold structure on $M$ triangulated with $\mathcal T$. Recall that for each edge $e \in E(\mathcal T)$ by $\nu_e$ we denote the total dihedral angle of $e$. Clearly, every $\dot g \in T_g \mathcal P(M, V)$ and $\dot z \in T_0\R^{V'}$ induce a deformation $\dot l=dl(\dot g, \dot z)$ such that for every $e \in E_{in}(\mathcal T)$ the infinitesimal change $d\nu_e(\dot l, \dot z)=0$. Now our aim is to show

\begin{lm}
\label{subst3}
Let $\dot l \in T_l \R^{E(\mathcal T)}$ be such that for each $e \in E_{bd}(\mathcal T)$ we have $\dot l_e=0$ and for each $e \in E_{in}(\mathcal T)$ we have $d\nu_e(\dot l)=0$. Then $\dot l=0$.
\end{lm}

It is clear that the combination of Lemmas~\ref{subst1}--\ref{subst3} imply Lemma~\ref{infrig}. We prove Lemma~\ref{subst3} in the next subsections.

\subsection{Affine automorphic vector fields}

Consider the Klein model of $\H^3$ and denote by $AF$ the set of affine vector fields restricted to $\H^3$. For every non-degenerate tetrahedron $T$ with a vertex set $V(T)$ there is a canonical isomorphism
$$J_T: \prod_{v \in V(T)} T_v \H^3 \ra AF,$$
mapping four tangent vectors at each vertex of $T$ to the unique affine vector field coinciding with these vectors at $V(T)$. This isomorphism is preserved under isometries of $\H^3$.
Note that $AF$ has dimension 12 and contains the 6-dimensional subspace $KF$ of Killing fields.

We look at the action of $AF$ on the edge lengths of $T$. Consider the 6-tuple of its edge lengths as an element $l \in \R^{E(T)}$. There is a natural linear map $AF \rightarrow T_l \R^{E(T)}$ sending a vector field to the induced infinitesimal change of the edge lengths. As hyperbolic tetrahedra are infinitesimally rigid, the kernel of this map exactly equals to $KF$, hence the map is surjective and produces an isomorphism
\begin{equation}
\label{tetriso}
T_l \R^{E(T)} \cong AF / KF.
\end{equation}

Consider now a (possibly non-convex) polyhedron $P \subset \H^3$ with a geodesic triangulation $\mathcal T$ (without new vertices). Denote all the tetrahedra of $\mc T$ by $T_1, \ldots, T_r$. Pick an open cover of $P$ by sets $U_1, \ldots, U_r$ such that for each $i$ we have $T_i \subset U_i$ and $U_i$ does not contain vertices of $P$ except the vertices of $T_i$. Choose a partition of unity over $P$ with respect to $U_i$, i.e., smooth functions $\lambda_i: \H^3 \ra [0,1]$, $i=1,\ldots,r$, such that the support of $\lambda_i$ is in $U_i$ and for every $p \in P$ we have $\sum_i \lambda_i(p)=1$. We denote the set of pairs $(U_i, \lambda_i)$ by $\mc U$. By $AF(P, \mc U)$ we denote the space of vector fields of the form $\sum_i \lambda_i X_i$, where $X_i \in AF$ and if $T_i$ and $T_j$ are two tetrahedra adjacent by a 2-face, then the restriction of $X_i$ to the plane containing this face coincides with the restriction of $X_j$. We call such fields \emph{affine fields on $P$ with respect to $\mc U$}. The meaning of this construction is that we construct a vector field on $P$ by assigning an affine field $X_i$ to each tetrahedron $T_i$ so that the assigned fields coincide on the common faces. However, the resulting field will be continuous, but not smooth, so we smooth it out with the help of a partition of unity. We do it in a way that the values of the smoothed vector field does not change at the vertices of $P$.

If $V$ is the set of vertices of $P$, there is a natural isomorphism
$$J_P: \prod_{v \in V} T_v \H^3 \ra AF(P, \mc U)$$
defined by $$J_P(Y)=\sum_i \lambda_i J_{T_i}(Y|_{V(T_i)}),$$ where $Y \in \prod_{v \in V} T_v \H^3$ and $V(T_i)$ is the vertex set of $T_i$. Note that for every $v \in V$ and $Y \in \prod_{v \in V} T_v \H^3$ we have $J_P(Y)|_v=Y|_v$. There is the 6-dimensional subspace $KF(P, \mc U) \subset AF(P, \mc U)$ consisting of the fields $\sum_i \lambda_i X_i$ such that all $X_i$ are equal to the same Killing field $X$, so the restriction of an element of $KF(P, \mc U)$ to $P$ is exactly $X$.


Similarly to the consideration above, we regard the tuple of edge lengths of $\mathcal T$ as a tuple $l \in \R^{E(\mathcal T)}$ and look at the natural map $AF(P, \mathcal U) \rightarrow T_l\R^{E(\mathcal T)}$. 
Again, its kernel is equal $KF(P, \mc U)$.
We now consider cone structures on $P$ induced by deformations of $l$ in $\R^{E(\mathcal T)}$. Denote the set of the interior edges of $\mc T$ by $E_{in}(\mc T)$. For every $e \in E_{in}(\mathcal T)$ the total dihedral angle $\nu_e$ is an analytic function in a neighborhood of $l$ in $\R^{E(\mathcal T)}$. We denote
$$T^P_l \R^{E(\mathcal T)}=\left\{\bigcap_{e \in E_{in}(\mathcal T)}\ker d\nu_e \right\}.$$
Clearly, the image of $AF(P, \mathcal U)$ belongs to $T^P_l \R^{E(\mathcal T)}$. We now prove

\begin{lm}
\label{isom1}
$$T^P_l \R^{E(\mathcal T)} \cong AF(P, \mathcal T)/KF(P, \mc U).$$
\end{lm}

\begin{proof}
We need to show that $T^P_l\R^{E(\mc T)}$ is the image of our map. Pick $\dot l \in T_l^P\R^{E(\mathcal T)}$. Start from an arbitrary tetrahedron $T_1$ of $\mathcal T$. As described above, there is always an affine field inducing the restriction of $\dot l$ to $T_1$, we pick one of them. Now let $T_2$ be a tetrahedron adjacent to $T_1$ by a 2-face. We try to find a tangent vector at the fourth vertex $v$ of $T_2$ such that the induced infinitesimal lengths-change is the restriction of $\dot l$ to $T_2$. Consider the hyperboloid model of $\H^3$ in $\R^{3,1}$. Denote the other three vertices of $T_2$ by $v_i$, $i=1,\ldots,3$, and denote the edge from $v_i$ to $v$ by $e_i$. We have the system of equations
$$\langle \dot v, v_i \rangle+\langle \dot v_i, v \rangle=-\dot l_{e_i}\cosh l_{e_i},~~~~~i=1, \ldots, 3.$$
Here we identify $v_i$, $v$ with their radius-vectors in $\R^{3,1}$, $\dot v_i \in T_{v_i}\H^3$ are the chosen tangent vectors to $v_i$ and $\dot v$ is a vector that we want to find in $T_v \H^3$. Suppose that there is no solution. Then from linearity there exists a non-trivial solution to the system
$$\langle \dot v, v_i \rangle=0,~~~~~i=1, \ldots, 3.$$
Due to the infinitesimal rigidity of tetrahedra, such a solution $\dot v \in T_v\H^3$ together with the assignment of zero vectors in $T_{v_i}\H^3$ is the restriction of a global Killing field of $\H^3$, but clearly such a non-zero Killing field does not exist. 

We continue this procedure. It remains to check what happens when we close a cycle of tetrahedra around an interior edge $e$ of $\mathcal T$. Let $T$ be the last tetrahedron to close the chain and $e'$ be its last edge. The previous choices of tangent vectors at the vertices of $T$ automatically induce an affine vector field on $T$, which induces an infinitesimal change of the length $l_{e'}$. We need to see that it coincides with the one from $\dot l$. But this follows from $\frac{\partial \nu_{e}}{\partial l_{e'}} \neq 0$, as this implies that for any assignment of infinitesimal lengths-changes to all edges of tetrahedra adjacent to $e$ except $e'$ there exists a unique choice of $\dot l_{e'}$ such that $\dot \nu_e=0$. This finishes the proof. 
\end{proof}

Now let $(K, \Sigma, \mathcal T)$ be a closed triangulated hyperbolic cone-3-manifold, $L$ be its subset obtained by deleting from $K$ the singular locus $\Sigma$. By $E_0(\mc T)$ denote the set of edges not belonging to $\Sigma$, so with $\nu_e=2\pi$. Let $T_1, \ldots, T_r$ be the tetrahedra of $\mc T$. 
We need to show that there exists some special open cover of $L$.

\begin{lm}
\label{cover}
Let $\mc T$ be large. There exists an open cover of $L$ by sets $U_1, \ldots, U_r$ such that 
\begin{itemize}
\item each $U_i$ is simply-connected;
\item $U_i \supset (T_i \cap L)$;
\item the closure of each $U_i$ in $K$ does not contain vertices of $\mc T$ except the vertices of $T_i$.
\end{itemize}
\end{lm}

\begin{proof}
Pick a tetrahedron $T_i$ and develop it to $\H^3$, we denote the image by $\hat T$ and the set of its vertices by $\hat V$. Let us call the edges of $\hat T$ coming from an edge in $E_0(\mc T)$ \emph{important}. For each important edge of $\hat T$ develop all tetrahedra adjacent to the preimage of this edge. For every 2-face of $\hat T$ not adjacent to important edges develop also the other tetrahedron adjacent to its preimage. Enumerate the obtained tetrahedra except $\hat T$ by $\hat T_1, \ldots, \hat T_s$ and denote their set by $\hat{\mc T}$. Note that as $\mc T$ is large, among the tetrahedra of $\hat{\mc T}$ there is no other copy of $T_i$, but some other tetrahedra of $\mc T$ could be developed more than once. Denote the barycenter of $\hat T_j$ by $p_j$. Also for each 2-face of $\hat T_j$ that is not a 2-face of $\hat T$, but contains an important edge, mark the barycenter of this face and denote the set of marked points by $\hat S$. In each $\hat T_j$ by $\hat U_j$ we denote the interior of the convex hull of $p_j$, $\hat S \cap \hat T_j$ and $\hat V \cap T_j$. We are now constructing the set $\hat U$ as follows
\begin{itemize}
\item add all $\hat U_j$ to $\hat U$;
\item for each face $F$ that is not a face of $\hat T$, but is adjacent to an important edge, we add to $\hat U$ the relative interior of the triangle formed by $\hat S \cap F$ and $\hat V \cap F$;
\item we add to $\hat U$ the relative interior of every important edge;
\item we add to $\hat U$ the interior of $\hat T$.
\end{itemize}
By $U_i \subset L$ we denote the preimage of the set $\hat U$. It is not hard to see from the construction that $U_i$ is simply-connected and open in $L$. Doing this for every $T_i$ we obtain an open cover of $L$ satisfying the desired properties.
%
\end{proof}

Take an open cover $U_1, \ldots, U_r$ of $L$ from Lemma~\ref{cover} and fix a partition of unity of $L$ with respect to this cover. Denote the corresponding partition functions by $\lambda_i: L \ra [0,1]$. We denote the set of pairs $(U_i, \lambda_i)$ by $\mc U$, but we will also abuse the notation and sometimes will say just $U_i \in \mc U$. Consider a developing map $dev: \tilde L \rightarrow \H^3$. Let $\tilde{\mc U}$ be the partition of unity of $\tilde L$ that is the lift of $\mc U$. By $I$ we denote a countable set indexing elements of $\mc U$. By $AF(\tilde L, \tilde{\mc U})$ we denote the space of vector fields on $\tilde L$ of the form $\sum_{i \in I} \tilde \lambda_i X_i$ where $\tilde \lambda_i$ are the partition functions and $X_i$ are pull-backs on $\tilde L$ of affine vector fields on $\H^3$ via $dev$ such that for every two tetrahedra $\tilde T_i$ and $\tilde T_j$ in $\tilde L$ adjacent by a 2-face the fields $X_i$ and $X_j$ coincide on this face. We call such fields \emph{affine vector fields on $\tilde L$ with respect to $\mc U$}. It is easy to see that $AF(\tilde L, \tilde{\mc U})$ is independent on the choice of $dev$. By $KF(\tilde L)$ we denote the space of the Killing fields on $\tilde L$. Note that all vector fields in $KF(\tilde L)$ are pull-backs of Killing fields on $\H^3$ via $dev$ and $KF(\tilde L) \subset AF(\tilde L, \mc U)$ as any such field can be represented as $\sum_{i \in I} \tilde \lambda_i X_i$ where all $X_i$ are the same Killing field on $\tilde L$.


As previously, we consider the tuple of edge lengths of $\mathcal T$ as $l \in \R^{E(\mathcal T)}$ and by $T^L_l \R^{E(\mathcal T)}$ we denote the intersection of kernels of $d\nu_e$ for all edges $e \in E_0(\mc T)$. We want to establish an isomorphism of $T^L_l \R^{E(\mathcal T)}$ with a subspace of $AF(\tilde L, \tilde{\mc U})$ modulo $KF(\tilde L)$. 
We need to understand which vector fields from $AF(\tilde L, \mathcal U)$ well-define variations of the edge lengths of $\mathcal T$. These are so-called \emph{automorphic} ones.

This notion applies to a slightly more general context and is heavily connected with variations of a hyperbolic structure on manifolds.
For the moment, let $L$ be any open hyperbolic 3-manifold and $\mathcal {HM}(L)$ be the set of hyperbolic metrics on $L$ up to isotopies and thickenings. (For the notion of thickening, see, e.g.,~\cite[Chapter I.1.6]{CEG}. This is a simple equivalence relation for incomplete hyperbolic metrics on open manifolds.) We note that the space of Killing vector fields on $\tilde L$ always has dimension exactly 6: it is easy to see it from any developing map to $\H^3$. 
Consider a smooth path of hyperbolic metrics on $L$. It gives rise to a smooth path of developing maps from $\tilde L$ to $\H^3$ defined up to a path of global isometries of $\H^3$. In the first order it produces a vector field $X$ on $\tilde L$ defined up to Killing fields. 
It is easy to see that $X$ satisfies the \emph{automorphic property}: for every $\gamma \in \pi_1(L)$ considered as a deck transformation $\gamma: \tilde L \rightarrow \tilde L$ the vector field $X-d\gamma^{-1}(X)$
is a Killing field (various for various $\gamma$). If the path of metrics is trivial, e.g., coming from the pull-back of an isotopy, then $X$ is the sum of an equivariant vector field and a Killing vector field, i.e., the vector field above is the same for all $\gamma$.

Now we return to our setting and by $AF_a(\tilde L, \mathcal U)$ denote the space of automorphic affine vector fields on $\tilde L$ with respect to $\mathcal U$. Note that $KF(\tilde L) \subset AF_a(\tilde L, \mathcal U)$. However, we are going to see that the space $AF_a(\tilde L, \mathcal U)$ is larger and to prove

\begin{lm}
\label{isom2}
$$T^L_l \R^{E(\mathcal T)} \cong AF_a(\tilde L, \mathcal U)/KF(\tilde L).$$
\end{lm}

\begin{proof}
Let $X\in AF_a(\tilde L, \mathcal U)$ and take $\gamma \in \pi_1(L)$. As $X-d\gamma^{-1}(X)$ is a Killing field, we see that for every tetrahedron $T$ of $\mathcal T$ and every two its lifts $\tilde T_1$ and $\tilde T_2$ in $\tilde L$ that differ by $\gamma$, the edge-length change induced by $X$ is the same. By construction, $X$ does not change infinitesimally the dihedral angles of the edges from $E_0(\mc T)$. Hence, all lengths changes of tetrahedra in $\tilde L$ project to a length change $\dot l \in T^L_l \R^{E(\mathcal T)}$. This defines a linear map $$AF_a(\tilde L, \mathcal T) \rightarrow T^L_l \R^{E(\mathcal T)}$$
with $KF(\tilde L)$ equal to its kernel.

Conversely, let $\dot l \in \R^{E(\mathcal T)}$. Similarly to the proof of Lemma~\ref{isom1}, we produce inductively an affine vector field $X$ on $\tilde L$ with respect to $\mc U$ inducing $\dot l$. We need to show that it is automorphic. 

Pick a tetrahedron $\tilde T_1$ of $\mathcal T$ in $\tilde L$ and $\gamma \in \pi_1(L)$. Let $\tilde T'_1$ be the tetrahedron $\gamma(\tilde T_1)$, and let $X_1$ and $X'_1$ be the components of $X$ corresponding to the tetrahedra $\tilde T_1$ and $\tilde T'_1$ respectively. As $X_1$ and $d\gamma^{-1}(X'_1)$ produce the same length change on the tetrahedron $\tilde T_1$, we see that the difference $X_1-d\gamma^{-1}(X'_1)$ is a Killing field on the tetrahedron $\tilde T_1$. 

We need to show that for any other tetrahedron $\tilde T_2$ if $\tilde T'_2:=\gamma(\tilde T_2)$ and the corresponding components of $X$ are $X_2$ and $X'_2$, then the vector field $X_2-d\gamma^{-1}(X'_2)$ is the same Killing field as $X_1-d\gamma^{-1}(X'_1)$. Assume that $\tilde T_2$ is adjacent to $\tilde T_1$ by a 2-face and that a developing map $dev: \tilde L \ra \H^3$ is fixed, so we treat the components of $X$ as elements of $AF$. The deformations of the edge lengths of $\tilde T_1$ and $\tilde T_2$ induced by $\dot l$ together with the isomorphism~(\ref{tetriso}) define respectively two affine 6-dimensional spaces $A_1, A_2 \subset AF$ parallel to $KF$ and a map $$A_{21}: A_2 \rightarrow A_1$$ mapping an affine field $X_2 \in A_2$ to the unique affine field $X_1 \in A_1$ such that the restrictions of $X_1$ and $X_2$ to the common face of $\tilde T_1$ and $\tilde T_2$ coincide. We extend $A_{21}$ linearly to the linear span of $A_2$. The latter contains $KF$ and $A_{21}$ preserves $KF$ pointwise.

Similarly we get affine subspaces $A'_1, A'_2 \subset AF$ defined by the restriction of $\dot l$ to the edge lengths of $\tilde T'_1$ and $\tilde T'_2$ and get an affine map
$$A'_{21}: A'_2 \rightarrow A'_1,$$
which we extend to the linear spans of $A'_1$, $A'_2$. We have
$$A'_{21}=d\gamma\circ A_{21} \circ d\gamma^{-1},$$
$$d\gamma(X_1)-X'_1=d\gamma\circ A_{21}(X_2)-A'_{21}(X'_2)=d\gamma\circ A_{21}(X_2-d\gamma^{-1}(X'_2)).$$
$$X_1-d\gamma^{-1}(X'_1)=A_{21}(X_2-d\gamma^{-1}(X'_2)).$$
It follows that $X_1-d\gamma^{-1}(X'_1)$ and $X_2-d\gamma^{-1}(X'_2)$ is the same Killing field. By induction we show that for every component $X_i$ of $X$, $X_i-d\gamma^{-1}(X_i)$ is the same Killing field on the whole $\tilde L$. This finishes the proof.
\end{proof}

We now describe affine equivariant vector fields on $\tilde L$ with respect to $\mc U$.

\begin{lm}
\label{equiv}
Let $X \in AF(\tilde L, \mc U)$ be equivariant and $\tilde T$ be a tetrahedron of $\mathcal T$ in $\tilde L$. Consider a developing map $dev: \tilde T \rightarrow \H^3$, identify $\tilde T$ with the image and extend $X$ to the vertices. Then $X$ is zero at each vertex that has valence greater than 2 in $\Sigma$. For vertices of valence~2, $X$ is parallel to (the $dev$-image) of the singular locus.
\end{lm}

\begin{proof}
Indeed, suppose that for a vertex $v$ of $\tilde T$ having valence at least 3 in $\Sigma$ the extension of $X$ is non-zero. Then there exists an axis of $\Sigma$ that, after developing to $\H^3$, is not parallel to $X$ at $v$. Consider a chain of tetrahedra adjacent by 2-faces, starting at $\tilde T$ and looping around this axis. We develop it to $\H^3$. Due to the equivariance, the extension of $X$ to $v$ in the last tetrahedron of the chain should differ from the extension in $\tilde T$. On the other hand, since all tetrahedra in the chain are adjacent by 2-faces, these extensions must coincide. This is a contradiction. 

This argument similarly shows that if $v$ has valence 2 in $\Sigma$, then the extension of $X$ must be parallel to (the $\dev$-image of) the singular axis of $\Sigma$ passing through $X$.
\end{proof}

\subsection{Bundle-valued differential forms}
\label{bundle}

Recall that $G=PSL(2, \C)$, $\tilde G=SL(2, \C)$. By $\mathfrak g=\mathfrak{sl}(2, \C)$ we denote the Lie algebra of $\tilde G$. By a result of Culler~\cite{Cul}, every holonomy map $\pi_1(L) \ra G$ lifts to a holonomy map $\pi_1(L) \ra \tilde G$. Denote the latter by $\rho$. Consider the trivial bundle $\tilde E \rightarrow \tilde L$, where $\tilde E=\tilde L \times \mathfrak g$. It is naturally equipped with a flat connection coming from the trivialization. Consider the action of $\pi_1(L)$ on $\tilde E$ where the action on $\tilde L$ is given by the deck transformations, and the action on $\mathfrak g$ is given via $Ad \circ \rho$, where $Ad$ is the adjoint representation of $\tilde G$ on $\mathfrak g$. By taking the quotient, we obtain a flat bundle $E \rightarrow L$, which is the bundle of germs of Killing fields on $L$. Let $\Omega^*(L, E)$ be the de Rham complex of $E$-valued differential forms on $L$, where the exterior derivative $d$ is induced by the flat connection on $E$. This complex is naturally isomorphic to the complex $\Omega^*(\tilde L, \tilde E, Ad\circ\rho)$ consisting of $\tilde E$-valued differential forms on $\tilde L$ that are equivariant with respect to the action $Ad\circ\rho$. This construction goes back to Weil~\cite{Weil}, see a general exposition in the book~\cite{Rag} of Raghunathan, and papers~\cite{HK, Wei, Wei2} for the applications to the infinitesimal theory of hyperbolic cone-3-manifolds.

Let $e$ be an edge of $\Sigma$ and $\gamma_e$ be a meridian around $e$. In~\cite[Theorem 3.15]{Wei2} Weiss showed that

\begin{thm}
\label{weiss}
Let $\omega \in \Omega^1(L, E)$ belong to $\ker(d\tr(\gamma_e))$ for all $e \in \Sigma$. Then $\omega$ is cohomologous to zero.
\end{thm}

For an open hyperbolic 3-manifold $L$ there are two main ways to describe an infinitesimal deformation of the hyperbolic structure: via the first cohomology group $H^1(L, E)$ and via the automorphic vector fields. Following~\cite{HK} we now sketch why these ways are equivalent.

At each point $p \in L$ the fiber $E_p\cong \mathfrak g$ has a natural decomposition $$E_p=E_{1,p} \oplus E_{2, p}$$ into the subspace $E_{1,p}$ of infinitesimal translations at $p$ and the subspace $E_{2,p}$ of infinitesimal rotations around $p$. This produces the decomposition $E=E_1 \oplus E_2$ of $E$ into two respective sub-bundles. We remark that this decomposition is not preserved by the flat connection of $E$. The subspace $E_{1,p}$ can be naturally identified with the tangent space $T_pL$ by associating to each tangent vector the infinitesimal translation in its direction, hence the sub-bundle $E_1$ can be naturally identified with the tangent bundle $TL$, and we will use this identification freely. The Lie algebra $\mathfrak g= \mathfrak{sl}(2,\C)$ has the standard complex structure, which induces a complex structure on $E$ such that $E_2=iE_1$. For $X \in T_pL\cong E_{1,p}$ the element $iX \in E_{2, p}$ is the infinitesimal rotation with the axis parallel to $X$ and having the same norm. Here the norm on $E_{1,p}$ comes from the metric on $TL$ and the norm on $E_{1,p}$ comes from the usual metric on $\mathfrak{so}(3)$.

Lift a 1-form $\omega \in \Omega^1(L, E)$ to an equivariant 1-form $\tilde \omega \in \Omega^1(\tilde L, \tilde E, Ad\circ\rho)$. If $\omega$ is closed, so is $\tilde \omega$, and, as $\tilde L$ is simply connected, $\tilde \omega$ is exact. Hence, $\tilde \omega=ds$, where $s$ is a section of $\tilde E$. However, $s$ may be non-equivariant. But the notion of automorphicity extends to sections of $\tilde E$ and $s$ is automorphic, see, e.g.,~\cite[Lemma 2.1]{HK}. Here a section is called automorphic if for every $\gamma \in \pi_1(L)$ the section
$s-\gamma^{-1}_*s$
is closed. It follows conversely that for an automorphic section $s$ the 1-form $ds$ is equivariant. Using the complex structure on $\tilde E$, we can consider every section of $\tilde E$ as a pair of vector fields. Clearly, $s$ is automorphic (resp. equivariant or constant) if and only if the both corresponding vector fields are automorphic (resp. equivariant or Killing). It is evident that $s$ is the sum of an equivariant section and a constant section if and only if $\omega$ is exact, so it is cohomologous to zero. 

Now we need to learn the passage from an automorphic vector field $X$ on $\tilde L$ to a closed $E$-valued 1-form. Let $D$ be the exterior derivative on $\Omega^*(\tilde L, T\tilde L)$ induced by the Levi--Civita connection for the metric on $\tilde L$. Thus, $DX$ is a $(T\tilde L)$-valued 1-form. Any $(T\tilde L)$-valued 1-form can be viewed as an element of $Hom(T\tilde L, T\tilde L)$. We decompose it into the symmetric and the skew-symmetric parts using the metric on $\tilde L$. As $T_p\tilde L$ is a 3-dimensional vector space with a positively-definite metric, it is canonically isomorphic to the skew-symmetric subspace of $Hom(T_pM, T_pM)$. Indeed, if $\{e_1, e_2, e_3\}$ is an orthonormal basis of $T_pM$ and $\{e^1, e^2, e^3\}$ is the dual basis, then the isomorphism is defined by sending $e_i$ to $e_j\otimes e^h-e_h\otimes e^j$, where $(i,j,h)$ is a cyclic permutation of $(1,2,3)$. This defines a bundle-isomorphism of $T\tilde L$ and $Hom(T\tilde L, T\tilde L)$. We denote by $\curl(X)$ the vector field obtained from the skew-symmetric part of $DX$ under this isomorphism. By $s_X:=X-i\curl(X)$ we denote the \emph{canonical lift} of $X$, which is a section of $\tilde E$. Hodgson--Kerckhoff showed

\begin{lm}[\cite{HK}, Lemma 2.3]
The canonical lift of an automorphic vector field is an automorphic section. Every cohomology class in $H^1(L, E)$ can be represented by a 1-form $\omega$ such that the associated local section $s: \tilde L \rightarrow \tilde E$ is the canonical lift of the vector field determined by its real part.
\end{lm}

The following lemma follows from the discussion above.

\begin{lm}
\label{form}
Assume that for an automorphic vector field $X$ the 1-form $\omega \in \Omega^1(L, E)$ obtained from the projection of the 1-form $ds_X$ is cohomologous to zero. Then $X$ is the sum of an equivariant and a Killing vector field.
\end{lm}

\subsection{Proof of the infinitesimal rigidity}

We now have all the ingredients to prove Lemma~\ref{infrig}. As it was showed in Subsection~\ref{subst}, it is enough to prove Lemma~\ref{subst3}. 

\begin{proof}[Proof of Lemma~\ref{subst3}.]
Let $K$ be the closed hyperbolic cone-3-manifold obtained from $M(g)$ by reflecting it along the boundary. As usually, by $\Sigma$ we denote the singular locus and by $L$ we denote $K \backslash \Sigma$. The triangulation $\mathcal T$ extends to a triangulation $\mathcal T^D$ of~$K$. Note that due to Lemma~\ref{triang2}, $\mathcal T$ is large, and one can see that $\mc T^D$ is also large. By $\mc U$ we denote a partition of unity with respect to an open cover from Lemma~\ref{cover}. We can make $\mc U$ symmetric with respect to the antipodal involution of $K$ switching the copies of $M(g)$. Consider an infinitesimal deformation $\dot l$ of the edge lengths of $\mathcal T$ as in the statement of Lemma~\ref{subst3}. We associate to it a deformation $\dot l^D$ of the edges of $\mathcal T^D$ as follows. For every edge $e \in E(\mathcal T^D)$ coming from the boundary of $M(g)$ we set $\dot l^D_e=0$. Next, there are two isometric copies of $M(g)$ in $K$. In one copy, for every edge $e_1 \in E(\mathcal T^D)$ coming from an interior edge $e\in E(\mathcal T)$ we set $\dot l^D_{e_1}=\dot l_e$. For the second copy and an edge $e_2 \in E(\mathcal T^D)$ coming from the same interior edge $e\in E(\mathcal T)$ we set $\dot l^D_{e_1}=-\dot l_e$. It is trivial to see now that for every edge $e \in E(\mathcal T^D)$ we have $d\nu_e(\dot l^D)=0$.

From Lemma~\ref{isom2} we associate to $\dot l^D$ an affine automorphic vector field $X$ on $\tilde L$ with respect to $\mc U$. Using Subsection~\ref{bundle} we consider its canonical lift $s_X: \tilde L \rightarrow \tilde E$ and the 1-form $\omega \in \Omega^1(L, E)$ obtained from the equivariant 1-form $\tilde \omega=ds_X$.

We now remark that for every meridian $\gamma_e$ around an edge $e$ of $\Sigma$ we have $\omega \in \ker(d\tr(\gamma_e))$. Indeed, we can choose $\gamma_e$ to be symmetric with respect to the antipodal involution of $K$. Now let us consider a tetrahedron $T \subset \H^3$, a deformation $\dot l_T$ of its edges producing an affine vector field $X_T=X_{\dot l_T}$ on $T$ corresponding to a 1-form $\omega_{X_T}$, and consider a path $\gamma_T$ in $T$. Due to linearity, it is easy to see that $-X_T$ corresponds to the deformation $-\dot l_T$. Hence, from the linearity of all constructions, we get
$$\int_{\gamma_T}\omega_{X_T}=-\int_{\gamma_T}\omega_{-X_T}.$$
Thus, returning to $L$, from the symmetry we see that 
$$\int_{\gamma_e}\omega=0.$$

Theorem~\ref{weiss} implies that $\omega$ is cohomologous to zero. Lemma~\ref{form} shows that the vector field $X$ is the sum of an equivariant one and a Killing one. Subtracting the Killing part, we assume that it is equivariant. 

Let $e'$ be any edge of $\mathcal T^D$ coming from an edge $e$ of $\mathcal T$ connecting two vertices $v_1, v_2 \in V$. Note that every $v \in V$ is adjacent to at least three edges of $g$, which belong to $\Sigma$ after the doubling. Develop to $\H^3$ any tetrahedron containing $e'$ and consider the extension of $X$ to the vertices of $e$. Lemma~\ref{equiv} shows that this extension is zero. Hence, $\dot l^D_{e'}=\pm \dot l_e=0$. 

Now fix a copy of $M(g)$ in $K$, consider a geodesic $\psi \subset \partial M(g)$ and a face $F \subset \partial M(g)$ adjacent to $\psi$. Take all tetrahedra adjacent to $\psi$ having a face or an edge in $F$ and develop them to $\H^3$, we continue to denote by $\psi$ the obtained geodesic segment in $\H^3$. Let $T$ be such a tetrahedron with an edge $e$ belonging to $\psi$. By Lemma~\ref{equiv}, the extension of $X$ to every vertex of $e$ is parallel to $\psi$. As $\dot l_e=0$, we see that the extension vectors have the same direction and the same norm. This holds for all vertices $w\in \psi$. However, there exists at least one edge $e'$ of $\mathcal T$ in $F$, which is adjacent to $\psi$ by only one vertex and which is not orthogonal to $\psi$. The second endpoint of $e'$ is in $V$, hence the extension of $X$ there is zero. Since $\dot l_{e'}=0$, we get that all the obtained vectors at the points on $\psi$ are zero. 

Hence, the extension of $X$ to any vertex of $K$ is zero. Thus, we have $\dot l^D=0$, which implies $\dot l=0$. This finishes the proof of Lemma~\ref{subst3}.
\end{proof}

\section{Proofs of the main results}

\subsection{Properness}


Recall Definition~\ref{fucdef} of proto-Fuchsian metrics on $\pt M$. We denote the set of classes of such metrics by $\mc D_f(\pt M) \subset \mc D_c(\pt M, \emptyset)$. By $\mc D_f(\pt M, V)$ we denote the subset of $\ol{\mc D}_c(\pt M, V)$ corresponding to the lifts of metrics from $\mc D_f(\pt M)$ and by $\mc D_{nf}(\pt M, V)$ we denote its complement in $\ol{\mc D}_c(\pt M, V)$. While studying converging sequences in $\ol{\mc P}_b(M,V)$ we need to exclude the degenerate limits: when the sequence of underlying convex cocompact metrics on $N$ converges to a Fuchsian metric $\ol g$ and $V$ falls to the span of $C(\ol g)$. We claim that in this case the induced metric on $\pt M$ degenerates to a proto-Fuchsian one.

\begin{lm}
\label{fuchsian}
Let $\{g_i\}$ be a sequence in $\mathcal P_b(M, V)$ and $\{(\rho_i, \tilde f_i)\}$ be the sequence of its lifts to $\mc R(\pi_1(M), G) \times (\H^3)^{V}$ such that $\{\rho_i\}$ converges to a Fuchsian representation $\rho$ of a Fuchsian metric $\overline g$ and $\{\tilde f_i\}$ converges to a map $\tilde f: V \hookrightarrow \tilde \Pi$, where $\Pi$ is the geodesic plane containing $C(\overline g)$. Then the sequence $\{d_i\}$ of induced metrics on $\pt M$ converges in $\ol{\mc D}_c(\pt M, V)$ to a metric $d \in \mc {D}_f(\pt M, V)$.
\end{lm}

\begin{proof}
We claim that the proof of Lemma~\ref{contin} still applies to the current situation after slight changes. Indeed, as in Lemma~\ref{hausdd} one can see that $\tilde M(g_i)$ converge in the Hausdorff sense to $\tilde C(\ol g)$. In the Fuchsian case of the first kind Lemma~\ref{distconv} applies without problems and the proof of Lemma~\ref{contin} shows that $\{d_i\}$ converges to $d \in \ol{\mc D}_c(\pt M, V)$ that is a lift of the induced metric on $C(\ol g)$ on each component of $\pt M$ (recall that here $C(\ol g)$ is a totally geodesic surface), so $d \in {\mc D}_f(\pt M, V)$. 

In the Fuchsian case of the second kind Lemma~\ref{distconv} still applies, but the notion of the convergences should be corrected, as described in~\cite{Ale2, Ale3}. Let $\tilde S_i \subset \H^3$ be a sequence of complete convex surfaces converging in the Hausdorff sense to a planar closed convex subset $\tilde S \subset \Pi \subset \H^3$, where $\Pi$ is a geodesic plane. We consider the double $D\tilde S$ of $\tilde S$ along the relative boundary of $S$ with the induced metric. Orient $\Pi$ and for each $p$ in the relative interior of $\tilde S$, choose two rays orthogonal to $\Pi$, $r_p^{\uparrow}$ and $r_p^{\downarrow}$, denoted with respect to the orientation. Denote the two copies of $p$ in $DS$ by $p^{\uparrow}$ and $p^{\downarrow}$, depending on the copy of the relative interior of $\tilde S$ in~$D\tilde S$. We say that a sequence of points $p_i \in \tilde S_i$ converge intrinsically to $p^{\uparrow}$ (respectively to~$p^{\downarrow}$) if $p_i$ converges to $p$ and for all sufficiently large $i$ the ray $r_p^{\uparrow}$ (respectively $r_p^{\downarrow}$) points at infinity to the outer half-space of some supporting plane to $S_i$ at $p_i$. In this case Lemma~\ref{distconv} still works in the sense that if $p_i$ converge intrinsically to $p \in D\tilde S$ and $q_i$ converge intrinsically to $q \in D\tilde S$, then the intrinsic distance between $p_i$ and $q_i$ on $\tilde S_i$ converges to the intrinsic distance between $p$ and $q$ in $D\tilde S$.

In the Fuchsian case of the second king, the convex core $C(\ol g)$ is a surface with non-empty geodesic relative boundary. Let $\tilde S = \clconv(\tilde V)$ be contained in the plane $\Pi$ containing $\tilde C(\ol g)$, where $\tilde V$ is the $\rho_{\ol g}$-orbit of $\tilde f(V)$. Then the quotient $S \subset N(\ol g)$ is a totally geodesic surface with convex piecewise-geodesic relative boundary. Then the proof of Lemma~\ref{contin} works to show that $\{d_i\}$ converges to $d \in \ol{\mc D}_c(\pt M, V)$ that is a lift of the double of $S$ along its relative boundary, so $d \in {\mc D}_f(\pt M, V)$. 
\end{proof}

\begin{lm}
\label{proper}
Let $\{g_i\}$ be a sequence of metrics in $\ol{\mathcal P}_b(M, V)$ such that $I_V(g_i)=:d_i \in \ol{\mathcal D}_c(\partial M, V)$ converge to $d \in \ol{\mathcal D}_{nf}(\partial M, V)$. 
Then up to passing to a subsequence $g_i$ converge in $\ol{\mathcal P}_b(M, V)$ to $g \in \ol{\mathcal P}_b(M, V)$. We have $V(g)=V(d)$. In particular, if $d \in \mc D_c(\pt M, V)$, then $g \in \mc P_b(M, V)$.
\end{lm}

\begin{proof}
Lemma 4.4 in~\cite{Pro3} implies that up to passing to a subsequence $\ol g_i$ converge to $\ol g \in \mc{CH}(N)$. We just recall that the proof is a combination of the Busemann--Feller lemma, the Sullivan--Bridgeman--Canary theorem~\cite{Sul, BC} producing a bilipschitz map between $\pt C(\ol g)$ and $\pt_\infty N(\ol g)$, and of a compactness criterion for Teichm\"uller spaces.

Suppose that for some $v \in V$ the distance $g_i(v, C(\overline g_i))$ tends to infinity. Then if for some other point $p \in \partial M$ the distance $g_i(p, C(\overline g_i))$ does not tend to infinity, the distance $g_i(v, p)$ tends to infinity. The Busemann--Feller lemma implies that then $d_i(v, p)$ tends to infinity, which contradicts to the convergence of $d_i$.

For $h \in \R_{>0}$ let $S_{i,h} \subset N(\overline g_i)$ be the surface at the distance $h$ from $C(\overline g_i)$. Then the area of $S_{i,h}$ is at least the area of $\partial C(\overline g_i)$ times $\cosh^{-1}(\sqrt h)$. Indeed, each face of $\partial C(\overline g_i)$ corresponds to a constant-curvature piece on $S_i$ of area equal to the area of the face times $\cosh^{-1}(\sqrt h)$. Since the union of all edges of $\partial C(\overline g_i)$ has measure zero~\cite[Theorem 4.9]{CB}, we prove the claim. 

Thereby, the distance $g_i(\partial M, C(\overline g_i))$ grows to infinity. Hence, there exists a sequence of $h_i$ growing to infinity such that $M(g_i) \supset S_{i,h_i}$. Due to the argument above and the convergence of $\overline g_i$, the areas of $S_i$ grow to infinity. Since the nearest-point projection from $\partial M(g_i)$ to $S_i$ is contracting, we get that the areas of $d_i$ grow to infinity, which is a contradiction. Thus, for each $v \in V$ the distances $g_i(v, C(\overline g_i))$ stay uniformly bounded. 

As in Lemma~\ref{hausdd}, we can choose the embeddings of $\tilde C(\overline g_i)$ and $\tilde C(\overline g)$ to $\H^3$ such that $\tilde C(\overline g_i)$ converge in the Hausdorff sense to $\tilde C(\overline g)$. As for each $v \in V$ the distance $g_i(v, C(\overline g_i))$ stay uniformly bounded, we can choose the maps $\tilde f_i: V \hookrightarrow  \H^3$ such that up to taking a subsequence it converges to a map $\tilde f: V \hookrightarrow \H^3$. By $\tilde V_i$ and $\tilde V$ we denote the orbits of $\tilde f_i(V)$ ans $\tilde f(V)$ under the respective holonomy representations. Consider $\clconv(\tilde V)$ quotiented by the action of $\rho_{\overline g}$. We get a manifold homeomorphic to $M$ except the case when $\overline g$ is Fuchsian and $\tilde V \subset \tilde C(\overline g)$. In the latter case Lemma~\ref{fuchsian} shows that $d \notin {\mc D}_{nf}(\pt M, V)$, hence this can not happen. As $\tilde V_i \subset \pt \clconv(\tilde V_i)$, we get $\tilde V \subset \pt \clconv(\tilde V)$, and the metrics $g_i$ converge to $g \in \ol{\mc P}_b(M, V)$. By Lemma~\ref{contin}, the induced metric on $\pt M(g)$ is $d$ and it is easy to see that $V(g)=V(d)$. This finishes the proof.

%
%
%
\end{proof}

%
%
%
%
%
%
%

\subsection{Existence of bent realizations}

In this section we prove the existence part of Theorem~\ref{main1}. 

\begin{proof}

Take arbitrary $d_0 \in \mathcal D_{nf}(\partial M, V)$ and $g_1 \in \mathcal P_b(M, V)$, define $$d_1:=\mathcal I_V(g_1) \in \mathcal D_c(\partial M, V).$$ 
By Lemma~\ref{connect1} there exist $W \supseteq V$ and a path $\alpha: [0, 1] \rightarrow \ol{\mc D}_{bc}(\pt M, W)$ such that $\alpha(0)$, $\alpha(1)$ are lifts of $d_0$, $d_1$ and for all $t\in (0,1)$ we have $\alpha(t) \in \mathcal D_{bc}(\partial M, W)$. We want to slightly modify the path so that also $\alpha(1) \in \mathcal D_{bc}(\partial M, W)$ and $\alpha(1)$ remains in the image of ${\mathcal I}_W$. Since $\overline{\mathcal D}_{bc}(\partial M, W)$ is open in $\ol{\mc D}_c(\pt M, W)$, we can choose a neighborhood $U\ni \alpha(1)$ in $\ol{\mc D}_c(\pt M, W)$ contained in $\overline{\mathcal D}_{bc}(\partial M, W)$. As ${\mathcal I}_W$ is continuous, ${\mathcal I}^{-1}_W(U)$ is open in $\ol {\mc P}_b(M, W)$ and there is $g \in (\mc P_b(M, W)\cap {\mc I}^{-1}_W(U))$. Let $\alpha': [0,1] \rightarrow U$ be a path connecting $\alpha(t') \in (U\cap \mc D_c(\pt M, W))$ for some $t'$ close to 1 with $\mc I_W(g) \in (U \cap \mc D_c(\pt M, W))$. 
We use Lemma~\ref{homotop} to modify $\alpha'$ to a path $\alpha'': [0,1] \rightarrow \mc D_c(\pt M, W)$ connecting the same endpoints.
As $U$ is open, we can do it so that $\alpha''(t) \in U \subset \ol{\mc D}_{bc}(\pt M, W)$. We modify $\alpha$ by replacing the restriction of $\alpha$ on $[t', 1]$ with the path $\alpha''$.

Let $T \subset [0,1]$ be a subset consisting of those $t$ that $\alpha(t)$ is in the image of ${\mathcal I}_W$. If $t \in T$ and ${t>0}$, then $\alpha(t)$ is balanced, thus Corollary~\ref{balreal} implies that $\mathcal I^{-1}_W(\alpha(t)) \subset \mathcal P_{cp}(M, W)$. Hence, Corollary~\ref{locrig} implies that $\mathcal I_W$ is a local homeomorphism at points of $\mathcal I^{-1}_W(\alpha(t))$. Thus, there is a neighborhood of $t$ in $[0,1]$ belonging to $T$. It follows that $T \cap (0, 1]$ is open. Lemma~\ref{proper} implies that $T$ is also closed. Therefore, $T=[0,1]$ and $\alpha(0)$ is in the image of ${\mathcal I}_W$. This implies that $d_0$ is in the image of $\mc I_V$ and finishes the proof.
\end{proof}

\subsection{Topological interlude}

For the proof in the next section we need to construct a covering space over $\mc P_b(M, V)$, on which we can define the induced metric map with values in $\mc D_c^\sharp(\pt M, V)$. First, define the space of \emph{pointed} convex cocompact hyperbolic metrics $\mc{PCH}(N)$ to be $\mc{CH}(N, V)$ for a set $V$ of cardinality one. Next, by a \emph{geometric end} of $N(\ol g)$ we will mean a connected component of the complement of $N(\ol g)$ to $C(\ol g)$. There is a 1:1 correspondence between components of $\pt M$ and geometric ends of $N(\ol g)$. Denote the components of $\pt M$ by $S_j$, $j=1 \ldots m$, and denote the positive real numbers by $\R_+$. Let $\mc{PCH}_j(N)$ be the subset of $\mc{PCH}(N)$ with the marked point belonging to the $j$-th geometric end. It was shown in~\cite[Section 5.1]{Pro3} that the space $\mc{PCH}_j(N)$ can be trivialized as
\begin{equation}
\label{tr1}
\mc{PCH}_j(N) \cong \mc{CH}(N) \times S_j\times \R_+
\end{equation}
such that 
the sets $(\ol g, p, \R_+)$, $\ol g \in \mc{CH}(N), p \in S_j,$ are the gradient lines of the distance function to the convex core in the $j$-th geometric end of $N(\ol g)$. In what follows we fix such a trivialization for every $j=1\ldots m$. The purpose of this is that now for every $\ol g \in \mc{CH}(N)$ and for every totally convex surface $\Sigma$ in $N(\ol g)$ (i.e., $\Sigma$ is a boundary component of a totally convex set) that does not touch $C(\ol g)$, we have a homeomorphism $S_j \ra \Sigma$ for $j$ indicating the geometric end containing $\Sigma$. Indeed, it is given by the inverse of the projection of $\Sigma$ to the $S_j$-factor in trivialization~(\ref{tr1}). Things are slightly subtler if $\Sigma$ touches the convex core, but one can see that in this case anyway the chosen trivializations furnish a continuous surjective map $S_j \ra \Sigma$. It is not injective exactly if $\Sigma$ contains an edge that is a closed geodesic. Then the preimage of this geodesic is a closed tubular neighborhood of a simple closed curve in $S_j$.

Next, consider arbitrary $V \subset \pt M$ and define $V_j:=V \cap S_j$.
Denote by $\mc{CH}^*(N, V)$ the subset of $ \mc{CH}(N, V)$ such that points from $V_j$ mark points in the $j$-th geometric end, and no two marked points belong to the same gradient line of the distance function to the convex core. (Note that the first condition implies that all the marked points are at positive distance from the convex core.) By $S_{*V}$ denote the product of $S_j^{*V_j}$, where $S_j^{*V_j} \subset S_j^{V_j}$ is the space of injective maps from $V_j$ to $S_j$. Trivializations~(\ref{tr1}) define a trivialization
\begin{equation}
\label{tr3}
\mc{CH}^*(N, V) \cong \mc{CH}(N) \times S_{*V} \times \R^V_+.
\end{equation}
Denote by $\mc{CH}^\sharp$ the universal cover of the latter space. We have a trivialization
\begin{equation}
\label{tr2}
\mc{CH}^\sharp(N, V) \cong \mc{CH}(N) \times \tilde{S_{*V}} \times \R^V_+.
\end{equation}
It is easy to see that $\mc P_b(M, V)$ belongs to $\mc{CH}^*(N, V)$. Denote its respective covering space, as a subset of $\mc{CH}^\sharp(N, V)$, by $\mc P_b^\sharp(M, V)$.


Our current goal is to define a map $\mc I^\sharp_V: {\mc P}_b^\sharp(M, V) \ra {\mc D}_c^\sharp(\pt M, V)$. Denote by $H_0(\pt M)$ the space of all self-homeomorphisms of $\pt M$ isotopic to identity, endowed with the compact-open topology. Recall that $H_0^\sharp(\pt M, V)$ denotes its normal subgroup consisting of those homeomorphisms, for whom an isotopy can be chosen fixing $V$. The inclusion $V \hookrightarrow \pt M$ can be considered as an element of $S_{*V}$. Fix over this element some element $\hat V$ of $\tilde{S_{*V}}$. Every homeomorphism $h \in H_0(\pt M)$, by the homotopy lifting, determines a self-homeomorphism of $\tilde{S_{*V}}$. By considering its evaluation at $\hat V$, we define a map $H_0(\pt M) \ra \tilde{S_{*V}}$, which factors through $H_0^\sharp(\pt M, V)$ to a map
\[{\rm ev}: H_0(\pt M)/H_0^\sharp(\pt M, V) \ra \tilde{S_{*V}}.\]
It was shown in~\cite[Lemma 3.1]{FP} that this map is a homeomorphism.

Let $g \in {\mc P}_b(M, V)$, represented by a pair $(\ol g, f)$, and $g^\sharp$ be its lift to ${\mc P}_b^\sharp(M, V)$. Trivialization~(\ref{tr1}) provides a map $\chi:  \pt M \ra \pt M(g)$, as described above. We need to slightly modify it to obtain a homeomorphism. Indeed, let $\psi$ be the union of all edges of $M(g)$ that are simple closed geodesics. The preimage $\chi^{-1}(\psi)$ is a closed tubular neighborhood $X$ of a multicurve $\lambda \subset \pt M$, which has a parametrization $X \cong [-1,1] \times \lambda$ so that $\chi$ sends each fiber $[-1,1] \times p$, $p \in \lambda$, to a point. Consider a small thickening $Y \cong [-2, 2] \times \lambda$ so that $Y$ is disjoint with $\chi^{-1}(V(g))$. Then $Z:=\chi(Y)$ is a tubular neighborhood of $\psi$ in $\pt M$. We can modify $\chi$ inside $Y$ by a fiberwise homeomorphism between $Y$ and $Z$. We continue to denote the modified map by $\chi$. Trivialization~(\ref{tr2}) and the choice of the evaluation map provide a class in $H_0(\pt M)/H_0^\sharp(\pt M, V)$, for which we choose a representative $h$. One can see that the composition $f^{-1}\circ \chi \circ h$ is the identity on $V$. Thereby, the pull-back of the induced metric on $\pt M(g)$ by $\chi \circ h$ to $(S, V)$ gives an element of ${\mc D}_c^\sharp(\pt M, V)$, which we denote by $\mc I^\sharp_V(g^\sharp)$, independent on the choice of $h$ in its class and of the modification of $\chi$ above. (Let us stress out that it is dependent only on trivialization~(\ref{tr1}) and on $\hat V$.)

\begin{rmk}
\label{deck}
The group of deck transformations of $\tilde{S_{\ast V}} \ra S_{\ast V}$ is naturally isomorphic to the pure braid group $B_0(\pt M, V)=H_0(\pt M, V)/H_0^\sharp(\pt M, V)$ from Section~\ref{conesec}, see~\cite[Remark 3.5]{FP}. Thereby, it is also the group of deck transformations of the covering map ${\mc P}_b^\sharp(M, V)\ra {\mc P}_b(M, V)$. It is also the group of deck tranformations of the covering map ${\mc D}^\sharp_c(\pt M, V) \ra {\mc D}_c(\pt M, V)$. By construction, the map $\mc I^\sharp_V$ is equivariant with respect to $B_0(\pt M, V)$.
\end{rmk}

\begin{lm}
\label{connect}
The space ${\mathcal P}_b^\sharp(M,V)$ is connected.
\end{lm}

\begin{proof}
Take $g_0^\sharp, g_1^\sharp \in {\mathcal P}_b^\sharp(M,V)$. Since $\mc{CH}^\sharp(N, V)$ is connected, there is a path $g_t^\sharp$ connecting $g_0^\sharp$ and $g_1^\sharp$ in $\mc{CH}^\sharp(N, V)$. Represent the projection of $g_t^\sharp$ to $\mc{CH}(N, V)$ as $(\ol g_t, f_t)$, and denote $\clconv(f_t(V))$ by $M(g_t)$. It is clear that since $f_t(V)$ does not belong to $C(\ol g_t)$, the set $M(g_t)$ is homeomorphic to $M$.

For every $v \in V$ define
$$O_v:=\{t \in [0,1]: f_t(v) \in \clconv(f_t(V\backslash\{v\}))\}.$$
This is a closed set. Now for a given $v \in V$ we can continuously increase its $\R_+$-coordinate from trivialization~(\ref{tr3}) over a small neighborhood of $O_v$ in $[0,1]$ so that for no other $w \in V$ the set $O_w$ enlarges. After doing this for every vertex, we obtain that $f_t(V)$ is in a strictly convex position for every $t \in [0,1]$. This allows to modify $g_t^\sharp$ to a path in $\mc P_b^\sharp(M, V)$.

\end{proof}

\subsection{From local to global rigidity}

In this section we prove Proposition~\ref{prop1} from the introduction stating that the local rigidity of a bent metric implies global rigidity. 

\begin{proof}
Define $\mathcal P_{bp}(M, V):=\mathcal I^{-1}_V(\mathcal D_{bc}(\partial M, V))$. As $\mathcal D_{bc}(\partial M, V)$ is open in $\mathcal D_{c}(\partial M, V)$ and $\mathcal I_V$ is continuous, $\mathcal P_{bp}(M, V)$ is open in $\mathcal P_b(M, V)$. Due to Corollary~\ref{balreal}, we get $\mathcal P_{bp}(M, V) \subset \mathcal P_{cp}(M, V)$. Recall that Corollary~\ref{locrig} states that $\mc I_V$ is a local homeomorphism over $\mc P_{cp}(M, V)$, hence it is a local homeomorphism over $\mc P_{bp}(M, V)$. Lemma~\ref{proper} together with the definition of $\mathcal P_{bp}(M, V)$ imply that $\mathcal I_V$ is proper over $\mathcal P_{bp}(M, V)$. Hence, $\mathcal I_V$ is a covering map over $\mathcal P_{bp}(M, V)$. Take into account, however, that $\mathcal P_{bp}(M, V)$ and $\mc D_{bc}(\pt M, V)$ may be disconnected.

Suppose that there are $g_0, g_1 \in \mc P_b(M, V)$ such that $\mc I_V$ is a local homeomorphism around $g_0$ and $$\mathcal I_V(g_0)=\mathcal I_V(g_1) \in \mathcal D_{c}(\partial M, V).$$
We have a commutative diagram
\begin{center}
\begin{tikzcd}
\mc P^\sharp_{b} (M, V) \arrow[r, "\mc I^\sharp_V"] \arrow[d]
& \mc D^\sharp_{c}(\pt M, V) \arrow[d] \\
\mc P_{b}(M, V) \arrow[r, "\mc I_V"]
& \mc D_{c}(\pt M, V)
\end{tikzcd}
\end{center}
The vertical maps are covering maps. Moreover, their groups of deck transformations are naturally isomorphic to $B_0(\pt M, V)$ and $\mc I_V^\sharp$ is equivariant with respect to it, see Remark~\ref{deck}. It follows that there are lifts $g_0^\sharp$ and $g_1^\sharp$ of $g_0$ and $g_1$ such that
$$\mathcal I_V^\sharp(g_0^\sharp)=\mathcal I_V^\sharp(g_1^\sharp) \in \mathcal D_{c}^\sharp(\partial M, V).$$
Lemma~\ref{connect} implies that there exists a path $\zeta^\sharp: [0,1] \rightarrow {\mathcal P}_{b}^\sharp(M, V)$ connecting $g_0^\sharp$ and $g_1^\sharp$. Consider $[0,1]$ with the identified endpoints as $S^1$. By composing ${\mathcal I}_V^\sharp$ with $\zeta^\sharp$, we obtain a map $\beta^\sharp: S^1 \rightarrow {\mathcal D}_{c}^\sharp(\partial M, V)$. Lemma~\ref{connect2} implies that there exist $W \supseteq V$ and a lift $\alpha^\sharp: S^1 \rightarrow \ol{\mc D}_{bc}^\sharp(\pt M, W)$ of $\beta^\sharp$ such that the loop $\alpha^\sharp$ is contractible in $\overline {\mathcal D}_{bc}^\sharp(\partial M, W)$. Denote the projections of $\zeta^\sharp$ and $\alpha^\sharp$ to $\mc P_b(M, V)$ and $\ol{\mc D}_{bc}(\pt M, W)$ by $\zeta$ and $\alpha$ respectively. It follows that $\alpha$ is contractible in $\overline {\mathcal D}_{bc}(\partial M, W)$.

The path $\zeta: [0,1] \rightarrow {\mathcal P}_{b}(M, V)$ uniquely lifts to a path $\eta: [0,1] \rightarrow \overline{\mathcal P}_{b}(M, W)$ such that $\alpha={\mc I}_W \circ \eta$. We now construct a homotopy $\eta_s: [0,1] \rightarrow \ol{\mc P}_b(M, W)$ of the path $\eta=\eta_0$ such that $s \in [0, \e]$ for some small $\e>0$ and for all $s>0$ and $t \in [0,1]$ we have $\eta_s(t) \in \mathcal P_{bp}(M, W)$. To this purpose, for each $w\in W$ and each $t \in [0,1]$ choose an outer normal ray at $w$ in $N(\overline \eta(t))$ continuously depending on~$t$. Move all points of $W$ along these rays at distance $s$ and take the closed convex hull. It is not hard to see that for every $s$ the obtained points are in a strictly convex position in $N(\overline \eta(t))$. The closed convex hull is isotopic to $M(\eta(t))$ in $N(\overline \eta(t))$ by an isotopy sending the vertices to the respective points of $W$. This determines a class in ${\mc P}_b(M, W)$, which we denote by $\eta_s(t)$. We also denote ${\mc I}_W(\eta_s(t))$ by $\theta_s(t) \in \mc D_c(\pt M, W)$. For each $t\in [0,1]$ we see that $\theta_s(t)$ converge in $\ol{\mathcal D}_c(\partial M, W)$ to $\alpha(t) \in \ol{\mathcal D}_{bc}(\partial M, W)$ as $s \rightarrow 0$. Then, as $\overline{\mathcal D}_{bc}(\partial M, W)$ is open in $\overline{\mathcal D}_{c}(\partial M, W)$, provided that $\e$ is small enough, for all $s\in(0, \e]$ we get $\theta_s(t) \in \mathcal D_{bc}(\partial M, W)$ and so $\eta_s(t) \in \mc P_{bp}(M, W)$.

We claim that for each $s \in (0, \e]$ the metrics $\theta_s(0)$ and $\theta_s(1)$ can be connected by arcs $\theta'_s \subset \mc D_{bc}(\partial M, W)$ such that after a reparametrization of curves $\theta_s \cup \theta'_s$ as loops $\alpha'_s: S^1 \rightarrow \mc D_{bc}(\partial M, W)$ they constitute a homotopy of the loop $\alpha$. First, we consider the arcs $$\theta''_s=\{\theta_r(0): r \in [0,s]\} \cup \{\theta_r(1): r \in [0,s]\}\subset \ol{\mc D}_{bc}(\partial M, W).$$ 
We reparametrize $\theta_s \cup \theta''_s$ as loops $\alpha''_s: S^1 \rightarrow \ol{\mc D}_{bc}(\partial M, W)$ so that $\alpha''_s$ constitute a homotopy of $\alpha$, $\alpha''_s(0)=\alpha''_s(1)=\alpha(0)$, and for all $s\in (0, \e]$ and $t \in (0, 1)$ we have $\alpha''_s(t) \in \mc D_{bc}(\partial M, W)$. Moreover, we reparametrize them so that the arc $\theta_s$ corresponds to the restriction of $\alpha''_s$ to the interval $[s, 1-s]$ (we can assume that $\e<1/2$). 
Since the loop $\alpha$ is contractible in $\ol{\mc D}_{bc}(\pt M, W)$, so are the loops $\alpha''_s$, and the homotopy $\alpha''_s$ of $\alpha$ can be lifted to $\ol{\mc D}^\sharp_{bc}(\partial M, W)$. Then Lemma~\ref{homotop} can be used to modify it to have the image in $\mc D_{bc}^\sharp(\partial M, W)$ for $s>0$. Denote the projection of the modified homotopy back to $\ol{\mc D}_{bc}(\pt M, W)$ by $\alpha'_s$. Then the loops $\alpha'_s$ constitute the homotopy of $\alpha$ and for $s \in (0, \e]$ we have $\alpha'_s: S^1 \rightarrow  \mc D_{bc}(\partial M, W)$.

We claim that if a loop $\alpha': S^1 \rightarrow \mathcal D_{bc}(\partial M, W)$ is contractible in $\overline{\mathcal D}_{bc}(\partial M, W)$, then it is contractible in $\mathcal D_{bc}(\partial M, W)$.
Indeed, we take a map $\xi': \overline D^2 \rightarrow \overline{\mathcal D}_{bc}(\partial M, W)$ contracting $\alpha'$, lift it to $\overline{\mathcal D}_{bc}^\sharp(\partial M, W)$, modify it there with the help of Lemma~\ref{homotop}, and project it back to a map $\xi'':\overline D^2 \rightarrow \mathcal D_{bc}(\partial M, W)$.
Thereby, the loops $\alpha'_s$, constructed above, are contractible in $\mc D_{bc}(\partial M, W)$ and hence $(\mathcal I_W)^{-1}(\alpha'_s)$ consist of closed loops. Particularly, for each $s \in (0, \e]$ the preimage $(\mathcal I_W)^{-1}(\theta'_s)$ contains an arc $\eta'_s$ connecting $\eta_s(0)$ and $\eta_s(1)$. 


Let $U$ be a compact neighborhood of $\alpha(0)$ in $\overline{\mathcal D}_c(\partial M, W)$ containing all $\theta'_s$. By Lemma~\ref{proper}, ${\mathcal I}_W$ is proper. Hence, $X:={\mathcal I}^{-1}_W(U)$ is compact too. All arcs $\eta'_s$ belong to $X$ and we can consider their Hausdorff limit $\eta' \subset X$ (note that $\mc P_b(M, V)$ is a finite-dimensional manifold and the Hausdorff limits are well-defined; we consider $\eta'$ as a set, not as a map). Since the arcs $\theta'_s$ converge to $\alpha(0)$, we have ${\mc I}_W(\eta')=\alpha(0)$. Also we have $\eta(0), \eta(1) \in \eta'$. Since $\mc I_V$ is a local homeomorphism around $g_0$, Proposition~\ref{prop6} implies that $\eta(0)$ is isolated in ${\mathcal I}^{-1}_W(\alpha(0))$.
Therefore, $\eta(0)$ and the rest of ${\mc I}_W^{-1}(\alpha(0))$ have disjoint open neighborhoods $Y_1$ and $Y_2$. For all small enough $s$ the arcs $\eta'_s$ should belong to $Y_1\cup Y_2$ and meet both $Y_1$ and $Y_2$, which is impossible since $\eta'_s$ are connected arcs. This contradiction finishes the proof.
\end{proof}

\bibliographystyle{abbrv}
\bibliography{polyhedral}

\begin{thebibliography}{10}

\bibitem{Ale}
A.~D. Alexandrov.
\newblock Existence of a convex polyhedron and of a convex surface with a given
  metric.
\newblock {\em Mat. Sb. (N.S.)}, 11(53):15--65, 1942.

\bibitem{Ale2}
A.~D. Alexandrov.
\newblock Complete convex surfaces in {L}obachevskian space.
\newblock {\em Bull. Acad. Sci. URSS. S\'{e}r. Math. [Izvestia Akad. Nauk
  SSSR]}, 9:113--120, 1945.

\bibitem{Ale4}
A.~D. Alexandrov.
\newblock {\em Convex polyhedra}.
\newblock Springer-Verlag, Berlin, 2005.

\bibitem{Ale3}
A.~D. Alexandrov.
\newblock {\em Intrinsic geometry of convex surfaces}.
\newblock Chapman \& Hall/CRC, Boca Raton, FL, 2006.

\bibitem{AC}
J.~W. Anderson and R.~D. Canary.
\newblock Cores of hyperbolic {$3$}-manifolds and limits of {K}leinian groups.
\newblock {\em Amer. J. Math.}, 118(4):745--779, 1996.

\bibitem{And}
E.~M. Andreev.
\newblock Convex polyhedra of finite volume in {L}obacevskii space.
\newblock {\em Mat. Sb. (N.S.)}, 83 (125):256--260, 1970.

\bibitem{BM}
A.~F. Beardon and B.~Maskit.
\newblock Limit points of {K}leinian groups and finite sided fundamental
  polyhedra.
\newblock {\em Acta Math.}, 132:1--12, 1974.

\bibitem{BP}
R.~Benedetti and C.~Petronio.
\newblock {\em Lectures on hyperbolic geometry}.
\newblock Universitext. Springer-Verlag, Berlin, 1992.

\bibitem{Ber}
L.~Bers.
\newblock Spaces of {K}leinian groups.
\newblock In {\em Several {C}omplex {V}ariables, {I} ({P}roc. {C}onf., {U}niv.
  of {M}aryland, {C}ollege {P}ark, {M}d., 1970)}, pages 9--34. Springer,
  Berlin, 1970.

\bibitem{BI}
A.~I. Bobenko and I.~Izmestiev.
\newblock Alexandrov's theorem, weighted {D}elaunay triangulations, and mixed
  volumes.
\newblock {\em Ann. Inst. Fourier (Grenoble)}, 58(2):447--505, 2008.

\bibitem{BLP}
M.~Boileau, B.~Leeb, and J.~Porti.
\newblock Geometrization of 3-dimensional orbifolds.
\newblock {\em Ann. of Math. (2)}, 162(1):195--290, 2005.

\bibitem{Bon2}
F.~Bonahon.
\newblock {\em Closed curves on surfaces}.
\newblock preprint of a book.

\bibitem{Bon4}
F.~Bonahon.
\newblock Shearing hyperbolic surfaces, bending pleated surfaces and
  {T}hurston's symplectic form.
\newblock {\em Ann. Fac. Sci. Toulouse Math. (6)}, 5(2):233--297, 1996.

\bibitem{Bon}
F.~Bonahon.
\newblock Variations of the boundary geometry of {$3$}-dimensional hyperbolic
  convex cores.
\newblock {\em J. Differential Geom.}, 50(1):1--24, 1998.

\bibitem{Bon3}
F.~Bonahon.
\newblock Geodesic laminations on surfaces.
\newblock In {\em Laminations and foliations in dynamics, geometry and topology
  ({S}tony {B}rook, {NY}, 1998)}, pages 1--37. Amer. Math. Soc., Providence,
  RI, 2001.

\bibitem{Bon5}
F.~Bonahon.
\newblock Kleinian groups which are almost {F}uchsian.
\newblock {\em J. Reine Angew. Math.}, 587:1--15, 2005.

\bibitem{BO}
F.~Bonahon and J.-P. Otal.
\newblock Laminations measur\'{e}es de plissage des vari\'{e}t\'{e}s
  hyperboliques de dimension 3.
\newblock {\em Ann. of Math. (2)}, 160(3):1013--1055, 2004.

\bibitem{BS}
F.~Bonsante and J.-M. Schlenker.
\newblock Ad{S} manifolds with particles and earthquakes on singular surfaces.
\newblock {\em Geom. Funct. Anal.}, 19(1):41--82, 2009.

\bibitem{BC}
M.~Bridgeman and R.~D. Canary.
\newblock From the boundary of the convex core to the conformal boundary.
\newblock {\em Geom. Dedicata}, 96:211--240, 2003.

\bibitem{BCM}
J.~F. Brock, R.~D. Canary, and Y.~N. Minsky.
\newblock The classification of {K}leinian surface groups, {II}: {T}he ending
  lamination conjecture.
\newblock {\em Ann. of Math. (2)}, 176(1):1--149, 2012.

\bibitem{Bro}
K.~Bromberg.
\newblock Hyperbolic cone-manifolds, short geodesics, and {S}chwarzian
  derivatives.
\newblock {\em J. Amer. Math. Soc.}, 17(4):783--826, 2004.

\bibitem{BBI}
D.~Burago, Y.~Burago, and S.~Ivanov.
\newblock {\em A course in metric geometry}.
\newblock American Mathematical Society, Providence, RI, 2001.

\bibitem{Bus}
P.~Buser.
\newblock {\em Geometry and spectra of compact {R}iemann surfaces}.
\newblock Birkh\"{a}user Boston, Inc., Boston, MA, 1992.

\bibitem{CEG}
R.~D. Canary, D.~B.~A. Epstein, and P.~L. Green.
\newblock Notes on notes of {T}hurston.
\newblock In {\em Fundamentals of hyperbolic geometry: selected expositions},
  pages 1--115. Cambridge Univ. Press, Cambridge, 2006.

\bibitem{CB}
A.~J. Casson and S.~A. Bleiler.
\newblock {\em Automorphisms of surfaces after {N}ielsen and {T}hurston}.
\newblock Cambridge University Press, Cambridge, 1988.

\bibitem{CH}
Q.~Chen and J.-M. Schlenker.
\newblock The geometric data on the boundary of convex subsets of hyperbolic
  manifolds, 2022.
\newblock ArXiv e-print 2210.11782.

\bibitem{CoV}
S.~{Cohn-Vossen}.
\newblock {Zwei S\"atze \"uber die Starrheit der Eifl\"achen.}
\newblock {\em {Nachr. Ges. Wiss. G\"ottingen, Math.-Phys. Kl.}},
  1927:125--134, 1927.

\bibitem{CHK}
D.~Cooper, C.~D. Hodgson, and S.~P. Kerckhoff.
\newblock {\em Three-dimensional orbifolds and cone-manifolds}.
\newblock Mathematical Society of Japan, Tokyo, 2000.

\bibitem{Cul}
M.~Culler.
\newblock Lifting representations to covering groups.
\newblock {\em Adv. in Math.}, 59(1):64--70, 1986.

\bibitem{DeB}
J.~DeBlois.
\newblock The {D}elaunay tessellation in hyperbolic space.
\newblock {\em Math. Proc. Cambridge Philos. Soc.}, 164(1):15--46, 2018.

\bibitem{DS}
B.~Dular and J.-M. Schlenker.
\newblock Convex co-compact hyperbolic manifolds are determined by their
  pleating lamination, 2024.
\newblock ArXiv e-print 2403.10090.

\bibitem{EM}
D.~B.~A. Epstein and A.~Marden.
\newblock Convex hulls in hyperbolic space, a theorem of {S}ullivan, and
  measured pleated surfaces.
\newblock In {\em Fundamentals of hyperbolic geometry: selected expositions},
  pages 117--266. Cambridge Univ. Press, Cambridge, 2006.

\bibitem{EP}
D.~B.~A. Epstein and R.~C. Penner.
\newblock Euclidean decompositions of noncompact hyperbolic manifolds.
\newblock {\em J. Differential Geom.}, 27(1):67--80, 1988.

\bibitem{Fil}
F.~Fillastre.
\newblock Polyhedral realisation of hyperbolic metrics with conical
  singularities on compact surfaces.
\newblock {\em Ann. Inst. Fourier (Grenoble)}, 57(1):163--195, 2007.

\bibitem{Fil3}
F.~Fillastre.
\newblock Polyhedral hyperbolic metrics on surfaces.
\newblock {\em Geom. Dedicata}, 134:177--196, 2008.

\bibitem{FI}
F.~Fillastre and I.~Izmestiev.
\newblock Hyperbolic cusps with convex polyhedral boundary.
\newblock {\em Geom. Topol.}, 13(1):457--492, 2009.

\bibitem{FI2}
F.~Fillastre and I.~Izmestiev.
\newblock Gauss images of hyperbolic cusps with convex polyhedral boundary.
\newblock {\em Trans. Amer. Math. Soc.}, 363(10):5481--5536, 2011.

\bibitem{FIV}
F.~Fillastre, I.~Izmestiev, and G.~Veronelli.
\newblock Hyperbolization of cusps with convex boundary.
\newblock {\em Manuscripta Math.}, 150(3-4):475--492, 2016.

\bibitem{FP}
F.~Fillastre and R.~Prosanov.
\newblock Polyhedral surfaces in flat (2+1)-spacetimes and balanced
  cellulations on hyperbolic surfaces, 2023.
\newblock ArXiv e-print 2312.14266.

\bibitem{FS2}
F.~Fillastre and D.~Slutskiy.
\newblock Embeddings of non-positively curved compact surfaces in flat
  {L}orentzian manifolds.
\newblock {\em Math. Z.}, 291(1-2):149--178, 2019.

\bibitem{GMT}
D.~Gabai, G.~R. Meyerhoff, and N.~Thurston.
\newblock Homotopy hyperbolic 3-manifolds are hyperbolic.
\newblock {\em Ann. of Math. (2)}, 157(2):335--431, 2003.

\bibitem{Gol2}
W.~M. Goldman.
\newblock The symplectic nature of fundamental groups of surfaces.
\newblock {\em Adv. in Math.}, 54(2):200--225, 1984.

\bibitem{GGLSW}
X.~Gu, R.~Guo, F.~Luo, J.~Sun, and T.~Wu.
\newblock A discrete uniformization theorem for polyhedral surfaces {II}.
\newblock {\em J. Differential Geom.}, 109(3):431--466, 2018.

\bibitem{Her}
G.~{Herglotz}.
\newblock {\"Uber die Starrheit der Eifl\"achen.}
\newblock {\em {Abh. Math. Semin. Univ. Hamb.}}, 15:127--129, 1943.

\bibitem{HK}
C.~D. Hodgson and S.~P. Kerckhoff.
\newblock Rigidity of hyperbolic cone-manifolds and hyperbolic {D}ehn surgery.
\newblock {\em J. Differential Geom.}, 48(1):1--59, 1998.

\bibitem{HK2}
C.~D. Hodgson and S.~P. Kerckhoff.
\newblock Universal bounds for hyperbolic {D}ehn surgery.
\newblock {\em Ann. of Math. (2)}, 162(1):367--421, 2005.

\bibitem{KT}
Y.~Kamishima and S.~P. Tan.
\newblock Deformation spaces on geometric structures.
\newblock In {\em Aspects of low-dimensional manifolds}, pages 263--299.
  Kinokuniya, Tokyo, 1992.

\bibitem{Kap}
M.~Kapovich.
\newblock {\em Hyperbolic manifolds and discrete groups}.
\newblock Birkh\"{a}user Boston, Inc., Boston, MA, 2001.

\bibitem{Kub}
T.~Kubota.
\newblock On the extended {P}tolomy's theorem in hyperbolic geometry.
\newblock {\em Sci. Rep. T{\^o}hoku Univ., I. Ser.}, 1:131--156, 1913.

\bibitem{Lab2}
H.~Labeni.
\newblock Realizing metrics of curvature $\geq -1$ on closed surfaces in
  fuchsian anti-de sitter manifolds.
\newblock {\em J. of the Australian Math. Soc., to appear}, 2021.

\bibitem{Lab}
F.~Labourie.
\newblock M\'{e}triques prescrites sur le bord des vari\'{e}t\'{e}s
  hyperboliques de dimension {$3$}.
\newblock {\em J. Differential Geom.}, 35(3):609--626, 1992.

\bibitem{Lec}
C.~Lecuire.
\newblock Plissage des vari\'{e}t\'{e}s hyperboliques de dimension 3.
\newblock {\em Invent. Math.}, 164(1):85--141, 2006.

\bibitem{Lew}
H.~{Lewy}.
\newblock {On the existence of a closed convex surface realizing a given
  Riemannian metric.}
\newblock {\em {Proc. Natl. Acad. Sci. USA}}, 24:104--106, 1938.

\bibitem{LST}
F.~Luo, S.~Schleimer, and S.~Tillmann.
\newblock Geodesic ideal triangulations exist virtually.
\newblock {\em Proc. Amer. Math. Soc.}, 136(7):2625--2630, 2008.

\bibitem{LS}
R.~C. Lyndon and P.~E. Schupp.
\newblock {\em Combinatorial group theory}.
\newblock Springer-Verlag, Berlin, 2001.

\bibitem{Mar}
A.~Marden.
\newblock The geometry of finitely generated {K}leinian groups.
\newblock {\em Ann. of Math. (2)}, 99:383--462, 1974.

\bibitem{Mar2}
A.~Marden.
\newblock {\em Hyperbolic manifolds. An introduction in 2 and 3 dimensions}.
\newblock Cambridge University Press, Cambridge, 2016.

\bibitem{Mar3}
B.~Martelli.
\newblock An introduction to geometric topology, 2016.
\newblock ArXiv e-print 1610.02592.

\bibitem{MS}
H.~Masur and J.~Smillie.
\newblock Hausdorff dimension of sets of nonergodic measured foliations.
\newblock {\em Ann. of Math. (2)}, 134(3):455--543, 1991.

\bibitem{MM3}
R.~Mazzeo and G.~Montcouquiol.
\newblock Infinitesimal rigidity of cone-manifolds and the {S}toker problem for
  hyperbolic and {E}uclidean polyhedra.
\newblock {\em J. Differential Geom.}, 87(3):525--576, 2011.

\bibitem{MM2}
D.~McCullough and A.~Miller.
\newblock Homeomorphisms of {$3$}-manifolds with compressible boundary.
\newblock {\em Mem. Amer. Math. Soc.}, 61(344):xii+100, 1986.

\bibitem{Mil}
A.~D. Milka.
\newblock The lemma of {B}usemann and {F}eller in spherical and hyperbolic
  spaces.
\newblock {\em Ukrain. Geometr. Sb.}, (10):40--49, 1971.

\bibitem{Mil2}
A.~D. Milka.
\newblock Unique determinacy of general closed convex surfaces in {L}obacevskii
  space.
\newblock {\em Ukrain. Geom. Sb.}, (23):99--107, iii, 1980.

\bibitem{Min}
Y.~Minsky.
\newblock The classification of {K}leinian surface groups. {I}. {M}odels and
  bounds.
\newblock {\em Ann. of Math. (2)}, 171(1):1--107, 2010.

\bibitem{Mon}
G.~Montcouquiol.
\newblock Deformations of hyperbolic convex polyhedra and cone-3-manifolds.
\newblock {\em Geom. Dedicata}, 166:163--183, 2013.

\bibitem{Mos}
G.~D. Mostow.
\newblock Quasi-conformal mappings in {$n$}-space and the rigidity of
  hyperbolic space forms.
\newblock {\em Inst. Hautes \'{E}tudes Sci. Publ. Math.}, (34):53--104, 1968.

\bibitem{Nic}
P.~J. Nicholls.
\newblock {\em The ergodic theory of discrete groups}.
\newblock Cambridge University Press, Cambridge, 1989.

\bibitem{Nir}
L.~{Nirenberg}.
\newblock {The Weyl and Minkowski problems in differential geometry in the
  large.}
\newblock {\em {Commun. Pure Appl. Math.}}, 6:337--394, 1953.

\bibitem{Olo}
S.~Olovyanishnikov.
\newblock G\'{e}n\'{e}ralisation du th\'{e}or\`eme de {C}auchy sur les
  poly\`edres convexes.
\newblock {\em Rec. Math. [Mat. Sbornik] N.S.}, 18(60):441--446, 1946.

\bibitem{PH}
R.~C. Penner and J.~L. Harer.
\newblock {\em Combinatorics of train tracks}.
\newblock Princeton University Press, Princeton, NJ, 1992.

\bibitem{Pog}
A.~V. Pogorelov.
\newblock {\em Extrinsic geometry of convex surfaces}.
\newblock American Mathematical Society, Providence, R.I., 1973.

\bibitem{Pro2}
R.~Prosanov.
\newblock Rigidity of compact convex {F}uchsian manifolds.
\newblock {\em Int. Math. Res. Not.}
\newblock to appear.

\bibitem{Pro}
R.~Prosanov.
\newblock Ideal polyhedral surfaces in {F}uchsian manifolds.
\newblock {\em Geometria Dedicata}, 206(1):151--179, 2020.

\bibitem{Pro3}
R.~Prosanov.
\newblock Dual metrics on the boundary of strictly polyhedral hyperbolic
  3-manifolds, 2022.
\newblock ArXiv e-print 2203.16971.

\bibitem{Rag}
M.~S. Raghunathan.
\newblock {\em Discrete subgroups of {L}ie groups}.
\newblock Springer-Verlag, New York-Heidelberg, 1972.

\bibitem{Ray}
J.~S. Raymond.
\newblock Local inversion for differentiable functions and the {D}arboux
  property.
\newblock {\em Mathematika}, 49(1-2):141--158, 2002.

\bibitem{Riv}
I.~Rivin.
\newblock A characterization of ideal polyhedra in hyperbolic {$3$}-space.
\newblock {\em Ann. of Math. (2)}, 143(1):51--70, 1996.

\bibitem{HR}
I.~Rivin and C.~D. Hodgson.
\newblock A characterization of compact convex polyhedra in hyperbolic
  {$3$}-space.
\newblock {\em Invent. Math.}, 111(1):77--111, 1993.

\bibitem{Sch2}
J.-M. Schlenker.
\newblock Surfaces convexes dans des espaces lorentziens \`a courbure
  constante.
\newblock {\em Comm. Anal. Geom.}, 4(1-2):285--331, 1996.

\bibitem{Sch3}
J.-M. Schlenker.
\newblock Hyperbolic manifolds with polyhedral boundary, 2001.
\newblock ArXiv e-print 0111136.

\bibitem{Sch}
J.-M. Schlenker.
\newblock Hyperbolic manifolds with convex boundary.
\newblock {\em Invent. Math.}, 163(1):109--169, 2006.

\bibitem{Schn}
R.~Schneider.
\newblock {\em Convex bodies: the {B}runn-{M}inkowski theory}.
\newblock Cambridge University Press, Cambridge, 2014.

\bibitem{Ser}
C.~Series.
\newblock Thurston's bending measure conjecture for once punctured torus
  groups.
\newblock In {\em Spaces of {K}leinian groups}, pages 75--89. Cambridge Univ.
  Press, Cambridge, 2006.

\bibitem{Slu}
D.~Slutskiy.
\newblock Compact domains with prescribed convex boundary metrics in
  quasi-{F}uchsian manifolds.
\newblock {\em Bull. Soc. Math. France}, 146(2):309--353, 2018.

\bibitem{Smi}
G.~Smith.
\newblock A short proof of an assertion of {T}hurston concerning convex hulls.
\newblock In {\em In the tradition of {T}hurston -- geometry and topology},
  pages 255--261. Springer, Cham, 2020.

\bibitem{Smi2}
G.~Smith.
\newblock {On the Weyl Problem in Minkowski Space}.
\newblock {\em International Mathematics Research Notices}, 2021.

\bibitem{Str}
S.~{Straszewicz}.
\newblock {\"Uber exponierte Punkte abgeschlossener Punktmengen.}
\newblock {\em {Fundam. Math.}}, 24:139--143, 1935.

\bibitem{Sul}
D.~Sullivan.
\newblock Travaux de {T}hurston sur les groupes quasi-fuchsiens et les
  vari\'{e}t\'{e}s hyperboliques de dimension {$3$} fibr\'{e}es sur {$S^{1}$}.
\newblock In {\em Bourbaki {S}eminar, {V}ol. 1979/80}, pages 196--214.
  Springer, Berlin-New York, 1981.

\bibitem{Tam}
A.~Tamburelli.
\newblock Prescribing metrics on the boundary of anti--de {S}itter 3-manifolds.
\newblock {\em Int. Math. Res. Not. IMRN}, (5):1281--1313, 2018.

\bibitem{Thu}
W.~P. Thurston.
\newblock {\em The geometry and topology of 3-manifold}.
\newblock Princeton University Press, Princeton, NJ, 1978.

\bibitem{Tro}
M.~Troyanov.
\newblock Prescribing curvature on compact surfaces with conical singularities.
\newblock {\em Trans. Amer. Math. Soc.}, 324(2):793--821, 1991.

\bibitem{Vol2}
Y.~A. Volkov.
\newblock An estimate of the deformation of a convex surface as a function of
  the change in its intrinsic metric.
\newblock {\em Ukrain. Geometr. Sb. Vyp.}, 5--6:44--69, 1968.

\bibitem{Vol}
Y.~A. Volkov.
\newblock Existence of a polyhedron with prescribed development.
\newblock {\em Zap. Nauchn. Sem. S.-Peterburg. Otdel. Mat. Inst. Steklov.
  (POMI)}, 476(13):50--78, 2018.

\bibitem{Wal}
F.~Waldhausen.
\newblock On irreducible {$3$}-manifolds which are sufficiently large.
\newblock {\em Ann. of Math. (2)}, 87:56--88, 1968.

\bibitem{Weil}
A.~Weil.
\newblock Remarks on the cohomology of groups.
\newblock {\em Ann. of Math. (2)}, 80:149--157, 1964.

\bibitem{Wei}
H.~Weiss.
\newblock Local rigidity of 3-dimensional cone-manifolds.
\newblock {\em J. Differential Geom.}, 71(3):437--506, 2005.

\bibitem{Wei2}
H.~Weiss.
\newblock The deformation theory of hyperbolic cone-3-manifolds with
  cone-angles less than {$2\pi$}.
\newblock {\em Geom. Topol.}, 17(1):329--367, 2013.

\end{thebibliography}
\bigskip{\footnotesize\par
  \textsc{University of Vienna, Faculty of Mathematics, Oskar-Morgenstern-Platz 1, A-1090 Vienna, Austria} \par
  \textit{E-mail}: \texttt{roman.prosanov@univie.ac.at}
}
\end{document}